\documentclass[aop,preprint]{imsart}
\RequirePackage[OT1]{fontenc}
\RequirePackage{amsthm,amsmath}
\RequirePackage[numbers]{natbib}
\RequirePackage[colorlinks,citecolor=blue,urlcolor=blue]{hyperref}
\arxiv{arXiv:math.PR/0000000}
\usepackage{latexsym}
\usepackage{amsmath}
\usepackage{amsthm}
\usepackage{amsfonts}

\startlocaldefs
\theoremstyle{plain}\theoremstyle{plain}
\newtheorem{theorem}{Theorem}
\newtheorem{definition}{Definition}

\newtheorem{corollary}{Corollary}
\newtheorem{lemma}{Lemma}
\newtheorem{remark}{Remark}

\newcommand{\field}[1]{\mathbb{#1}}
\newcommand{\C}{\field{C}}
\newcommand{\R}{\field{R}}
\newcommand{\N}{\field{N}}
\newcommand{\Z}{\field{Z}}
\numberwithin{equation}{section}
\numberwithin{lemma}{section}
\numberwithin{theorem}{section}
\numberwithin{corollary}{section}
\numberwithin{remark}{section}
\numberwithin{definition}{section}
\newcommand{\at}{\makeatletter@\makeatother}
\endlocaldefs

\begin{document}
\begin{frontmatter}
\title{The Characteristic Function of the Renormalized Intersection Local Time of the Planar Brownian Motion}
\runtitle{Characteristic Function}
%\thankstext{T1}{Footnote to the title with the ``thankstext'' command.}

\begin{aug}
\author{\fnms{Daniel} \snm{H\"of}
%\thanksref{t1,t2,m1}
\ead[label=e1]{dhoef@gmx.net}}

%,
%\author{\fnms{Second} \snm{Author}\thanksref{t3,m1,m2}\ead[label=e2]{second@somewhere.com}}
%\and
%\author{\fnms{Third} \snm{Author}\thanksref{t1,m2}
%\ead[label=e3]{third@somewhere.com}
%\ead[label=u1,url]{http://www.foo.com}}

%\thankstext{t1}{Some comment}
%\thankstext{t2}{First supporter of the project}
%\thankstext{t3}{Second supporter of the project}
%\runauthor{F. Author et al.}

\affiliation{Provide e.V.}

\address{
2. K\i{}s\i{}m Mah. Anadolu Cad. 16 D: 8\\
Bah\c ce\c sehir\\
TR-34538 \.Istanbul \\
Republic of Turkey\\
%Address of the First and Second authors\\
%Usually a few lines long\\
\printead{e1}\\
%\phantom{E-mail:\ }\printead*{e2}
}

%\address{Address of the Third author\\
%Usually a few lines long\\
%Usually a few lines long\\
%\printead{e3}\\
%\printead{u1}}
\end{aug}

\begin{abstract}
In this article we study the distribution of the number of points of a simple random walk, visited a given number of times (the $k$-multiple point range). In a previous article \citep{dhoef1} we had developed a graph theoretical approach which is now extended from the closed to the the non restricted walk. Based on a study of the analytic properties of the first hit determinant a general method to define and calculate the generating functions of the moments of the distribution as analytical functions and express them as an absolutely converging series of graph contributions is given. So a method to calculate the moments for large length in any dimension $d \geq 2$ is developed. As an application the centralized moments of the distribution in two dimensions are completely calculated for the closed and the non restricted simple random walk in leading order with all logarithmic corrections of any order. As is well known \citep{legall,hamana_ann} we therefore have also calculated the characteristic function of the distribution of the renormalized intersection local time of the planar Brownian motion. It turns out to be closely related to the planar $\Phi^4$ theory.  
\end{abstract}



\end{frontmatter}

\noindent
%  \keywords{Multiple Points, Random Walks, Intersection Local Time}
%\emph{AMS 2000 subject classifications.} Primary: 60J10, 60FJ65; secondary: 60E07, 60E10, 60B12, 11M26 \hfil \\

\noindent
%\emph{Key words and phrases.} Multiple point range of a random walk, Intersection Local Time, Range of a random walk, multiple points, Brownian motion \hfil \\
\section{Introduction}\label{S:intro}
The local geometry of a simple random walk on a step by step basis is so simple that even a child can understand it. Yet its global geometry is so complex, intriguing and mysterious that we have just begun to scratch the surface in understanding it. 
The range  (i.e. the number of points visited) and multiple point range (i.e. the number of points visited a given number of times) of random walks, which are an important part of this statistical geometry, as one might call it, have received growing attention in the last sixty years \citep{Tek,Pitt,Fla,Ham1,Ham2,Ham3} (for an overview see \citep{chenxia}). This is also due to its close relationship to the Brownian motion as the encapsulation of its universal properties. There is e.g. a close relationship between the leading behaviour (for large length) of the moments of the range of a random walk and those of the renormalized intersection local time of the Brownian motion in two dimensions established by Le Gall \citep{legall}. This has been extended to the multiple point range by Hamana \citep{hamana_ann}.\\
The intriguing relationship between the local interactions of molecules and the global properties of the matter it constitutes macroscopically is not just a philosophical analogy in the realm of physics. As has been shown by Symanzik \citep{symanzik} and worked on by Brydges, Spencer and Fr\"ohlich \citep{bry}, studying a large class of lattice models of Quantum field theory (QFT) describing this relationship means calculating the  characteristic function of weighted ranges of intersecting random walks. So the field of statistical geometry has a direct relevance not just in mathematics but also in physics. QFT is just one example among others \citep{chenxia}. \\  
In this paper we continue to discuss the joint distribution of the variables $N_{2k}(w)$,
the number 
of points of a closed simple random walk $w$ on $\Z^{d}$ visited by
$w$ exactly $k$ times (the $k$-multiple point range).  (As we will only deal with simple random walks we will sometimes drop the word simple in the sequel). We had completely calculated the leading behaviour for large length of all moments in the one dimensional case in \citep{dhoef1}. The distribution turned out to be a generalized geometric distribution. We followed a geometrical approach, relating the distribution of multiple points to the distribution of Eulerian multidigraphs and their adjacency matrices $F$ as geometrical archetypes of random walks. In this paper we will do the same for two dimensions and prove the following:
\begin{theorem}
Let $\beta_1$ be the renormalized intersection local time of a planar Brownian motion as defined e.g. in \citep{basschen}. Then the characteristic function of its distribution is given by
\begin{multline}\label{eq_charbr}
E\left(e^{it\beta_1}\right) = e^{(1-\gamma)\frac{it}{2\pi}} \Biggl(\frac{1}{\Gamma\left(2+ \frac{it}{2\pi}\right)} 
+
\\ 4 \cdot \sum_{r=2}^{\infty} \frac{(it)^r }{8^r\cdot r\,! \cdot \Gamma\left(r+2 + \frac{it}{2\pi}\right)}\cdot\\
 \left( \sum_{F \in \tilde{H}_r(2,\ldots,2)} \left((r+1)\cdot \mathcal{I}(F) + \frac{it}{2}\cdot (r-1) \cdot I(F) \right) \cdot \mathcal{U}(F) \right)
\Biggr)
\end{multline}
where $E(.)$ denotes the expectation value, $\Gamma$ is the Eulerian integral of the second kind, $\gamma$ the Euler--Mascheroni constant,
$\tilde{H}_r(h_1,\ldots,h_r)$ denotes the set of the $r\times r$ matrices $F$ with zero diagonal and nonnegative integer entries which are balanced:
\begin{equation}
\sum_{j=1}^r F_{i,j} = \sum_{j=1}^r F_{j,i} = h_i 
\end{equation}
and can be understood as adjacency matrices of Eulerian multidigraphs
\begin{equation}
cof(A-F) \neq 0
\end{equation}
where $A = diag(h_1,\ldots,h_r)$ and $cof$ the cofactor. $\mathcal{I}(F)$ is given by 
\begin{multline}\label{eq_guzI}
\mathcal{I}(F) := \int_{(\R^2)^r} \prod_{m=1}^r{d^2y_m} \cdot \delta(y_1) \cdot \\ \Biggl(\sum_{\substack{1\leq i \leq r \\ 1 \leq j \leq r \\ i\neq j}}  \frac{\partial}{\partial H_{i,j}} 
\left(
 \prod_{\substack{1 \leq k \leq r \\ 1 \leq l \leq r \\ l \neq k}} 
 H_{k,l}^{F_{k,l}}
\right) 
\Biggr|_{H_{k,l} = \frac{2}{\pi} K_0(\left|y_k - y_l \right|)}
\Biggr)
\end{multline}
$I(F)$ is given by 
\begin{equation}\label{eq_guzIf}
I(F) := \int_{(\R^2)^r} \prod_{m=1}^r{d^2y_m} \cdot \delta(y_1) \cdot \Biggl(  
 \prod_{\substack{1 \leq k \leq r \\ 1 \leq l \leq r \\ l \neq k}} 
 \left(\frac{2}{\pi} K_0(\left|y_k - y_l \right|) \right)^{F_{k,l}}
\Biggr)
\end{equation}
with the Dirac $\delta$ distribution and the modified Bessel function $K_0$ and 
\begin{equation}
 \mathcal{U}(F) := M(F) \cdot cof(A-F)
\end{equation}
\begin{equation}
M(F) := \prod_{\substack{1\leq i \leq r \\ 1 \leq j \leq r \\ i\neq j}}\frac{1}{  F_{i,j}\,!}
\end{equation}
\end{theorem}
The path to reach this result is long but straightforward: \\
In \citep{dhoef1} we had expressed the moments of the joint distribution of the multiple point range for the closed simple random walk as a formal power series of the form 
\begin{equation}
P_A(z) := \sum_{w \in W_{2N}} A(w)\cdot z^{length(w)}
\end{equation}
where the quantities $A(w)$ are finite linear combinations of the moments and we sum over the set $W_{2N}$  of closed walks $w$ whose length is smaller or equal to $2N$ and who start and end in $0\in \Z^d$.
We had then expressed  $P_A(z)$  as a formal sum of known functions $g_ {A,F}(z)$ over the infinite set of adjacency matrices $F$ of Eulerian multidigraphs in the form
\begin{equation}\label{eq_sumfeyn}
P_A(z) = \sum_{F} g_{A,F}(z) 
\end{equation}
as an exact relationship between Taylor coefficients, this way avoiding discussions of convergence of the sum over the adjacency matrices. 
To get the asymptotic behaviour of $A(w)$ for large $length(w)$ one has to
\begin{enumerate}
\item study the analytic properties of the functions $g_{A,F}(z)$ in the complex plane
 \item find a way to transform the formal sum over $F$ into an absolutely converging series, so results for the functions $g_{A,F}(z)$ can be related safely to those of $P_A(z)$
\item find the singularities at the boundary of the radius of convergence for $P_A(z)$  and 
\item calculate the asymptotic singular behaviour in a vicinity of any of the singularities.
\end{enumerate} 
This paper does all these steps for dimension $d = 2$ and all but the last one for any dimension $d >2$. \\ 
In the second section we first generalize the basic enumaration formulas of the number of structure preserving maps (archetypes) from closed to non-closed simple random walks and from products of the multiple point range to walks which hit fixed points a given number of times.
In the sequel of the paper we will till section nine deal with closed simple random walks only, as the formulas are much simpler and the corresponding formulas for the non-closed simple random walk can easily be derived from them.  
\\
In the third section we then start with a study of the analytic properties of the lowest order building blocks of $g_{A,F}(z)$, i.e. the functions $h(x,d,z)$ which are the generating functions of the number of walks between $0$ and $x \in \Z^d$. They are shown to be holomorphic functions on the open set  $D = \{z \in \C; \lvert \Re(z) \rvert < \frac{1}{2d} \}$ with nice properties, e.g. we show that $\left|1+h(0,d,z) \right| \geq 1/2$ for $z \in \bar U_{1/2d}(0) \cap D$. \\
The fourth section is dedicated to the so called first hit determinant 
\begin{equation}\label{eq_fhitd1}
\Delta_r(Y,z) = \det\left(1_{r\times r} + U_{i,j}(Y,z)\right)
\end{equation}
where 
\begin{equation}\label{eq_fhitd2}
U_{i,j}(Y,z) := \frac{h(Y_i - Y_j,d,z)}{1 + h(0,d,z)}
\end{equation}
is the generating function of the number of walks starting in $Y_i$ and ending in $Y_j$ and hitting $Y_j$ the first time when they reach it. This determinant is the key to transforming the series over matrices $F$ into an absolutely converging series in any dimension. We prove that for any $Y$ in the set $J_r := \{ Y:=(Y_1,\ldots,Y_{r}) \in (\Z^d)^r \mid Y_1 = 0 \wedge i \neq j \Longrightarrow Y_i \neq Y_j \}$  and $z$ in an open neighborhood of  $\bar U_{1/2d}(0) \cap D$  
\begin{equation} 
\left|\Delta_r(Y,z) \right| \geq  \frac{1}{2^r \cdot (1 + \left|h(0,d,z)\right|)^{(r-1)}}
\end{equation}
and so it is only the vicinity of the points $z = \pm \frac{1}{2d}$ which is related to the asymptotic behaviour of the moments. We also show for $d \geq 3$ that for any $Y \in J_r$ and any $f \in \N \setminus \{0\}$ there is an $\epsilon(f) >0$ independant of $Y$ such that for $z \in U_{\epsilon(f)}(\pm 1/2d) \cap D$ we have 
\begin{equation}\label{eq_deltat}
\left|1 - \Delta_r(Y,z)^f \right| < 1
\end{equation}
\\
In the fifth section we show \eqref{eq_deltat} for $d = 2$ . This case has to be dealt with seperately as it requires completely new concepts for its proof. \\
Based on the fourth and fifth section we can now in the sixth section use the key relation \eqref{eq_deltat} which allows the transformation of summations over $F$ to summations over absolutely converging geometric series of the form
\begin{equation}\label{eq_geodel}
\frac{1}{\Delta_r(Y,z)^f} = \sum_{\ell = 0}^{\infty} (1 - \Delta_r(Y,z)^f)^{\ell}
\end{equation}
where the terms related to $\ell > 0$ belong to finite sets of matrices $F$ and so we can define the moment functions $P_A(z)$ as analytic functions and relate their analytic and asymptotic behaviour to that of the functions $g_{A,F}(z)$.\\
In the seventh section this is now applied to the special case of the moments in $d=2$ where known asymptotic series allow the complete calculation of the $r^{th}$ moment of the distribution to any order $M$ of logarithmic corrections of the form 
\begin{equation}
\frac{length(w)^r}{\ln(length(w))^{M}}
\end{equation}
for the closed simple random walk. \\ 
In the eighth section we now use the subset of Eulerian multidigraphs without dams (vertices of degree $2$ or more precisely of indegree $1$ and outdegree $1$ ) and those with dams generated from them (by consecutively lifting the edges incident to dams and removing loops if this operation creates them) to calculate the leading behaviour and all logarithmic corrections of the centralized moments and the characteristic function of the joint distribution of the multiple point range of the closed simple random walk. We get a result similar to equation \eqref{eq_charbr} \\
In the ninth section we generalize these results to the non restricted simple random walk generally discussed in the literature to reach the corresponding formulas.  Then the characteristic function of the renormalized intersection local time of the Brownian motion, based on the work of  Le Gall \citep{legall} and Hamana \citep{hamana_ann} is derived. It of course reproduces the classical result for the leading behaviour of the second moment of the range by Jain and Pruitt \citep{jain}.  A short discussion of some known results from complementary approaches is given. \\
The tenth section then is dedicated to a discussion of the results and a view towards where one can go from here based on the above results.
\section{Generalizations of the concept of Archetypes}
In the sequel of this article we use the notation of \citep{dhoef1} as it is the basis for this article. The set $J_r := \{ Y:=(Y_1,\ldots,Y_{r}) \in L^r \mid Y_1 = 0 \wedge i \neq j \Longrightarrow Y_i \neq Y_j \}$ unfortunately was given with the wrong indices starting in $0$ and ending in $r-1$ in \citep{dhoef1}, this is corrected here. \\
\subsection{Definitions}
We start with generalizing the definitions of \citep[section 2.2]{dhoef1}:
\begin{definition}
For a walk $w$ and an integer $k \in \N \setminus\{0\}$ we define
the number $N_{k}(w)$, the $k/2$-multiple point range as
$N_{k}(w) := \sharp \{q \in L \mid mu(q,w) = k \} $. We also define the range
$ran(w) := \sum_{k=1}^{\infty } N_{k}(w) $.
Note that for walks which are not closed the start point and the endpoint are the only points which have an odd multiplicity (and the distribution of such points is generally not discussed in the literature of the multiple point range yet). 
\end{definition}
\begin{definition}
Let $ w=(p,s) $ be a walk which is not closed and of length $ n $. Let $ u = (q,t) $
be a walk of length $ m \leq n $ such that there
exists an index $ 0 \leq ind < n $ such that 
$ i = 1,\ldots,m \Rightarrow  t_{i} = s_{i + ind}$ and 
$ q = p_{ind} $. Then the pair $ (ind;u) $ is called a fixed subwalk of $ w $
\end{definition}
\begin{definition}
Let $(ind_1;u_1)$ and $(ind_2;u_2)$ be two fixed subwalks of a walk $w = (p,s)$, which is not closed,
with lengths $m_1$ and $m_2$ and $n$ and sequences $s^{(1)}$, $s^{(2)}$ and $s$ respectively. 
If $ind_1 + m_1  = ind_2$ then the two fixed subwalks can be concatenated to 
the fixed subwalk $(ind_1;u_3)$ with $u_3 = (p_{ind_1},s^{(3)})$ with 
\begin{equation}
s^{(3)}_i := 
\begin{cases}
s^{(1)}_i \Longleftarrow i=1,\ldots,m_1 \\
s^{(2)}_{i-m_1} \Longleftarrow i=m_1+1,\ldots,m_1+m_2
\end{cases}
\end{equation}
\end{definition}
\begin{definition}
Let $(ind_1;u_1)$ and $(ind_2;u_2)$ be two fixed subwalks of a walk $w = (p,s)$ which is not closed
with lengths $m_1$ and $m_2$ and $n$ and sequences $s^{(1)}$, $s^{(2)}$ and $s$ respectively. The fixed subwalks are said to overlap if there are numbers 
$1 \leq i_1 \leq  m_1, 1 \leq i_2 \leq m_2$ such that 
$(ind_1 + i_1) = (ind_2 + i_2)$
\end{definition}
\begin{definition}
Let $w=(x,s)$ be a walk of length $n$, which is not closed. A refinement  of  $w$ is 
a collection of fixed subwalks of $w$ which do not overlap and can be concatenated to yield a fixed subwalk of length $n$, i.e. loosely speaking whose concatenation
yields $w$.
\end{definition}
\begin{definition}\label{D:Arch}
Let $\Xi := (G_1,\ldots,G_k) $  be a family of $k$ semi Eulerian graphs \\
$ G_j = (V_j,E_j,J_j) $
with mutually 
disjoint edge sets. Let $\Sigma := (w_1,\ldots,w_k)$ be a family of walks and   
$P_j$ the set of points of $w_j$  with nonzero multiplicity and $S_j$ the set 
of fixed subwalks of $ w_j $ . A $2k$ tupel of maps
$ f = (f_{V_1},\ldots,f_{V_k};f_{E_1},\ldots,f_{E_k})  $  
is called \textbf{ archetype }
if and only if 
\begin{description}
\item[Structure] $ f_{E_i}:E_i \mapsto S_i  $ 
is injective and maps the different edges 
of $ E_i $  onto a refinement of $ w_i $ . $ f_{V_i}:V_i \mapsto P_i $  is injective 
and for $ v \in V_i \cap V_j \Rightarrow f_{V_i}(v) = f_{V_j}(v) $ .
\item[Orientation] $f_{V_i}$ maps the starting and ending vertices of edges as 
given by $J_i$ on 
the starting and ending points of the corresponding fixed subwalks.
\end{description}
\end{definition}
\subsection{Lemmata}
We now prove a generalisation of \citep[ Lemma 2.2.]{dhoef1}:
\begin{lemma} \label{le_close}
For $Y\in J_r$ and an r vector $k = (k_1,\ldots,k_r)$ of integers with $k_i > 0$ we define the following subset of the set $W_{2N} \times \Z^d$ where $W_{2N}$ is the set of all closed walks $w$ with $length(w) \leq 2N$ starting and ending in $0 \in \Z^d$.
\begin{multline}
W_{2N}(Y,k_1,\hdots,k_r) := \\ \left\{(w,y) \in W_{2N} \times \Z^d \wedge   mu(Y_i - y,w) = 2\cdot k_i \;\forall{i=1,\hdots,r}\right\}
\end{multline}
Then 
\begin {multline} \label{eq_mdet}
\sum_{(w,y) \in W_{2N}(Y,k_1,\hdots,k_r)} z^{length(w)} = 
 \\
  z \cdot \frac {\partial } {\partial z} 
\Biggl[ 
\left( \prod_{i=1}^r \hat K(x_i,k_{i},h(0,d,z)) \right)
cof(\hat A - \hat F) \\
\left( \prod_{j=1}^r \frac {1} {(1-x_j)^2} \right)
\frac {1} {det(W^{-1} + X)} 
\Bigr\rvert_{\forall a: x_a = 0; \quad \forall b,c: X_{b,c}= U_{b,c}(Y,z)}
\Biggr]_{z \diamond 2N}
\end {multline}
\end{lemma}
\begin{proof} Let $y \in \Z^d$ and $w$ be a closed walk which for $i=1,\ldots,r$ has points of multiplicities $2\cdot k_i$ in $Y_i -y$.  Let $\Theta =\Theta (m_1,\ldots,m_r)$
be the set of equivalence classes under isomorphy of 
all Eulerian graphs with exactly $r$  vertices with degrees 
$2m_1,\ldots,2m_r$. For $[G] \in \Theta$ define $Arch(G,w,Y,y)$ as the set of archetypes between $G = (V,E,J)$, $V = \{v_1,\ldots,v_r\}$ and $w$ which for any $v_i \in V$ fulfills $f_V(v_i) = Y_i - y$ . Then following the arguments in \citep[ Theorem 2.1.]{dhoef1} we have
\begin{equation}\label{eq_en1}
\sum_{[G] \in \Theta (m_1,\ldots,m_r)} \frac {\sharp Arch(G,w,Y,y)} {\sharp Aut(G)} =
\sum_{k_i \geq m_i}
\prod_{j=1}^{l} \binom {k_j} {m_j}
\end{equation}
independantly of $y$. 
Using \citep[Theorem 2.2.]{dhoef1} and the same arguments as in the proof of \citep[Theorem 2.3., Lemma 2.1.]{dhoef1} we reach equation \eqref{eq_mdet}. 
\end{proof}
\begin{remark} In the case of $r = 1$ the Lemma is true, in this case one has to define $cof(\hat A - \hat F) := 1$.
\end{remark}
\begin{corollary}\label{co_close}
\begin{multline}\label{eq_mome}
\sum_{(w,y) \in W_{2N}(Y,k_1,\ldots,k_r)} z^{length(w)} =  \\
 z \cdot \frac 
{\partial } {\partial z}
\Bigl[ 
\sum_{h_i = 0}^{\infty }
\sum_{F \in H_r(h_1,\ldots,h_r)} \\ 
cof(A-F) 
\left( \prod_{\substack{1 \leq a \leq r \\ 1 \leq b \leq r}} \frac {h(Y_a - Y_b,d,z)^{F_{a,b}}} 
{(1+ h(0,d,z))^{F_{a,b}}\cdot F_{a,b}\,!}
\right) \cdot  \\
\prod_{j=1}^r (-1)^{h_j} \cdot h_j\,!\cdot K(h_j,k_{j},h(0,d,z))
\Bigr]_{z \diamond 2N}
\end{multline}
For $r \geq 2$ for $F \in H_r(h_1,\ldots,h_r)$ with $h_i = 0 \Rightarrow cof(A-F) = 0$.
In the case of $r=1$ however the $1 \times 1$ matrix $F_p = 0$ belonging to the graph being just one point, has by definition $cof(A_p - F_p) := 1$ and the sum over $h_1$ and $F$ in this case just consists of $h_1 = 0$ and $F = F_p$ reproducing the first moment in \citep[Corollary 2.1.]{dhoef1}.
\end{corollary}
We also prove a generalisation of \citep[Theorem 2.2.]{dhoef1}
\begin{lemma}\label{le_fGr2}
Let $G=(V,E,J)$ be a Eulerian graph with the vertices $v_1,\ldots,v_r$
with degree $m_1,\ldots,m_r$. Let $Y \in J_r$ and $Y_{r+1} \in \Z^d$ be a point $Y_{r+1} \neq Y_i \quad \forall i = 1,\ldots,r$.
Let 
\begin{equation}\label{eq_fGr2}
\begin{array}{rccl}
\pi:& V & \longmapsto & \{Y_1,\ldots,Y_r\} \\
    & v_i & \longmapsto & Y_i 
\end{array}
\end{equation}
be the natural projection.\\ Let $g_{2n}(G,\pi,Y,Y_{r+1})$  denote the number of different archetypes 
$(f_V;f_E)$ of $G$ into closed walks of length $2n$ with $f_V = \pi$  which start at $Y_{r+1}$ and their generating function
\begin{equation}
N(z,G,\pi,Y,Y_{r+1}) := \sum_{n=0}^{\infty } g_{2n}(G,\pi,Y,Y_{r+1})\cdot z^{2n}
\end{equation}
then the following equation 
\begin{multline} \label{E:fGr}
N(z,G,\pi,Y,Y_{r+1}) = N_E(G)\cdot \\ \Biggl(\sum_{\tilde{e} \in E} h(\pi(J(\tilde{e})_1)-Y_{r+1},d,z) \cdot 
h(Y_{r+1}-\pi(J(\tilde{e})_2),d,z) \\
\left( \prod_{e \in E \setminus\{\tilde{e}\}} h(\pi(J(e)_1) - \pi(J(e)_2),d,z) \right)\Biggr)
\end{multline}
holds, where $N_E(G)$ is the number of Euler trails of $G$. 
\end{lemma}
\begin{proof} From the Markov feature the sum 
\begin{multline}
\sum_{\tilde{e} \in E} h(\pi(J(\tilde{e})_1)-Y_{r+1},d,z) \cdot 
h(Y_{r+1}-\pi(J(\tilde{e})_2),d,z)
\\
\left( \prod_{e \in E \setminus\{\tilde{e}\}} h(\pi(J(e)_1) - \pi(J(e)_2),d,z) \right)
\end{multline}
is the generating function for the number of collections of subwalks determining a walk $w$ starting in $Y_{r+1}$ and visiting the points $Y_1,\ldots,Y_r$ in the order induced by a given Euler trail on $G$ determining an Archetype $(f_V,f_E)$ in the form requested by the Lemma uniquely. As has been shown before in the proof of \citep[Theorem 2.2.]{dhoef1} an archetype uniquely induces a Euler trail on $G$ and so the above generalization of the Theorem is also true. 
\end{proof}  
As a preparation for the walks which are not closed we now prove the following 
\begin{lemma}\label{le_diff}
Let $Y \in J_{r}$  and $k$ as before and $Y_{r+1} \in \Z^d$ with $Y_{r+1} \neq Y_i \quad \forall i=1,\ldots,r$  denote another point. Let 
\begin{multline}
W_{2N,Y_{r+1}}(Y,k_1,\ldots,k_r) := \\ 
\left\{w \in W_{2N} \wedge  mu(Y_i - Y_{r+1},w) = 2\cdot k_i \;\forall{i=1,\hdots,r}\right\}
\end{multline}
If we define 
\begin{equation}
U_{i,j} := \frac{H_{i,j}}{\sqrt{(1+ \omega_i)\cdot(1+ \omega_j)}}
\end{equation}
and 
\begin{equation}
\mathcal{H}_{i,j}(Y,z) = h(Y_i - Y_j,d,z)
\end{equation}
Then 
\begin {multline} \label{eq_mdet2}
\sum_{w \in W_{2N,Y_{r+1}}(Y,k_1,\hdots,k_r)} z^{length(w)} = \\
\Biggl(\sum_{\substack{1 \leq \alpha \leq r \\ 1 \leq \beta \leq r \\ \alpha \neq \beta}} 
h(Y_{r+1} - Y_{\alpha},d,z)\cdot h(Y_{\beta} - Y_{r+1},d,z) \cdot 
\frac{\partial}{\partial H_{\alpha,\beta}}  + \\
\sum_{1 \leq \gamma \leq r} h(Y_{r+1} - Y_{\gamma},d,z)\cdot h(Y_{\gamma} - Y_{r+1},d,z) \cdot \frac{\partial}{\partial \omega_{\gamma}} \Biggr)
 \\ 
\Biggl[\Biggl[ 
\left( \prod_{i=1}^r \hat K(x_i,k_{i},\omega_i) \right)
cof(\hat A - \hat F) \\
\left( \prod_{j=1}^r \frac {1} {(1-x_j)^2} \right)
\frac {1} {det(W^{-1} + X)} 
\Bigr\rvert_{\substack{\forall a: x_a = 0 \\ \forall b,c: X_{b,c}= U_{b,c}}}
\Biggr]\Bigr \rvert_{\substack{\forall \gamma: \omega_{\gamma} = h(0,d,z) \\ \forall \alpha,\beta: H_{\alpha,\beta}= \mathcal{H}_{\alpha,\beta}(Y,z)}}
\Biggr]_{z \diamond 2N}
\end {multline}
\end{lemma}
\begin{proof} For a Graph $G = (V,E,J)$ with $l_i$ loops in vertex $v_i$ and $F_{i,j}$ edges $e$ with $J(e)_1 = v_i$ and $J(e)_2 = v_j$ the product 
\begin{multline}
\prod_{e \in E } h(\pi(J(e)_1) - \pi(J(e)_2),d,z)  = \\ 
\left(\prod_{1 \leq i \leq r} \omega_i^{l_i} \right) \cdot \left( \prod_{\substack{1\leq a \leq r \\ 1 \leq b \leq r \\ a \neq b}} H_{a,b}^{F_{a,b}} \right)  \Bigr \rvert_{\substack{\forall \gamma: \omega_{\gamma} = h(0,d,z) \\ \forall \alpha,\beta: H_{\alpha,\beta}=\mathcal{H}_{\alpha,\beta}(Y,z)}}
\end{multline}
Therefore the product in equation \eqref{eq_fGr2}
\begin{multline}\label{eq_opr}
\sum_{\tilde{e} \in E} h(\pi(J(\tilde{e})_1)-Y_{r+1},d,z) \cdot 
h(Y_{r+1}-\pi(J(\tilde{e})_2),d,z)
\\
\left( \prod_{e \in E \setminus\{\tilde{e}\}} h(\pi(J(e)_1) - \pi(J(e)_2),d,z) \right)
\end{multline} can be rewritten in the form
\begin{multline}
\Biggl(\sum_{\substack{1 \leq \alpha \leq r \\ 1 \leq \beta \leq r \\ \alpha \neq \beta}} 
h(Y_{r+1} - Y_{\alpha},d,z)\cdot h(Y_{\beta} - Y_{r+1},d,z) \cdot 
\frac{\partial}{\partial H_{\alpha,\beta}}  + \\
\sum_{1 \leq \gamma \leq r} h(Y_{r+1} - Y_{\gamma},d,z)\cdot h(Y_{\gamma} - Y_{r+1},d,z) \cdot \frac{\partial}{\partial \omega_{\gamma}} \Biggr) \\
\Biggl[\prod_{1 \leq i \leq r} \omega_{i}^{l_i} \prod_{\substack{1 \leq a \leq r \\ 1 \leq b \leq r}}
H_{a,b}^{F_{a,b}} \Biggr] \Bigr \rvert_{\substack{\forall \gamma: \omega_{\gamma} = h(0,d,z) \\ \forall \alpha,\beta: H_{\alpha,\beta}= \mathcal{H}_{\alpha,\beta}(Y,z)}}
\end{multline}
This is true for any graph $G$.  The left hand side of \eqref{eq_mdet2} according to the proof of \citep[Theorem 2.3, Lemma 2.1.]{dhoef1} is just a finite (because of $z \diamond 2N$) linear combination of graph contributions of the above form. Their  weigths are independant of $z$. Therefore the lemma is true.
\end{proof}
\begin{remark} In the case of $r = 1$ the Lemma is true, in this case one has to define $cof(\hat A - \hat F) := 1$.
\end{remark}
We now come to non closed random walks.
\begin{definition}
Let $Y \in J_{r}$  and $k$ as before and $Y_{r+1} \neq Y_{r+2} \in \Z^d$ with $Y_j \neq Y_i \quad \forall i=1,\ldots,r; \; j=r+1,r+2 $  denote two other points. Then we define
\begin{multline}
W_{N,Y_{r+1},Y_{r+1}}(Y,k_1,\ldots,k_r) := 
\{ 
w: \textrm{length(w)} \leq N  \wedge \\
\textrm{w starts in $Y_{r+1}$ and ends in $Y_{r+2}$} \wedge mu(Y_i,w) = 2\cdot k_i \; \forall i= 1,\ldots,r
\}
\end{multline}
\end{definition}
As we are not interested in this paper in the distribution of $N_{2k+1}(w)$ we also define
\begin{definition}
A semi Eulerian graph $Ds=(V,E,J)$ is called double simple if it has exactly two vertices of degree 1.  
\end{definition}
\begin{remark}
Taking the two vertices of degree 1 of a double simple graph $Ds$ and replacing them by an edge between the outgoing vertex of the one pendant edge and the incoming vertex (which can be the same) of the other one uniquely defines a Eulerian graph $C(Ds)$ which has the same number of Euler trails as $Ds$. Replacing any edge $\tilde{e}$ of a Eulerian graph $G$ by one outgoing  pendant edge and one incoming pendant edge compatible with the direction of $\tilde{e}$  one gets a uniquely defined double simple graph. We denote the set of these double simple graphs by $\partial G$ 
\end{remark}
\begin{theorem}\label{t_odet}
\begin {multline} \label{eq_mdet3}
\sum_{w \in W_{N,Y_{r+1},Y_{r+2}}(Y,k_1,\hdots,k_r)} z^{length(w)} = \\
\Biggl(\sum_{\substack{1 \leq \alpha \leq r \\ 1 \leq \beta \leq r \\ \alpha \neq \beta}} 
h(Y_{r+1} - Y_{\alpha},d,z)\cdot h(Y_{\beta} - Y_{r+2},d,z) \cdot 
\frac{\partial}{\partial H_{\alpha,\beta}}  + \\
\sum_{1 \leq \gamma \leq r} h(Y_{r+1} - Y_{\gamma},d,z)\cdot h(Y_{\gamma} - Y_{r+2},d,z) \cdot \frac{\partial}{\partial \omega_{\gamma}} \Biggr)
 \\ 
\Biggl[\Biggl[ 
\left( \prod_{i=1}^r \hat K(x_i,k_{i},\omega_i) \right)
cof(\hat A - \hat F) \\
\left( \prod_{j=1}^r \frac {1} {(1-x_j)^2} \right)
\frac {1} {det(W^{-1} + X)} 
\Bigr\rvert_{\substack{\forall a: x_a = 0 \\ \forall b,c: X_{b,c}= U_{b,c}}}
\Biggr]\Bigr \rvert_{\substack{\forall \gamma: \omega_{\gamma} = h(0,d,z) \\ \forall \alpha,\beta: H_{\alpha,\beta}= \mathcal{H}_{\alpha,\beta}(Y,z)}}
\Biggr]_{z \diamond N}
\end {multline}
\end{theorem}
\begin{proof} For a walk $w \in W_{N,Y_{r+1},Y_{r+2}}(Y,k_1,\hdots,k_r) $  we define the number of Archetypes between a double simple Graph $Ds$ which maps the respective vertices of degree 1 onto the starting and ending point with
\begin{equation} 
Arch(Ds,w,Y,Y_{r+1},Y_{r+2})
\end{equation} Using the same argument as in the proof of \citep[Theorem 2.1.]{dhoef1} one gets 
\begin{equation}
\sum_{Ds \in \partial{G}: [G] \in \Theta (m_1,\ldots,m_r)} \frac {\sharp Arch(Ds,w,Y,Y_{r+1},Y_{r+2})} {\sharp Aut(G)} =
\sum_{k_i \geq m_i}
\prod_{j=1}^{l} \binom {k_j} {m_j}
\end{equation} 
As the generating function of the  number of Archetypes between double simple graphs and walks from $Y_{r+1}$ to $Y_{r+2}$ with the right behaviour in the points $Y$,  $ N(z,Ds,\pi,Y,Y_{r+1},Y_{r+2})$  can on the other hand be expressed in the form 
\begin{multline} \label{eq_fGr3}
\sum_{Ds \in \partial G} N(z,Ds,\pi,Y,Y_{r+1},Y_{r+2}) = N_E(G)\cdot \\ \Biggl(\sum_{\tilde{e} \in E} h(\pi(J(\tilde{e})_1)-Y_{r+2},d,z) \cdot 
h(Y_{r+1}-\pi(J(\tilde{e})_2),d,z) \\
\left( \prod_{e \in E \setminus\{\tilde{e}\}} h(\pi(J(e)_1) - \pi(J(e)_2),d,z) \right)\Biggr)
\end{multline}
one gets to the equation \eqref{eq_mdet3} with the same reasoning as in the proof of Lemma \ref{le_diff}.
\end{proof}
\begin{remark} In the case of $r = 1$ the Theorem is true, in this case one has to define $cof(\hat A - \hat F) := 1$.
\end{remark}
\subsection{Existences of walks which hit once}
For later use we also show the existence of a walk which hits every point of a given set exactly once.
\begin{lemma}\label{le_hyper1}
For $d \geq 1$ and any $L \in \N\setminus\{0\}$  we define the hypercube $H_L := \{x \in \Z^d: \left| x_i \right| \leq L \}$. Then there is a walk from $x_s := (-L,\ldots,-L)$ to $x_e = (L,\ldots,L)$ for which any  point of $H_L \setminus \{x_s,x_e\}$ has multiplicity 2 and any point of $\Z^d \setminus H_L$ has multiplicity $0$.  
\end{lemma}
\begin{proof} By induction in $d$. For $d=1$ the walk is obvious. For $d \geq 2$ we start in
$x_s=(-L,\ldots,-L)$ and walk till $x_{(e,-L)} = (L,\ldots,L,-L)$ in the way the lemma allows us in $d-1$ dimensions, then take the step forward in the $d^{th}$ direction to 
$x_{(e,-L+1)} = (L,\ldots,L,-L+1)$ and continue our walk to $x_{(s,-L+1)} := (-L,\ldots,-L,-L+1)$ in the way the lemma allows us in $d-1$ dimensions. We so continue layer by layer in our meandering movements till we reach $x_e$. Any point of $H_L$ other than the starting and ending point has multiplicity $2$, any point outside $H_L$ has multiplicity $0$ obviously.
\end{proof}
\begin{lemma}\label{le_hito}
Let $d \geq 2$ and $Y \in J_r$. Then there is a closed walk such that each of the $Y_i$ has multiplicity $2$ with respect to this walk.
\end{lemma}
\begin{proof} For any $Y \in J_r$ all points $Y_i$ are inside a hypercube $H_L$ of sufficiently big $L$. For $d \geq 2$ there is obviously a walk from $x_e = (L,\ldots,L)$ to $x_s = (-L,\ldots,-L)$  where each point in $H_L \setminus\{x_s,x_e\}$ has multiplicity $0$,  i.e. a walk avoiding $H_L$. Concatenating this walk with the one constructed in Lemma   
\ref{le_hyper1} we get a closed walk for which any point of the hypercube and therefore also all points $Y_i$ have multiplicity $2$. 
\end{proof}
\section{Useful properties of the functions $h(x,d,z)$}
The functions $h(x,d,z)$ are the major building blocks of our formulas. We give a general overview over their properties and especially deal with the zero points of $1 + h(0,d,z)$ as it comes up in the denominator of formulas. 
\subsection{Analytic properties}
\begin{lemma}\label{A2}
Let $D = \{z \in \C; \lvert \Re(z) \rvert < \frac{1}{2d} \}$. Then for $z \in D$ and $x:=(x_1,...,x_d) \in \Z^d$ the integral
\begin{equation}\label{A2e1}
h(x,d,z) := - \delta_{x,0} + \int_0^{\infty} dy\cdot e^{-y} \prod_{i=1}^d I_{x_i}(2\cdot y \cdot z)
\end{equation}
a) is holomorphic \\
b) has a continuous continuation onto $\overline{D}$ for $d \geq 3$ and $h(0,2,z)$ has a continuous continuation onto $\overline{D}\setminus\{-1/4,+1/4\}$ and \\
c) there are continuous functions $\lambda,C: D \longmapsto \R_+^*$ in $z$ such that 
\begin{equation}\label{A2e2}
\lvert h(x,d,z) \rvert \leq C(z)\cdot exp\biggl(-\lambda(z)\cdot \sum_{i=1}^d \lvert x_i \rvert \biggr)
\end{equation}
\end{lemma}
\begin{proof} a)
From \citep[eq. 9.1.19]{danraf} we know that for $n \in \N$ and $z \in \C$
\begin{equation}\label{C1}
\lvert I_n(z) \rvert =
 \left| \frac{1}{\pi} \int_0^{\pi} e^{z\cdot cos(\theta)} cos(n\cdot \theta) d\theta \right| \leq
\frac{1}{\pi} \int_0^{\pi} e^{\lvert \Re(z) \rvert} d\theta  \leq e^{\lvert \Re(z) \rvert}
\end{equation}
 Therefore we immediately know that
\begin{equation} \label{A2e6}
%\begin{array}
 \Biggl\lvert  e^{-y}  \prod_{i=1}^d I_{x_i}(2\cdot y \cdot z) \Biggr\rvert \leq exp(-(1-2\cdot d \cdot \lvert \Re(z) \rvert ) \cdot y)
%\end{array}
\end{equation}
for $z \in D$ and $y \in \R^*_+$. Therefore the integral in equation \eqref{A2e1} converges absolutely and is holomorphic for $z \in D$ and 
\begin{equation}\label{A2e7}
\lvert h(x,z,d) \rvert \leq \delta_{x,0} + \frac{1}{1-2\cdot d \cdot \lvert \Re(z) \rvert}
\end{equation}
b) Using the asymptotic expansion \citep[eq. 9.7.1]{danraf} for $I_x$ we see that b) is true for $d \geq 3$. As $h(0,2,z) = 2/\pi \cdot K(16\cdot z^2) -1 $ with the complete Eliptic Integral of the first kind $K$ from \citep[Lemma 3]{gasull} we  see that $h(0,2,z)$ can be extended holomorphically to the superset $\C \setminus\{z \in \R, \left|z\right| \geq 1/4\}$ of   $\overline{D}\setminus\{-1/4,+1/4\}$. \\
For c) let us define $ \bar x := \sum_{i=1}^d \lvert x_i \rvert $ and assume $\bar x \geq 3d$ in the sequel. 
For the real variable $u \in \R^*_+$ we define
\begin{equation}\label{A2e8}
 In_1(u) := \int_{u\cdot \bar x}^{\infty}dy \cdot  e^{-y} \prod_{i=1}^d I_{x_i}(2\cdot y \cdot z) 
\end{equation}
\begin{equation}\label{A2e82}
In_2(u) := \int_{0}^{u\cdot \bar x}dy \cdot  e^{-y} \prod_{i=1}^d I_{x_i}(2\cdot y \cdot z) 
\end{equation}
and therefore $h(x,z,d) = In_1(u) + In_2(u)$. According to equation \ref{A2e6} 
\begin{equation} \label{A2e10}
\lvert In_1(u) \rvert \leq \frac{exp(-(1-2\cdot d\cdot \lvert \Re(z) \rvert)u\cdot \bar x)}{1-2\cdot d \cdot \lvert \Re(z) \rvert}
\end{equation}
For $In_2(u)$ we use the relation
\begin{equation}
\begin{array}{lcc}
\lvert I_m(w) \rvert \leq \frac{\lvert w \rvert^{\lvert m \rvert}\cdot e^{\lvert \Re(w) \rvert}}{2^{\lvert m \rvert|}\Gamma(\lvert m \rvert +1)} & \forall_{w \in \C} &\forall_{m \in \Z} \\ 
\end{array}
\end{equation}
from \citep[eq. 9.1.62]{danraf} and \citep[eq. 9.6.3., eq.9.6.6.]{danraf}. We note 
\begin{equation}
\lvert In_2(u) \rvert \leq \frac{ \lvert z \rvert^{\bar x} \cdot \bar x^{\bar x + 1}}{\prod_{i=1}^d \Gamma(\lvert x_i \rvert + 1)} \cdot \int_{0}^u d\bar y \cdot \bar y^{\bar x} \cdot exp(-(1-2\cdot d \cdot \lvert \Re(z) \rvert)\cdot \bar y \cdot \bar x)
\end{equation}
The exponential in the integral has an absolute value smaller than 1 and so we are left with 
\begin{equation}
\lvert In_2(u) \rvert \leq \frac{ \lvert z \rvert^{\bar x} \cdot (u\cdot \bar x)^{\bar x + 1}}{(\bar x + 1)\cdot \prod_{i=1}^d \Gamma(\lvert x_i \rvert + 1)}
\end{equation}
For a given $\bar x \geq 3\cdot d$ we use that $\Gamma$ is logarithmically convex together with Jensen's inequality to get
\begin{equation}
\prod_{i=0}^d \Gamma(\lvert x_i \rvert + 1) \geq \Gamma\left(\frac{\bar x}{d}+1 \right)^d
\end{equation}
With the Stirling formula \citep[eq. 6.1.38]{danraf} we find:
\begin{equation}\label{A2e12}
\lvert In_2(u) \rvert \leq (m(z)\cdot e \cdot u)^{\bar x} \cdot \Biggl(\frac{d}{2\pi \bar x} \Biggr)^{\frac{d}{2}} \cdot \frac{u\cdot \bar x}{\bar x + 1}
\end{equation}
with $m(z) := max(1, d \cdot \lvert z \rvert)$.
We realize that the equation 
\begin{equation}\label{A2e13}
u = \frac{1}{m(z)\cdot e}\cdot exp(-(1-2\cdot d \cdot \lvert \Re(z) \rvert)\cdot u)
\end{equation}
has a unique solution for any $z \in D$ as the left hand side of \eqref{A2e13} is a contiuous and stricly monotonous growing bijection of $R^*_+$ and the right hand side is continuous and stricly monotonously falling from $\frac{1}{m(z)\cdot e}$ to $0$. 
We denote this solution with $u^*(z)$ and then use equation \eqref{A2e8} and \eqref{A2e82} for $u=u^*(z)$ and use equations \eqref{A2e10} and \eqref{A2e12} and therefore notice:
\begin{equation}
\lvert h(x,d,z) \rvert \leq  e^{-(1-2\cdot d \cdot \lvert \Re(z) \rvert)\cdot u^*(z) \cdot \bar x} \cdot
\Biggl( \frac{1}{1-2\cdot d \cdot \lvert \Re(z) \rvert} + \frac{u^*(z)\cdot \bar x}{\bar x +1 }\cdot \Biggl( \frac{d}{2\pi\bar x}\Biggr)^{\frac{d}{2}}\Biggr)
\end{equation}
From equation \eqref{A2e13} it is also immediately clear that 
\begin{equation}
\frac{1}{m(z)\cdot e} \geq u^*(z) \geq \frac{1}{m(z)\cdot e}\cdot exp \Biggl(-\frac{1-2\cdot d \cdot \lvert \Re(z) \rvert}{m(z)\cdot e} \Biggr)
\end{equation}
Defining the function
\begin{equation}
\begin{array}{rccl}
\lambda:& D & \longmapsto & R^*_+\\
    & z & \longmapsto &  \frac{1-2\cdot d \cdot \lvert \Re(z) \rvert}{m(z)\cdot e}  \cdot exp\Biggl(-\frac{1-2\cdot d \cdot \lvert \Re(z) \rvert}{m(z)\cdot e} \Biggr)
\end{array}
\end{equation}
with equation \eqref{A2e7} we can now write (even for $\bar x \leq 3d$)
\begin{equation}
\lvert h(x,d,z) \rvert \leq \Biggl( \frac{exp(3\cdot\lambda(z)\cdot d)}{1 - 2\cdot d \cdot \lvert \Re(z) \rvert} + 1 + \Biggl(\frac{d}{2\pi}\Biggr)^{\frac{d}{2}} \cdot \frac{1}{m(z)\cdot e} \Biggr) \cdot exp(-\lambda(z)\cdot \bar x)
\end{equation}
which proves the Lemma. 
\end{proof}
As $1 + h(0,d,z)$ often appears in the denominator of functions we must understand where it becomes zero:
\begin{lemma}\label{A3}
Let $P_0 := \{ z \in D;1+ h(0,d,z) = 0\}$
then \\
a) $P_0$ is discrete in $D$ \\
b) there is an $\epsilon > 0$ such that $U_{\epsilon} (\frac{1}{2d}) \cap D \cap P_0 = \emptyset $ and $U_{\epsilon}(-\frac{1}{2d}) \cap D \cap P_0 = \emptyset$ \\
c) $P_0 \cap \bar U_{\frac{1}{2d}} (0) \cap D = \emptyset$  \\
d) $z \in \bar U_{\frac{1}{2d}}(0) \cap D \Rightarrow \lvert 1 + h(0,d,z) \rvert \geq 1/2$ 
\end{lemma}
\begin{proof} 
a) D is open and $h(0,d,z)$
%\begin{equation}
%\begin{array}{rccl}
%h_0^d:& D & \longmapsto & \C \\
 %   & z & \longmapsto &  h(0,d,z) \\
%\end{array}
%\end{equation}
is holomorphic on the open set $D$ according to Lemma \ref{A2} and therefore $P_0$ is discrete in $D$ and so a) is true.\\
b) For $d \geq 3$ from Lemma \ref{A2}, b) we know that the constant 
\begin{equation}
G_0(0,d) := \lim_{z \rightarrow \pm 1/2d} h(0,d,z)
\end{equation}
exists. It is related to the return probability $p(d) = 1- 1/(1+ G_0(0,d))$ of the $d$ dimensional simple random walk and so $G_0(0,d) > 0$. But as according to Lemma \ref{A2},  b) $h(0,d,z)$ has a continuous continuation on $\overline{D}$ then $h(0,d,z) +1 $ cannot be zero in a full neighborhood of $\pm 1/2d$ in $D$.\\
For $d=2$ from \citep[Lemma 5]{gasull} we know that $1 + h(0,2,z)$ is never zero on $D$ (actually never zero on  $\C \setminus\{z \in \R, \left|z\right| \geq 1/4\}$ ). \\
To prove c) and d) we use Lemma \ref{le_diff} for $r = 1$ and for a $y \in \Z^d \wedge y \neq 0$ and get
\begin{equation}\label{eq_h0p0}
\sum_{w \in W_{2N,y}(0,k)} z^{length(w)} = \left[ \frac{h(y,d,z)^2}{(1+ h(0,d,z))^2} \cdot \left(\frac{h(0,d,z)}{1+h(0,d,z)}\right)^{k-1}\right]_{z \diamond 2N}
\end{equation}
The left hand side of equation \eqref{eq_h0p0} for $z \in U_{\frac{1}{2d}}(0)$  is dominated by 
\begin{equation}\label{eq_bound}
\left| \sum_{w \in W_{2N,y}(0,k)} z^{length(w)} \right| \leq \sum_{w \in W_{2N}} \lvert z\rvert ^{length(w)} \leq h(0,d,\lvert z \rvert)
\end{equation}
independantly of $N$ and $k$ and therefore the right hand side must also converge absolutely for all $z \in U_{\frac{1}{2d}}(0)$ in the limit $N \rightarrow \infty$ to a value smaller or equal to $h(0,d,\lvert z \rvert)$. So $1 + h(0,d,z) =0$ cannot be true for any $z \in U_{\frac{1}{2d}}(0)$ as for sufficiently big $k$ the right hand side of equation  \eqref{eq_h0p0} otherwise had a pole and the radius of convergence of the series would be smaller than $1/2d$. 
We now consider the behaviour of  
\begin{equation}
g_{h}(z) := \frac{h(0,d,z)}{1+h(0,d,z)}
\end{equation}
as a meromorphic function on $D$.
Let us assume that there was a $z_0 \in  \bar U_{\frac{1}{2d}}(0) \cap D$ such that it 
fulfills $\lvert g_{h}(z_0) \rvert > 1$.
Then there is an $\epsilon > 0$ and a $\delta >0$ such that for any $z \in U_{\epsilon}(z_0) \cap D$ we have $\lvert g_h(z) \rvert > 1 + \delta$. As the zero points of $h(y,d,z)$ are discret in D we can find a  $z_1 \in U_{\epsilon}(z_0) \cap D$ with $\lvert z_1 \rvert < 1/2d$ and $ h(y,d,z_1) \neq 0$. For this $z_1$ the term 
\begin{equation}\label{eq_gh1}
 \left| \frac{h(y,d,z_1)^2}{(1+ h(0,d,z_1))^2} \cdot \left(\frac{h(0,d,z_1)}{1+h(0,d,z_1)}\right)^{k-1} \right| > \left| \frac{h(y,d,z_1)^2}{(1+ h(0,d,z_1))^2} \right| \cdot (1 + \delta)^{k-1}
\end{equation} 
is unbounded for $k \rightarrow \infty$,
which contradicts equation \eqref{eq_bound}. Therefore for any $z \in \bar  U_{\frac{1}{2d}}(0) \cap D$ we have $\left| g_h(z) \right| \leq 1$ which is equivalent to d). But of course d) implies c).  
\end{proof}
As we shall see later again arguments from geometry like the above are important for the discussion.  
\begin{definition}\label{d_rh}
For $d \geq 2$ we define the real number 
\begin{equation}
r_h(d) := inf \{\lvert z \rvert \; ; z \in P_0\}
\end{equation}
We already know that $r_h(2) = \infty$ and $r_h(d) > 1/2d$.
\end{definition}
\subsection{Asymptotic Expansions}
We now come to various asymtotic expansions of the functions $h(x,d,z)$. We start with 
\begin{definition}
We define the function
\begin{equation}
\begin{array}{rccl}
\varrho:& \Z^d & \longmapsto & \R \\
   	 & x       & \longmapsto &  
\begin{cases} 0 \Longleftarrow \left(\sum_{i=0}^d x_i \right)\%2 = 0 \\
1 \Longleftarrow otherwise \end{cases} \\
\end{array}
\end{equation}
\end{definition}
\begin{remark}
For $x, y \in \Z^d$ the function $\varrho$ fulfills the equations 
\begin{equation}\label{eq_rho}
\varrho(x + y) = \left(\varrho(x) + \varrho(y)\right)\%2
\end{equation}
\end{remark}
If we write 
\begin{equation}
h(x,d,z) = \sum_{m=0}^{\infty} a_m(x)\cdot (2dz)^m
\end{equation}
using known asymptotic features of the Bessel function for large argument \citep[eq. 8.451.5, vol. 2]{rg} and corresponing infinite series \citep[eq. 6.12.1]{tsp}  we can write an asymptotic expansion
\begin{equation}
a_m(x) = \left(1 + (-1)^{m + \varrho(x)}\right)
\cdot \left( \frac{d}{2 \pi  m} \right)^{\frac{d}{2}}
\cdot \left( 1 + \sum_{i=1}^{n} \frac{J_i(x,d)}{m^i} + O\left(\frac{1}{m^{(n+1)}}\right) \right)
\end{equation}
where we give the first two coefficients as an example:
\begin{multline}
J_1(x,d)= \frac{d\cdot(1 + 2\lvert x \rvert^2)}{4} \\
J_2(x,d)=  \frac{d \cdot \left(2d^2 - 3d + 4 + 12\lvert x\rvert^2(d + 4) + 12d\lvert x\rvert^4 \right)}{96} \\ 
\end{multline}
But therefore for a fixed $x$ and $d>0$ there is an asymptotic expansion
 in the vicinity of $z = \pm \frac{1}{2d}$
\begin{multline}\label{eq_A3_even}
h(x,d,z)  = (2dz)^{\varrho(x)}\biggl[\sum_{i=0}^n G_i(x,d)\cdot\left( 1-(2dz)^2\right)^i 
+ O\left( \left( 1-(2dz)^2\right)^{n+1}\right) \\
+2\cdot\left(\frac{d}{4\pi}\right)^\frac{d}{2} \cdot 
\frac{(-1)^{\frac{d + d\%2}{2}} \pi^{d\%2}}{\Gamma(\frac{d}{2})}\cdot\\
\left(\ln\left(1-(2dz)^2\right)\right)^{1-d\%2}\cdot
\left(1-(2dz)^2\right)^{\frac{d}{2}-1}
 \cdot \\
 \biggl( 1 
+\sum_{i=1}^n H_i(x,d)\cdot\left( 1-(2dz)^2\right)^i
+ O\left( \left( 1-(2dz)^2\right)^{n+1}\right)
\biggr)\biggr]
\end{multline}
where again we give the first two coefficients as an example
\begin{multline}
H_1(x,d)= \frac{d -1  + 2\lvert x \rvert^2}{4} + \varrho(x)\cdot\frac{1}{2} \\
H_2(x,d)=  \frac{3d^3 + 12d^2- 3d -12 + 12\lvert x\rvert^2(d^2 +3d + 4 ) + 12d\lvert x\rvert^4}{96(d+2)} + \\
\varrho(x)\cdot \frac{d +2  + 2\lvert x\rvert^2}{8} \\
%H_3(x,d) = \frac{(d-1)\cdot(d^4 + 16d^3 + 75d^2 + 120d + 96) - 16dx_4 + 2\cdot
%(3d^4 + 36d^3 + 157d^2 + 252d + 288)\lvert x \rvert^2 + 12d(d^2 + 7d + 20) \lvert x \rvert^2 + 8d^2 \lvert x \rvert^6}%{384(d+2)(d+4)}
\end{multline}
with the Euclidian norm $\lvert x \rvert$ and real constants $G_i(x,d)$. As mentioned before  for $d>2$ the $G_0(0,d)$ are related to the return probability $p(d) = 1- 1/(1+ G_0(0,d))$ and so $G_0(0,d) > 0$. For $d\leq 2$ from \citep[eq. 2.21]{dhoef1} we find 
\begin{equation}
G_0(0,2) = 4\frac{\ln(2)}{\pi}-1
\end{equation}
according to \citep[eq. 8.113, vol. 1]{rg} (where a closed form for all coefficients in equation \eqref{eq_A3_even} for $x=0$ and $d=2$  is given) and $G_0(0,1) = -1$ . 
\begin{remark}\label{re_heven}
$h(x,d,z)$ is an even function in $z$ if $\varrho(x) = 0$ and odd otherwise and for all $x \in \Z^d$
\begin{equation}
h(x,d,z) = h(-x,d,z)
\end{equation} 
\end{remark}
We now come to asymptotic expansions for large $x \in \Z^d$:
\begin{lemma}\label{B1}
For $\lambda \in \left]0, \frac{\pi}{2} \right[$ let the open set 
\begin{equation}
W_{\lambda} :=\{ z \in \C \setminus \{0\}, \left| \arg(z) \right| < \lambda \}
\end{equation}
 and 
$\overline{W_{\lambda}}$ be its closure. For $\nu \in \R; \nu \geq 0$ we define 
\begin{equation}
\begin{array}{rccl}
\Theta_{\nu} :& \longmapsto & \C \\
  w & \longmapsto & K_{\nu}(w) \cdot (\frac{w}{2} )^{\nu} \cdot 2
\end{array}
\end{equation} 
with the Modified Bessel function $K_\nu$. 
Then $\Theta_{\nu}$ is holomorphic on $W_{\lambda}$. For $\nu >0$ the function $\Theta_{\nu}$ has a continuous continuation onto $\overline{W_{\lambda}}$ and is bounded 
\begin{equation}
\left| \Theta_{\nu} (w) \right| < C_{\nu,\lambda} < \infty \Leftarrow w \in \overline{W_{\lambda}}
\end{equation}
\end{lemma}
\begin{proof} From the asymptotic expansion \citep[eq. 9.7.2.]{danraf} for $K_{\nu}$ we know that for $w \in W_{\lambda}$ $\lim_{\left| w \right| \rightarrow \infty} \Theta_{\nu} (w) = 0$. From \citep[eq. 9.6.9.]{danraf} we see that   for $\nu >0$ and $w \in W_{\lambda}$ $\lim_{\left| w \right| \rightarrow 0} \Theta_{\nu} (w) = \Gamma(\nu)$.  From \citep[ch. 9.6]{danraf} we know that $K_{\nu}$ for $\nu > 0$ has no other finite singularity or branch point on $\C$.  Therefore $\Theta_{\nu}$ is bounded on $\overline{W_{\lambda}}$ according to the maximum principle.  
\end{proof}
\begin{lemma}\label{B2}
We define the set $V_r := \{z \in \C, \frac{r}{2} < \left| z \right| < r \}$.
For $d \geq 2$ and $z \in V_{\frac{1 }{2d}} \cap W_{\lambda}$ and $\lambda < \frac{\pi}{2}$  there is an asypmtotic expansion for $\left| x \right| \rightarrow \infty$
\begin{multline}\label{eq_B2}
h(x,d,z)  = \frac{(2d + s ) }{4\pi^{\frac{d}{2}} \left| x \right|^{d-2}} \cdot \Bigl\{ \Theta_{\frac{d}{2} -1} (\left| x \right| \cdot \sqrt{s}) \\  
- \frac{1}{2 \left| x \right|^2} \cdot \left( 
2 \cdot \Theta_{\frac{d}{2} + 1} (\left| x \right| \cdot \sqrt{s}) 
-  \frac{d}{2}\cdot\Theta_{\frac{d}{2}} (\left| x \right| \cdot \sqrt{s})
-\frac{2\cdot x_4}{3 \left| x \right|^4} \cdot
 \Theta_{\frac{d}{2} + 2} (\left| x \right| \cdot \sqrt{s}) \right)  \\
+ O\left(\frac{1}{\left| x \right|^4}\right)
\Bigr\}
\end{multline}
where $s := \frac{1}{z} - 2d$, $x_4 = \sum_{i=1}^d x_i^4$. The higher orders also have coefficient functions which are bounded by $1/\left| x \right|^{2m}  $ times constants independant of $z$ and $x$.  
\end{lemma}
\begin{proof} For $z \in V_{\frac{1}{2d}} \cap W_{\lambda}$ by simple geometrical reasoning in the complex plain we find $\left| s \right| < 6d$ and $\arg(s) < \widehat{ \lambda}$ with 
\begin{equation}
\widehat{\lambda} := \frac{\pi}{2} - \arctan\left( \frac{\sin(\lambda)}{1-\cos(\lambda)}  \right) < \frac{3\pi}{4}
\end{equation}
and therefore for any $x$ we have $x\cdot\sqrt{s} \in W_{\frac{\widehat{\lambda}}{2}}$. 
The asymptotic expansion follows from the one for the Bessel functions in \citep[eq. 9.7.7.]{danraf}; it is uniform for $z \in W_{\lambda}$ and $\lambda < \frac{\pi}{2}$. 
But $\lambda < \frac{\pi}{2} \rightarrow \widehat{\lambda} < \frac{3\pi}{4}$ and therefore according to Lemma \ref{B1} we have for $\nu > 0$
\begin{equation}
\left| \Theta_{\nu}(\left| x \right| \cdot \sqrt{s})\right| < C_{\nu,\frac{3\pi}{8}}
\end{equation}
Now the higher coefficients of the asypmtotic expansion in eq. \eqref{eq_B2} are given by functions $\Theta_{\frac{d}{2} + k}(\left| x \right| \cdot \sqrt{s})$ with $k> 2$ times polynomials in s times elementary symmetric polynomials in $x_i^2$ of degree $\mu$ divided by $\left| x \right|^q$ where $q-2\mu \geq 2m$ beginning with $m=2$ for the first term not explicitely given in equation \eqref{eq_B2} which proves the lemma. 
\end{proof}
\begin{lemma}\label{B3} For $d \geq 2$ and $z \in D$ the ineqality 
\begin{equation}\label{eq_hxh0}
\left| h(x,d,z) \right| \leq h(x,d,\left|z\right|) < 1 + h(0,d,\left|z\right|) 
\end{equation}
is true. \\
For $d \geq 3$ there is a constant $C(d)$ such that for $z \in V_{\frac{1}{2d}} \cap W_{\frac{\pi}{2}}$ and $x \in \Z^d \setminus \{0\}$ 
\begin{equation}\label{eq_hx}
\left| h(x,d,z) \right| < \frac{C(d)}{\left| x \right|^{d-2}}
\end{equation}
\end{lemma}
\begin{proof}
As the Bessel function $I_n$ has only nonnegative Taylor coefficients we know 
\begin{equation}
\left| I_{x_i}(2yz) \right| \leq I_{x_i}(2y\left|z\right|) 
\end{equation}
and therefore the first inequality of \eqref{eq_hxh0} is true. Now for real nonegative $u \in \R_+$ from \citep[eq. 9.6.19]{danraf} we know that for $n \neq 0$
\begin{equation}
I_n(2yu) < I_0(2yu)
\end{equation}
and from this the second inequality of \eqref{eq_hxh0} follows.
As $h(x,d,z)$ for $z \in U_{\frac{1}{2d}}(0)$ is the generating function of walks between $0$ and $x$ it has Taylor coefficients which are nonnegative integers and therefore is a strictly monotonously growing function for nonegative real $u \in [0,\frac{1}{2d}[$.
For $d \geq 3$ therefore we have 
$h(x,d,u) \leq 1 + G(0,d)$.
From Lemma \ref{B2} on the other hand we know that for $d \geq 3$ there is an $\overline{x_0} \in \R$ such that for $\left| x \right| \geq \overline{x_0}$ and $z \in V_{\frac{1}{2d}} \cap W_{\frac{\pi}{2}}$
\begin{equation}
\left| h(x,d,z) \right| < \frac{8\cdot d}{4 \pi^{\frac{d}{2}} \left| x \right|^{d-2}} \cdot (C_{\frac{d}{2} - 1, \frac{3\pi}{8}} + 1) 
\end{equation}
but from this the Lemma follows by choosing
\begin{equation} 
C(d) = \max\left((1+ G_0(0,d))\cdot \overline{x_0}^{d-2}, \frac{2d}{ \pi^{\frac{d}{2}}} \cdot (C_{\frac{d}{2} - 1, \frac{3\pi}{8}} + 1) \right)
\end{equation}
\end{proof}
\begin{lemma}\label{B3_d2}
For $d=2$ let there be infinite sequences $x_n \in \Z^2$, $z_n \in U_{\frac{1}{2d}}(0)$ and $z_n \rightarrow \frac{1}{2d}$. Then there are four different scenarios for $h(x_n,2,z_n)$ in case $\sqrt{s_n} \left|x_n \right|$ converges. In the notation of Lemma \ref{B2}: 
\begin{equation}\label{B3_e2}
\begin{array}{rcccccl}
h(x_n,2,z_n) & = & \Delta_n & \rightarrow & 0 & \Leftarrow &  \sqrt{s_n} \cdot \left| x_n \right| \rightarrow \infty \\
h(x_n,2,z_n) & = & \Gamma_n & \rightarrow & \frac{2}{\pi} K_0(w) &  \Leftarrow & \sqrt{s_n} \cdot \left| x_n \right| \rightarrow w \in \C\setminus\{0\}
\end{array}
\end{equation}
For $ \sqrt{s_n} \cdot \left| x_n \right| \rightarrow 0 $ and $\left| x_n \right| \rightarrow \infty$
\begin{multline}\label{B3_e3}
h(x_n,2,z_n) = 1 + h(0,2,z_n)  -\frac{2}{\pi}\ln\left(\left| x_n \right|\right)  -\frac{2}{\pi}\ln\left( 2 \sqrt{2}\cdot e^{\gamma} \right) \\ + O(\frac{1}{\left| x_n \right|^2}, s_n\cdot\left| x_n \right|^2 \cdot \ln(s_n \cdot \left| x_n \right|^2) )
\end{multline}
For $x_n \equiv x$
\begin{equation}\label{B3_e4}
h(x,2,z_n) = 1 + h(0,2,z_n) +  G_0(x,2) - \frac{4}{\pi}\ln(2) + O(s_n \cdot \ln(s_n) )
\end{equation}
where $\gamma$ is the Eulerian constant.
\end{lemma}
\begin{proof} Equations \eqref{B3_e2} and \eqref{B3_e3} follow directly from Lemma \ref{B2} and \citep[eq. 9.5.53]{danraf}. Equation \eqref{B3_e4} follows from \eqref{eq_A3_even}.
\end{proof} 
\subsection{First hit functions}
\begin{lemma}\label{le_firsthit}
For $x \in \Z^d \setminus\{0\}$ 
\begin{equation}
\frac{h(x,d,z)}{1 + h(0,d,z)}
\end{equation}
is the generating function for the number of walks between $0$ and $x$ hitting $x$ the first time when it reaches $x$. 
\end{lemma}
\begin{proof} Any walk between $0$ and $x$ is either a walk which hits $x$ first when it arrives or is a unique concatenation of such a first hit walk with a walk from $x$ to $x$ (which is the same as a walk from $0$ to $0$ up to a shift of the starting point). Therefore denoting the generating function of first hit walks with $h_{fh}(x,d,z)$ we have 
\begin{equation}
h(x,d,z) = h_{fh}(x,d,z) + h_{fh}(x,d,z) \cdot h(0,d,z)
\end{equation}
which proves the Lemma.
\end{proof} 
\begin{lemma}\label{le_phi2}
For $d \geq 2$ and  $z \in U_{\frac{1}{2d}}(0) $ we define
\begin{equation}
\phi_2(Y,z) := 1- \Delta_2(Y,z) = \left(\frac{h(Y_2,d,z)}{1 + h(0,d,z)} \right)^2
\end{equation}
Then
\begin{equation}\label{eq_inephi}
\left| \phi_2(Y,z) \right| \leq \phi_2(Y,\left|z\right|) < 1 
\end{equation}
\begin{proof}
From Lemma \ref{le_firsthit} we know that the Taylor coefficients of $\phi_2$ are all nonnegative which proves the first inequality in \eqref{eq_inephi}. The second inequality follows from Lemma \ref{B3} equation \eqref{eq_hxh0}.
\end{proof}
\end{lemma}
\begin{lemma}\label{le_sec}
For $z \in U_{\frac{1}{2d}}(0)$ the following equation is well defined and true.
\begin{multline}\label{eq_sec}
\underset{N \rightarrow \infty}{\lim}\sum_{w \in W_{2N}} P(N_{2k_1}(w),N_{2k_2}(w)) \cdot z^{length(w)} = \\ \left(1 - \frac{\delta_{k_1,k_2}}{2} \right) z \cdot \frac{\partial}{\partial z} \Biggl[
\sum_{x \in \Z^d \wedge x \neq 0} \sum_{f=1}^{\infty}\\ \Biggl(  K(f,k_1,h(0,d,z)) \cdot K(f,k_2,h(0,d,z)) \cdot f \cdot   \left(\frac{h(x,d,z)}{1 + h(0,d,z)}\right)^{2f} \Biggr)\Biggr]
\end{multline}
\end{lemma}
\begin{proof}
From Lemma \ref{le_phi2} we know that the sum
\begin{equation}
\sum_{f=1}^{\infty} \phi_2(x,z)^{f} \cdot f^m = \left(\chi \frac{\partial}{\partial \chi}\right)^m\frac{\chi}{1-\chi} \rvert_{\chi = \phi_2(x,z)}
\end{equation}
converges absolutely in $f$ and therefore the sum in $f$ on the right hand side of equation \eqref{eq_sec}. The result for any $m \in \N_0$ is a sum of terms of the form 
\begin{equation}
\frac{\phi_2(x,z)^{q}}{(1-\phi_2(x,z))^{p}}
\end{equation}
 where $q \geq 1$ and $p \leq m + 1$.  The  subsequent sum in $x$ because of Lemma \ref{A2} c) then converges absolutely in $x \in \Z^d$. Therefore the right hand side of equation \eqref{eq_sec} is well defined and holomorphic for $z \in U_{\frac{1}{2d}}(0)$. As the Taylor coefficients on both sides are identical as we have shown in \citep[Theorem 2.3.]{dhoef1} the Lemma is true. 
\end{proof}
\section{Useful properties of the first hit determinant}
From the equations in Lemmata \ref{le_fGr2} and \ref{le_diff} and Theorem \ref{t_odet} it is clear that the asymptotic behaviour of the moments for large length of the walks crucially depends on the behaviour of the determinant constantly coming up in the denominator of the equations. As we have just seen in Lemma \ref{le_sec} it is studying this entity which will allow us to define the generating functions of the higher moments as holomorphic functions. We therefore start our discussion of this crucial building block in this section and continue in the next one for a special treatment in the case $d=2$. 
\subsection{Analytic properties}
\begin{definition}
For $Y \in J_r$ and  $z \in  D \setminus P_0$ we define the first hit determinant 
\begin{equation}
\Delta_r(Y,z) := \det( 1_{r\times r} + U(Y,z))
\end{equation}
with 
\begin{equation}
U_{i,j}(Y,z) := \frac{h(Y_i - Y_j,d,z)}{1 + h(0,d,z)}
\end{equation}
\end{definition}
\begin{remark}
We remark that the first hit determinant is a meromorphic function on $D$ with possible poles in $P_0$, which has a Taylor expansion around $z=0$ starting with $1$ as zeroth coefficient and has a radius of convergence of $1/2d$. 
\end{remark}
\begin{lemma}\label{le_deven}
For $Y \in J_r$ and $z \in D \setminus P_0$ the function $\Delta_r(Y,z)$ is an even function of $z$.  
\end{lemma}
\begin{proof}
$\Delta_r(Y,z)$ is a determinant of a matrix which has $1$ in its diagonal. Using the Leibniz formula we can write it as a sum over permutation terms. Each permuation term because of the diagonal of the matrix being $1$ and the decomposition of a permutation in cycles $\chi$ is (up to a factor $-1$) the product of cycle terms of the form 
\begin{equation}
\prod_{(i,j) \in \chi} h(Y_i - Y_j,d,z)
\end{equation}
As the cycles are closed we know 
\begin{equation}
\sum_{(i,j) \in \chi} (Y_i - Y_j) = 0
\end{equation}
But therefore using equation \eqref{eq_rho} and Remark \ref{re_heven} each cycle term is even in $z$ from which the Lemma follows. 
\end{proof}
\begin{lemma}\label{A4}
Let $Y \in J_r$ be a vector and $Y' \in J_{r-1}$ be the vector which is composed of the first 
$r-1$ components of $Y$. 
Then $\forall {Y \in J_r}$ and $\forall z \in D  \cap \bar U_{\frac{1}{2d}}(0)$
\begin{equation}
\Delta_r(Y,z) \neq 0
\end{equation}
\begin{equation}\label{eq_A4_1}
\left| 1 - \frac{\Delta_{r-1}(Y',z)}{(1+ h(0,d,z))\Delta_r(Y,z)} \right| \leq 1
\end{equation}
\end{lemma}
\begin{proof}
Let us first define the set 
\begin{equation}
P_{r,Y} : = \{z \in D; z \in P_0 \vee  \Delta_r(Y,z) = 0 \vee \Delta_{r-1}(Y',z) = 0\}
\end{equation}
then $ P_{r,Y}$  is discrete in $D$.\\
Now from Lemma \ref{le_close} we know that for $k \in \N \setminus \{0 \}$
\begin{multline}
\sum_{(w,y) \in W_{2N} (Y,1,\hdots,1,k)}  z^{length(w)} = 
z \cdot \frac{\partial}{\partial z} \biggl[  \frac{1}{(1 + h(0,d,z))^{r-1}} \\
cof(\hat A - \hat F) \frac{1}{(1 + h(0,d,z))^{k}} \sum_{\nu=0}^{k-1} \binom{k-1}{\nu}
\frac{h(0,d,z)^{\nu}}{(k-\nu)!} \cdot \\ \left( \frac{\partial}{\partial x_r} \right)^{k-\nu -1} 
\left\{ \frac{1}{(1-x_r)^2} \frac{1}{g(x_r,X)}\right\}\biggr\|_{x_r = 0; X_{i,j} = U_{i,j}(Y,z)}
\biggr]_{z \diamond N}
\end{multline}
with 
\begin{equation}
g(x_r,X) := \det(diag(1,\cdots,\frac{1}{1-x_r}) + X)
\end{equation}
with $ X: X_{i,i} := 0$ the general indefinite offdiagonal matrix with variables $X_{i,j}$.
Defining 
\begin{equation}
D_r(x) := \det(1_{r\times r} + X) 
\end{equation}
\begin{equation}
D_{r-1}(X') := \det(1_{(r-1) \times (r-1)} + X')
\end{equation} 
where $X'$ is the $(r-1)\times (r-1)$ upper left submatrix of $X$ we find
\begin{equation}
g(x_r,X) = D_r(X) + \frac{x_r}{1-x_r}\cdot D_{r-1}(X')
\end{equation}
Expanding in $x_r$ and after straightforward calculations we find:
\begin{multline}\label{eq_A4_2}
\sum_{(w,y) \in W_{2N} (Y,1,\hdots,1,k)}  z^{length(w)} = 
\frac{1}{k}\cdot z \cdot \frac{\partial}{\partial z} \biggl[  \frac{1}{(1 + h(0,d,z))^{r-1}}
\cdot \\
cof(\hat A - \hat F) \frac{1}{D_{r-1}(X')} \cdot \\
 \left\{ 1 - 
\left( 1- \frac{D_{r-1}(X')}{(1 + h(0,d,z))\cdot D_r(X)}\right)^k\right\}\biggr\|_{ X_{i,j} = U_{i,j}(Y,z)}
\biggr]_{z \diamond 2N} =\\
\frac{1}{k} \cdot \biggl[ p_1(Y,z) + p_2(k,Y,z)\cdot g_r(Y,z)^{k-r}
\biggr]_{z \diamond 2N} 
\end{multline}
with 
 \begin{multline}
p_1(Y,z) :=z \cdot \frac{\partial}{\partial z} \biggl[  \frac{1}{(1 + h(0,d,z))^{r-1}}
\cdot \\
cof(\hat A - \hat F) \frac{1}{D_{r-1}(X')}\biggr\|_{X_{i,j} = U_{i,j}(Y,z)}
\biggr]
\end{multline}
being independant of $k$ and holomorphic on $D \setminus P_{r,Y}$
and $p_2(k,Y,z)$ being a polynomial in $k$ with coefficients depending on $r$ polynomially and on $z$ holomorphically on $D \setminus P_{r,Y}$ and defining
\begin{equation}
g_r(Y,z) := 1 - \frac{\Delta_{r-1}(Y',z)}{(1+ h(0,d,z))\Delta_r(Y,z)}
\end{equation}   
where again $g_r$ is a holomorphic function on $D \setminus P_{r,Y}$\\
Let us now assume that  $p_2$ was the zero polynomial in $k$ then 
\begin{equation}
\sum_{(w,y) \in W_{2N} (Y,1,\hdots,1,k)}  z^{length(w)} = \frac{1}{k}\cdot [p_1(Y,z)]_{z \diamond 2N}
\end{equation}
For a given $k$ we know that $(w,y) \in  W_{2N} (Y,1,\hdots,1,k) \Rightarrow length(w) \geq 2k$ from simple geometrical reasoning. This then means that $p_1 $ is the zero function as all its coefficients can be shown to vanish for greater and greater $k$. As $p_1$ is independant of $k$ we then know that for $k = 1$
\begin{equation}
\sum_{(w,y) \in W_{2N} (Y,1,\hdots,1,1)}  z^{length(w)} = [p_1(Y,z)]_{z \diamond 2N} \equiv 0
\end{equation}
which contradicts Lemma \ref{le_hito}.\\
Therefore $p_2$ cannot be the zero polynomial in $k$ and therefore as a polynomial in $k$ has a  degree $b$ and a highest term $a_b(Y,z) \cdot k^b$ with a function $a_b(Y,z)$ meromorphic on $D$ and not being identical to $0$.  
The left hand side of equation \eqref{eq_A4_2} is a series which for $ z \in D \cap U_{\frac{1}{2d}}(0)$ converges absolutely for $N \rightarrow \infty$  
\begin{multline}
\sum_{(w,y) \in W_{2N} (Y,1,\hdots,1,k)} \left| z \right|^{length(w)} \leq
\sum_{w \in W_{2N} } \left| z \right|^{length(w)} \cdot length(w) = \\ 
 \zeta \frac{d}{d\zeta} \left(  h(0,d,\zeta)_{\zeta \diamond 2N} \right) \rvert_{\zeta = \left| z \right|} 
\end{multline}
and therefore is dominated by 
 a value which is smaller or equal than $\left| z \right| \cdot h^{'}(0,d,\left| z \right|)$ independant of $k$ and $r$ because the coefficients of $h^{'}(0,d,\zeta)$ are nonnegative integers. If $\Delta_r(Y,z)=0$ were true for a $z \in U_{\frac{1}{2d}}(0)$ which was not also a zero point of $\Delta_{r-1}(Y',z)$ then for sufficiently big $k$ the right hand side of equation \eqref{eq_A4_2} had a pole and the radius of convergence would not be $1/2d$. But by regress in $r$ because $\Delta_1 \equiv 1$ we therefore know that $ P_{r,Y} \cap U_{\frac{1}{2d}}(0) = \emptyset$.\\  
Now let us assume that there is a $z_0 \in D \cap \bar U_{\frac{1}{2d}}(0)$ such that 
$\left| g_r(z_0) \right| > 1$. Then there is a $z_1$ and an $\epsilon > 0$ and a $\delta > 0$ and a real $c > 0$ such that 
$U_{\epsilon}(z_1) \subset \left( D  \cap  U_{\frac{1}{2d}}(0) \right)$ and $\forall z \in U_{\epsilon}(z_1):$  $ \left| g_r(z) \right| > 1 + \delta \wedge \left| a_b(Y,z) \right| > c$.  But this means that the right hand side of  \eqref{eq_A4_2} also converges absolutely for $2N \rightarrow \infty$ but the values are unbounded in $k$ for $z \in U_{\epsilon}(z_1)$. So by contradicition 
\begin{equation}\label{eq_A4_31}
\left| 1 - \frac{\Delta_{r-1}(Y',z)}{(1+ h(0,d,z))\Delta_r(Y,z)} \right| \leq 1
\end{equation}
for any $z \in  (D \cap \bar U_{\frac{1}{2d}}(0)) \setminus P_{r,Y}$.  \\
Let us now assume that for a given $r$ and $Y$ there is a $z_2 \in  D \cap \bar U_{\frac{1}{2d}}(0)$ such that $\Delta_r(Y,z_2) = 0$. Then from equation \eqref{eq_A4_31}, as $P_{r,Y} $ is discrete we must have $\Delta_{r-1}(Y',z_2) = 0$ and by regress in $r$ (as $\Delta_1 \equiv 1$) we have $ P_{r,Y} \cap \bar U_{\frac{1}{2d}}(0) = \emptyset$ and therefore the Lemma is true.
\end{proof} 
As we have seen again in this important Lemma about analyticity the argument from geometry is crucial. 
\begin{corollary}\label{co_A3}
From equation \eqref{eq_A4_1} we immediately see that for $z \in D \cap \bar U_{\frac{1}{2d}}(0)$ 
\begin{equation}\label{eq_big}
\left| \Delta_r(Y,z) \right| \geq \left| \frac{\Delta_{r-1}(Y',z)}{2 \cdot (1 + h(0,d,z))} \right|
\end{equation}
and from $\Delta_1 \equiv 1$ by induction 
\begin{equation}\label{eq_A4_3}
\left| \Delta_r(Y,z) \right| \geq \left| \frac{1}{2 \cdot (1 + h(0,d,z))}\right|^{r-1}
\end{equation}
From Lemma \ref{A2} b) and equation \eqref{eq_A3_even} we know that for $d \geq 3$
the limit 
\begin{equation}\label{eq_A4_d}
d_r(Y) := \lim_{z \rightarrow \pm \frac{1}{2d}} \Delta_r(Y,z)
\end{equation}
exists and is a real number. As for  $z \in \R$ the function $\Delta_r(Y,z) $ is realvalued, continuous and nonzero and $\Delta_r(Y,0) = 1$,  $d_r(Y)$ must be a strictly positive number and so from equation \eqref{eq_A4_3}
\begin{equation}\label{eq_A4_s}
d_r(Y) \geq \frac{1}{(2\cdot(1 + G_0(0,d)))^{r-1}}
\end{equation}
\end{corollary}

\begin{theorem}\label{Th_delta}
For $h > 0$ and $R> 0$ we define the compact set 
\begin{equation}
M_{R,h} := \{ z \in \C; \left| z \right|\leq  R \wedge \left| \Re(z) \right| \leq h \}
\end{equation}
Then for $\delta > 0$ and $\gamma_r < 1$ there is an $\epsilon(\gamma_r)$ (which also depends on $\delta$ of course, for readability we do not explicitely show this dependance in the sequel) such that 
$\forall z \in M_{\frac{1}{2d} + \epsilon(\gamma_r),\frac{1}{2d}-\delta}$ and $\forall Y \in J_r$
\begin{equation}
\left| \Delta_r(Y,z) \right| \geq \frac{1}{(2 \cdot (1 + \left|  h(0,d,z) \right|))^{r-1} }\cdot \gamma_r
\end{equation}
\end{theorem}
\begin{proof}  If the Theorem is untrue for a given $r$ and $\gamma_r < 1$ then there is a sequence $z_n \in D \setminus{P_0}$ and $^{(n)} Y \in J_r$ such that $z_n \rightarrow w_0 \wedge w_0 \in M_{\frac{1}{2d},\frac{1}{2d} - \delta}$ and therefore $\left|\Re(z_n) \right| < 1/2d - \delta/2 $  for sufficiently big $n$ such that 
\begin{equation}\label{eq_A5_1}
\left| \Delta_r(^{(n)}Y,z_n) \right| < \frac{1}{(2 \cdot (1 + \left| h(0,d,z_n) \right| ))^{r-1}} \cdot \gamma_r
\end{equation}
If $^{(n)}Y$ is bounded then \eqref{eq_A5_1} contradicts \eqref{eq_A4_3} as $\Delta_r(Y,z)$ is holomorphic for $z \in D \setminus P_0$. If $^{(n)}Y$ is unbounded then we can choose an infinite subseries $n_q$ and two index sets $A_1 \dot \cup A_2 = {1,\ldots,r}$ such that $\forall k \in A_1 \wedge j \in A_2$
\begin{equation} \label{A5e3}
\left| ^{(n_q)}Y_k - ^{(n_q)}Y_j \right| \rightarrow \infty 
\end{equation}
We use $r_h(d)$ from Definition \ref{d_rh} to set
\begin{equation}\label{eq_rh}
\left| z_{n_1} \right| < r_h(d)
\end{equation}
and further claim
\begin{equation}
\left| z_{n_q} \right| < \left| z_{n_r} \right| \Leftarrow q > r
\end{equation}
In this case, using Lemma \ref{A2} we define
\begin{equation}
\lambda_0 := \min_{z \in M_{\left| z_{n_1} \right|,\frac{1}{2d}-\frac{\delta}{2}}} \lambda(z)
\end{equation}
\begin{equation}
C_0 := \max_{z \in M_{\left| z_{n_1} \right|,\frac{1}{2d}-\frac{\delta}{2}}} C(z)
\end{equation}
\begin{equation}
b_0 := \min_{z \in M_{\left| z_{n_1} \right|,\frac{1}{2d}-\frac{\delta}{2}}} \left|1 + h(0,d,z) \right|
\end{equation}
From equation \eqref{eq_rh} we immediately see that $b_0 > 0$. We define 
\begin{equation}
u := \frac{2}{2d\cdot\delta\cdot b_0}
\end{equation}
and note with equation \eqref{A2e7} that 
$\left| U_{i,j}(Y,z_{n_q}) \right| \leq u$ for sufficiently big $n_q$ and therefore
\begin{multline} \label{A6e2}
\left| \Delta_r(^{(n_q)}Y,z_{n_q} \right| \geq \left| \Delta_{r_1}(^{(n_q))}Y_{A_1},z_{n_q}) \right| \cdot \left|  \Delta_{r_2}(^{(n_q)}Y_{A_2},z_{n_q}) \right| \\
- r! \cdot \frac{C_0\cdot \exp\left(-\lambda_0 \cdot \min_{k \in A_1; j \in A_2}\left| ^{(n_q)}Y_k - ^{(n_q)}Y_j\right|\right)}{b_0} \cdot u^{r-1}
\end{multline}
with $r_1 := \#A_1$ and $r_2 := \#A_2$ and $Y_{A_i}$ is the projection of $Y$ onto the corresponding space. Because of $\Delta_1(Y,z) \equiv 1$ the Theorem is true for $r=1$. For $r>1$ it is true by induction in $r$:  We use the Theorem for $\gamma_{r_1} = \gamma_{r_2} :=\gamma_r^{1/3}$ and $\delta_1 = \delta_2 = \delta/2$ with a corresponding $\epsilon_1, \epsilon_2$ for $r_1$ and respectively for $r_2$. For sufficiently big $q$  all $z_{n_q} \in M_{\frac{1}{2d} + min(\epsilon_1,\epsilon_2),\frac{1}{2d}-\frac{\delta}{2}}$. Because of equation \eqref{A5e3} equation    \eqref{A6e2} then violates equation \eqref{eq_A5_1} for sufficiently big $n_q$.  But from this the Theorem follows.
\end{proof}
So we now know that $\Delta_r(Y,z)$ has no zero points in an open neighborhood $U$ of $\bar U_{\frac{1}{2d}}(0) \cap D$ which is independant of $Y$. So what we need to study is the vicinity of $\pm 1/2d$.
\subsection{Properties around $ z = \pm \frac{1}{2d}$}
\begin{lemma}\label{PD_l}
For $x \in \left] -\frac{1}{2d},\frac{1}{2d} \right[$ and $Y \in J_r$ the matrix
\begin{equation}
1_{r \times r} + U(Y,x)
\end{equation}
is positive definite. For $ \lambda_r(Y,x) := 1- \Delta_r(Y,x)$ the following equation is true for $ x \neq 0$:
\begin{equation}
0 \leq \lambda_{r-1}(Y',x)< \lambda_r(Y,x) < 1 
\end{equation}
\end{lemma}
\begin{proof} It is clear that $U_{i,j}(Y,0) = 0$ and therefore $\Delta_r(Y,0) = 1$ and because of Lemma \ref{A4} and continuity of $\Delta_r$ in $x$
\begin {equation}
\Delta_r(Y,x) > 0 \Leftarrow x \in \left] -\frac{1}{2d},\frac{1}{2d} \right[
\end{equation}
Therefore the matrix $1_{r \times r} + U(Y,x)$ is a real symmetric matrix with positive determinant. It's upper left quadratic $(r-1)\times(r-1)$ submatrix is $1_{(r-1)\times(r-1)} + U(Y',x)$ and therefore is also real symmetric and has positive determinant and so by induction in $r$ all upper left submatrices share this feature. But therefore they are all positive definite. Using the Cholesky decomposition for the first $(r-1)\times(r-1)$ submatrix 
\begin{equation}
1_{(r-1)\times(r-1)} + U(Y',x) = L^tL
\end{equation}
we find 
\begin{equation}
\Delta_r(Y,x) = \Delta_{r-1}(Y',x)  \cdot \left(1- \left< L^{-1}b| L^{-1}b\right> \right) 
\end{equation}
with the vector  $b_i := U_{i,r}(Y,x)$ for $i=1,\ldots,r-1$. But we therefore know 
\begin{equation}
\Delta_{r-1}(Y',x) - \Delta_r(Y,x) > 0 \Leftarrow x \neq 0 
\end{equation}
and therefore because of $\Delta_1(0,z)  \equiv 1$ by induction in $r$
\begin{equation}
1- \Delta_r(Y,x) > 0 \Leftarrow x \neq 0 \wedge r > 1
\end{equation}
and therefore the Lemma is true. 
\end{proof}
\begin{lemma}\label{le_d3}
For $d \geq 3$ and any integer $f \geq 1$ there is an $\epsilon(f) > 0$ such that $\forall z \in U_{\frac{1}{2d}}(0) \cap U_{\epsilon(f)}(\pm \frac{1}{2d})$  and $\forall Y \in J_r$ the equation 
\begin{equation}\label{eq_delta_b}
\left| 1 - \Delta_r(Y,z)^f \right| < 1 - \left(\frac{1}{2^r\cdot(1 + G_0(0,d))^{r-1}} \right)^f
\end{equation}
is true.
\end{lemma}
\begin{proof} From Lemma \ref{le_deven} we know that we can limit our proof to a discussion of the vicinity of $z = 1/2d$.  To make the proof more readable we start with $f=1$. Let's assume the Lemma was not true. Then there is a sequence $^{(n)}Y \in J_r$ and a sequence $z_n \in U_{\frac{1}{2d}}(0) $ with $z_n \rightarrow \frac{1}{2d}$ for $n \rightarrow \infty$  with 
\begin{equation}\label{eq_A9}
\left| \Delta(^{(n)}Y,z_n) -1 \right| \geq  1 - \frac{1}{2^r\cdot(1 + G_0(0,d))^{r-1}}
\end{equation} 
Let $A := \{ 1,\ldots,r \}$ and $C_0 := A$. For an infinite sequence (m) of different integers we define the set 
\begin{equation}
^{(m)}B_{i,j} := \{ ^{(m)}Y_i - ^{(m)}Y_j \}
\end{equation}
We start with $i_0:=1$ and $(m_{0,0})$ the series of all integers and define a function
$f_0: \{1,\ldots,r\} \mapsto \{0,1\}$ inductively in the following way: Iff 
$ \#\left(^{(m_{0,i-1})}B_{i_0,i} \right) < \infty$ then $f_0(i) := 1$ otherwise $f_0(i) := 0$. If $f_0(i) = 1$ then we can find an infinite subseries $(m_{0,i})$ of $(m_{0,i-1})$ such that $\#\left( ^{(m_{0,i})}B_{i_0,i} \right) = 1$. Iff $f_0(i) = 0$ then we can find an infinite subseries $(m_{0,i})$ such that $\left| ^{(m_{0,i})}Y_{i_0} - ^{(m_{0,i})}Y_i \right| \rightarrow \infty$ \\
We define $(m_{1,0}) := (m_{0,r})$ and $A_1 := f_0^{-1}(1)$ and because of $f_0(1) = 1$ we know that $A_1 \neq \emptyset$. We also define $C_1 := C_0 \setminus A_1$ and $i_1 := min(i: i \in C_1)$. \\
Inductively coming from the $k-1$ th step we define the entities of the $k$ th step: 
$(m_{k,0}) := (m_{k-1,max_k})$ where $max_k$ is the biggest integer in $C_{k-1}$. $A_k := f_{k-1}^{-1}(1)$ and because of $f_{k-1}(i_{k-1}) = 1$ we know that $A_k \neq \emptyset$. We also define $C_k := C_{k-1} \setminus A_k$ and $i_k := min(i: i \in C_k)$.
In the $k$ th step we now can write the set $C_k = \{j_1,\ldots,j_m\}$ where $j_l < j_{l+1}$. We define a function $f_k : C_k \mapsto \{0,1\}$ again inductively: $j_0 :=0$. Iff $\# \left( ^{(m_{k,j_{l-1}})}B_{i_k,j_l} \right) < \infty $ then $f_k(j_l) := 1$ and we choose $(m_{k,j_l})$ to be an infinite subsequence of $(m_{k,j_{l-1}})$ such that  
 $\# \left( ^{(m_{k,j_{l}})}B_{i_k,j_l} \right) = 1 $ Otherwise $f_k(j_l) := 0$ and we choose  $(m_{k,j_l})$ such that 
 $\left| ^{(m_{k,j_l})}Y_{i_k} - ^{(m_{k,j_l})}Y_{j_l} \right| \rightarrow \infty$ 
\\
As $A_k \neq \emptyset$ the induction leads to a disjoint union of sets $A = A_1 \dot \cup \ldots  \dot \cup A_p$ and an infinite sequence of different integers $(m_{p,0})$ such that 
for any $q \in \{1,\ldots,p\}$ and for any $i,j \in A_q$ we have $\#\{^{(m_{p,0})}Y_i - 
^{(m_{p,0})}Y_j  \} = 1$ and for $q \neq \hat q$ and $i \in A_q, j \in A_{\hat q}$ 
we have  $\left| ^{(m_{p,0})}Y_{i} - ^{(m_{p,0})}Y_{j} \right| \rightarrow \infty$ 
\\
By a permutation of the indices in $A$, such that the sets $A_q$ consist of adjacent integers we see that the matrices for $\Delta(^{(m_{p,0})}Y,z_{(m_{p,0})})$ after this permutation is a block matrix with constant vectors $Y_{A_q} \in J_{\#A_q}$ whose determinants according to equation  \eqref{eq_A4_d} converge towards the real positive numbers 
$d_{\#A_q}(Y_{A_q})$ for $(m_{p,0}) \rightarrow \infty$. The matrix elements outside of the diagonal block matrices however vanish according to Lemma \ref{B3} for 
 $(m_{p,0}) \rightarrow \infty$. But this means that for any $\delta > 0$ there exists an 
$N \in \N$ such that for infinitely many of the $m_{p,0} > N$ the equation 
\begin{equation}\label{eq_subdet}
\left| \Delta(^{(m_{p,0})}Y,z_{(m_{p,0})}) - \prod_{q=1}^p d_{\#A_q}(Y_{A_q}) \right| < \delta
\end{equation}
is true. 
Together with equation \eqref{eq_A4_s} and the triangle inequality this means that \eqref{eq_A9} is violated which means that the Lemma is true for $f=1$. For $f \geq 1$ the reasoning is equivalent: Any sequence $^{(n)}Y \in J_r$ and $z_n$ which violate equation \eqref{eq_delta_b}
for a fixed $f$ contain a subsequence where either $^{(n)}Y$ is bounded or which fulfills equation \eqref{eq_subdet} and so the Lemma follows again from equation \eqref{eq_A4_s}.
\end{proof}
\section{Proving $\left|\Delta_r(Y,z) -1 \right| \leq 1$ for $d=2$}
Lemma \ref{le_d3} is fundamental for the expansion of $1/\Delta^f$ in a geometrical series in $1-\Delta^f$  in the following sections. It is seemingly quite difficult to prove in $d=2$. We will do so in this section. But we hope the proof can be simplified and related to a geometric interpretation of the Taylor coefficients of $1-\Delta$ showing them to all be positive (at least for well behaved geometrical entities like $\Z^d$), also sheding light on the very nature of $1 - \Delta$.\\
\subsection{Sequences and subsequences in $\Z^d$}
We start with features of sequences in $\Z^d$.
\begin{lemma}\label{F1}
Let $x_n \in \Z^d$ be an infinite sequence. Then either $x_n$ has an infinite constant subsequence $x_{\nu_m} \equiv  x_c$ or there is an infinite sequence $(\mu_m)$ with 
\begin{equation}\label{F11}
j > i \Rightarrow \left|x_{\mu_j}\right| > \left|x_{\mu_i}\right|
\end{equation}
and therefore 
\begin{equation}\label{F12}
\left|x_{\mu_j} \right| \rightarrow \infty
\end{equation}   
\end{lemma}
\begin{proof}
If $\{x_n\}$ is a finite set then the first alternative must be true. If $\{x_n\}$ is an infinite set then choose $\mu_1 = 1$. For $i>1$ choose $\mu_i$ such that $\left|x_{\mu_i} \right|> \left|x_{\mu_{i-1}}\right|   $. If that is not possible then for $m > \mu_{i-1} \Rightarrow \left|x_m\right| \leq \left| x_{\mu_{i-1}}\right|$ which means that $\{x_n\}$ is a finite set which contradicts the assumption. Therefore the Lemma is true.
\end{proof}  
\begin{lemma}\label{F2}
Let $x_n \in \Z^d\setminus\{0\}$ and $y_n \in \Z^d\setminus\{0\}$ be infinite sequences.
Then there is an infinite sequence of integers $(\mu_n)$ such that either 
$\left|x_{\mu_i}\right| / \left|y_{\mu_i}\right|$ is a convergent series or 
$\left|y_{\mu_i}\right| / \left|x_{\mu_i}\right|$ is a convergent series.
\end{lemma}
\begin{proof}
If $\left|x_{n}\right| / \left|y_{n}\right|$ 
is bounded then it contains a convergent subseries. Let therefore 
$\left|x_{n}\right|/\left|y_{n}\right|$ be unbounded. Then it contains a monotonously increasing subseries which is unbounded, with the index sequence $(\mu_j)$.
 But in this case  $\left|y_{\mu_i}\right| / \left|x_{\mu_i}\right|$ is a monotonously decreasing series converging towards $0$.
\end{proof}

\begin{lemma}\label{F3}
Let $x^{(1)}_n,\ldots,x^{(r)}_n$ with $x^{(i)}_n \in \Z^d\setminus\{0\}$ be r sequences. Then there is an infinite series of natural numbers $(\mu_j)$ and a permutation $\sigma \in S_r$  and real numbers $q_i \in \left[0,\infty\right[ $such that 
\begin{equation}\label{F3e1}
\frac{\left|x^{(\sigma(i))}_{\mu_j}\right|}{\left|x^{(\sigma(i+1))}_{\mu_j}\right|} \rightarrow q_i
\end{equation}
\end{lemma}
\begin{proof} We prove this by induction to $r$. For $r=2$ see Lemma \ref{F2}. Let therefore the Lemma \ref{F3} be true for a given $r$ and let us have sequences  $x^{(1)}_n,\ldots,x^{(r+1)}_n$. Then there exists a series of natural number $(\mu_j)$ and a permutation $\sigma \in S_r$ and numbers $q_i \in \left[0,\infty\right[ $ such that \ref{F3e1} is true.
Now according to Lemma \ref{F2} for each $i$ there is an infinite subsequence $(\mu^{(i)}_j)$ of $(\mu_j)$ such that there is a $\hat{ q_i} \in \left[0,\infty\right[ $ with either 
\begin{equation}\label{F3e2}
\frac{\left|x^{(\sigma(i))}_{\mu^{(i)}_j}\right|}{\left|x^{(r+1)}_{\mu^{(i)}_j}\right|} \rightarrow \hat {q_i}
\end{equation}
or
\begin{equation}\label{F3e3}
\frac{\left|x^{(r+1)}_{\mu^{(i)}_j}\right|}{\left|x^{(\sigma(i))}_{\mu^{(i)}_j}\right|} \rightarrow \hat{q_i}
\end{equation}
Let us now assume that there is a minimal $i_0$,  $1\leq i_0 \leq r$, such that for every $i<i_0$ equation 
\ref{F3e3} is untrue for any infinite subsequence of $(\mu_j)$ but it is true for a given subsequence $(\mu^{(i_0)}_j)$ and the index $i = i_0$. In this case we can define $\hat{\sigma} \in S_{r+1}$ with
$\hat{\sigma}(i) = \sigma(i) \Leftarrow i=1,\ldots,i_0-1$ and $\hat{\sigma}(i_0) = r+1$ and 
$\hat{\sigma}(i) = \sigma(i-1) \Leftarrow i=i_0+1,\ldots,r+1$ such that the Lemma is true for the infinite subsequence $(\mu^{(i_0)}_j)$. If the assumption of the existence of a minimal $i_0$ is not true then according to Lemma \ref{F2} there must be an infinite subsequence $(\mu^{(r+1)}_j)$ of $(\mu_j)$ such that equation \eqref{F3e2} is true for every $1 \leq i \leq r$. Now for this infinite subsequence and the permutation $\hat{\sigma} \in S_ {r+1}$ with $\hat{\sigma}(i) = \sigma(i) \Leftarrow i=1,\ldots,r$ and $\hat{\sigma}(r+1) = r+1$ Lemma \ref{F3} is true which completes the proof. 
\end{proof}
\subsection{The distance space}
We now define an algrebraic entity flowing naturally from analyzing $1 - \Delta$ in two dimensions as we will see in the sequel. 
\begin{definition}
Let $S$ be a set and  $U(S)$ the set of unordered pairs $(a,b)$ with $a, b \in S$ and $a \neq b$. $S$ is called a distance space if there are relations $\leq, \sim$ on $U(S)$ such that:
\begin{description}
\item[poset] $(U(S), \leq)$ is a poset .
\item[completeness] for $(a,b), (c,d) \in U(S)$ one and only one of the relations $(a,b) < (c,d)$, $(a,b) > (c,d)$ or $(a,b) \sim (c,d)$ is true (where as usual $t<u$ is a shorthand for  $t\leq u \wedge t\neq u$).
\item[transitivity] for $(a_1,b_1), (a_2,b_2), (c,d) \in U(S)$ and $(a_1,b_1) \sim (a_2,b_2)$:
\begin{itemize}
\item $(a_1,b_1) > (c,d) \Rightarrow (a_2,b_2) > (c,d)$
\item $(a_1,b_1) < (c,d) \Rightarrow (a_2,b_2) < (c,d)$
\item $(a_1,b_1) \sim (c,d) \Rightarrow (a_2,b_2) \sim (c,d)$
\end{itemize}
\item[triangle] for $(a,b), (b,c) \in U(S)$: if $(a,b) > (b,c) \Rightarrow (a,b) \sim (a,c)$ 
\end{description}
\end{definition}
\emph{Note:} Because of completeness $\sim$ is reflexive, because of completeness and transitivity it is symmetric. 

\begin{lemma}\label{DS_sim}
Let $S$ be a finite distance space. We define a relation $\sim$ on $S$ by
$a \sim b$ if either $a=b$ or $(a,b)$ is minimal in $(U(S), \leq)$. Then $\sim$ is an equivalence relation, any two minimal elements in $(U(S),\leq)$ are equivalent in $U(S)$ and any element in $U(S)$ equivalent to a minimal element is itself minimal. The set of equivalence classes $S_1:=[S]$ is a distance space with the following relations on $U([S])$: For $([a],[b]), ([c],[d]) \in U([S])$ we define 
\begin{equation}\label{GS1}
([a],[b]) < ([c],[d]) \Leftarrow (a,b) < (c,d)
\end{equation}
\begin{equation}\label{GS2}
([a],[b]) \sim ([c],[d]) \Leftarrow (a,b) \sim (c,d)
\end{equation}
\end{lemma}

\begin{proof} It is clear from the definition of $\sim$ on $S$ that $\sim$ is reflexive and symmetric. Let $(a,b), (c,d) \in U(S)$ be two minimal elements. Then, as they are minimal because of completeness we must have $(a,b) \sim (c,d)$. By the same token any pair $(e,f) \sim (a,b)$  with $(a,b)$ minimal must be minimal too, otherwise because of transitivity $(a,b)$ was not minimal. But that means that $\sim$ on $S$ is also transitive. \\
For $([a],[b]) \in U([S])$ we already know that $[a] \neq [b]$. Let us assume that $a_1, a_2 \in [a]$ and $b_1, b_2 \in [b]$. Then $a_1 = a_2$ or $(a_1,a_2) < (a_1,b_1)$ and $b_1 = b_2$ or $(a_2,b_2) > (b_1,b_2)$. But then because of triangle $(a_1,b_1) \sim (a_2,b_2)$ and therefore by transitivity the definitions \ref{GS1} and \ref{GS2} are independant of the representative in the equivalent classes. But that immediately means that [S] is a distance space by inheritance. 
\end{proof}

\begin{corollary} For a finite distance space $S$ as $(U(S),\leq)$ is a finite poset there are always minimal elements if $\#S > 1$. Therefore there is a sequence $S_0:=S; S_{i+1} := [S_i], i=0,...,h_S$ of distance spaces with $\#S_i > \#S_{i+1}$  and $\#S_{h_S} = 1$. In the sequel for $a \in S$ we denote the corresponding class $[a]_i \in S_i$ and $[a]_0 := a$. By construction we also know that for $h_S \geq i > j \geq 0$ and elements $[a]_i , [b]_i \in S_i$ with $[a]_i \neq [b]_i$ and   $[e]_j , [f]_j \in S_j$ with $[e]_j \neq [f]_j$ and $[e]_j \sim [f]_j$ that $(a,b) > (e,f)$. 
\end{corollary}

\begin{definition}
For a finite distance space $S$ we define a  directed graph $ G_S= (V,E,J)$ by 
\begin{equation}\label{GS3}
V:=\{[a]_i, \#[a]_i > \#[a]_{i-1} \vee i=0\}
\end{equation}
\begin{equation}\label{GS4}
E := \{([a]_i,[a]_j); a\in S, [a]_i \in V, [a]_j \in V \wedge [a]_k \notin V \forall k=i+1,...,j-1\}
\end{equation}
\begin{equation}\label{GS5}
J(([a]_i,[a]_j))_1 = [a]_i \wedge J(([a]_i,[a]_j)_2 = [a]_j
\end{equation}
\end{definition}
\begin{lemma}
For any finite distance space $S$ $G_S$ is a tree. 
\end{lemma}
\begin{proof} Let $[a]_i \in V$. If $0 \leq i < h_S$ then there is a smallest $i<j_0 \leq h_S$ such that $[a]_{j_0} \in V$ and therefore there is one and only one edge $([a]_i,[a]_{j_0})$ starting in a given vertex $[a]_i$ in $V$ if $i<h_S$. By definition of $h_S$ there is exactly one vertex $[a]_{h_S}$ as $\#S_{h_S} = 1$. Any vertex is connected to this one by finitely many edges. So we have shown that $G_S$ is connected and $\#E = \#V-1$ and therefore $G_S$ is a tree. 
\end{proof}
\begin{definition}\label{def_v}
For a finite distance space $S$ and its tree $G_S = (V,E,J)$ and $v = [a]_i \in V$ we now define the useful entities. To make them more understandable we give rough descriptions for them, which view $G_S$ as a tree with the point $[a]_{h_S}$ at its top.
We define:  
\begin{enumerate}
\item the number $i$ is defined as the order $o(v)$ of $v$ .
\item the set $V_0 := \{v \in V; o(v) = 0 \}$ (the base).
\item the set $S(v) := \{b \in S; [b]_{o(v)} = v\}$ (the points of the base under a given vertex $v$).
\item the set $V(v) := \{[b]_j; 0 \leq j < o(v) \wedge  b \in S \wedge [b]_{o(v)} = v \wedge [b]_j \in V \}$ (the vertices connected to $v$ from below) 
\item for a vertex $\hat{v} \in V $ the set 
\begin{equation}
\left]v,\hat{v}\right[ := \begin{cases}
\{[a]_k; j>k>i \wedge [a]_k \in V \} \Longleftarrow \exists j: \hat{v} = [a]_j \\ 
\emptyset \Longleftarrow \; otherwise
\end{cases}
\end{equation}
(the vertices between $v$ and $\hat v$ exluding $v$ and $\hat v$).
\item for $ o(v)>  0$ the set $I(v) := \{\overline{v} \in V; (\overline{v},v) \in E\}$ (the vertices directly under the vertex $v$).
\item for $ o(v)> 0$ the class \\$D(v) := \{([c]_{o(v)-1},[d]_{o(v)-1}) \in  U(S_{o(v)-1}) \, \textrm{minimal}\}$. 
\item the set $V_t := V\setminus\{[a]_{h_S}\}$ (all vertices other than the top) 
\item for $v \in V_t$ we can uniquely define $\partial v :=  \hat{v} \Leftrightarrow (v,\hat{v}) \in E$ (the vertex directly above $v$).
\end{enumerate}
\end{definition} 

\begin{lemma}\label{GS_l}
Let $S$ be a finite distance space and $a,b \in S \wedge a \neq b$. There is exactly one vertex $v$ in the tree $G_S$ such that  $([a]_{o(v)-1},[b]_{o(v)-1}) \in D(v)$. We define  
$T(v) := \{(a,b) \in U(S), ([a]_{o(v)-1},[b]_{o(v)-1}) \in D(v) \}$
\end{lemma}

\begin{proof} There is a minimal number $i \in \{1,\ldots, h_S\}$ such that $[a]_i = [b]_i$ as $[a]_{h_S} = [b]_{h_S}$. As it is minimal $[a]_{i-1} \neq [b]_{i-1}$ and therefore $[a]_i \in V$. By definition of $\sim$ on $S_{i-1}$ the pair $([a]_{i-1},[b]_{i-1}) \in U(S_{i-1})$ must be minimal. Let us assume that there is a $c \in S$ and a $j \in \{1,\ldots,h_S\}$ such that $[c]_j$ is a vertex $\overline{v}$ and $([a]_{j-1},[b]_{j-1}) \in D(\overline{v})$. But then by the definition of $\sim$ on $S_{j-1}$ we have $[a]_j = [b]_j = [c]_j$ and $[a]_{j-1} \neq [b]_{j-1}$. As $i$ was minimal with the feature $[a]_i = [b]_i$ we have $i \leq j$, but we also see $i \geq j$ because $[a]_{j-1} \neq  [b]_{j-1}$ and therefore $i=j$ and $\overline{v} = [a]_i$.  
\end{proof}

\begin{lemma}\label{GS_l2}
Let $S$ be a finite distance space and $a,b \in S \wedge a \neq b$.  Let $v$ be a vertex. Then either
\begin{itemize} 
\item $a, b \in S(v)$ and $(a,b) \in T(v)$ or
\item  $a, b \in S(v)$ and there is a vertex $\overline{v} \in I(v)$ such that $a,b \in S(\overline{v})$ and therefore $(a,b) \notin T(v)$ or
\item $a \notin S(v) \vee b \notin S(v)$.
\end{itemize}
\end{lemma}
\begin{proof} If $a,b \in S(v)$ we know $[a]_{o(v)} = [b]_{o(v)}$. If we have $[a]_{o(v)-1} \neq [b]_{o(v)-1}$ then $(a,b) \in T(v)$. Otherwise  $[a]_{o(v)-1} = [b]_{o(v)-1}$ and there is a minimal $k \in \{0,\ldots, o(v) - 1\}$ such that $\#[a]_k = \#[a]_{o(v)-1}$. But then for this $k$ $[a]_k = [b]_k$ is a vertex and by construction $[a]_k \in I(v)$ which proves the Lemma.
\end{proof}
\subsection{Applying the concept of the distance space to se\-quen\-ces on $\Z^d$}
\begin{definition}
Let $^{(n)}Y \in J_r$ be a sequence of vectors. We define 
\begin{equation}\label{G0e1}
^{(n)}h_{i,j} := ^{(n)}Y_i - ^{(n)}Y_j 
\end{equation} 
We call the series  $^{(n)}Y$ a comparative sequence
iff for any two distinct index pairs $(i_0,j_0)$ and $(i_1,j_1)$ 
there is either\\ a $q(i_0,i_1,j_0,j_1) \in \left[0,\infty\right[$ such that
\begin{equation}\label{eq_buraq}
\frac{\left|^{(n)}h_{i_0,j_0}\right|}{\left|^{(n)}h_{i_1,j_1}\right|} \rightarrow q(i_0,i_1,j_0,j_1) 
\end{equation} 
or there is a  $\tilde{q}(i_0,i_1,j_0,j_1)  \in \left[0,\infty\right[$ such that
\begin{equation}\label{eq_kcq}
\frac{\left|^{(n)}h_{i_1,j_1}\right|}{\left|^{(n)}h_{i_0,j_0}\right|} \rightarrow \tilde{q}(i_0,i_1,j_0,j_1) 
\end{equation} 
 
\end{definition}
\emph{Note:} According to Lemma \ref{F3} any sequence  $^{(n)}Y \in J_r$ contains a comparative subsequence. 

\begin{lemma}\label{Q1}
Let $^{(n)}Y \in J_r$ be a comparative sequence of vectors. Let $S:=\{0,\ldots,r\}$ be their index set. Then $^{(n)}Y$ induces a distance space structure on $S$ in a natural way by the following definitions for $(i_0,j_0), (i_1,j_1) \in U(S)$. 
\begin{itemize}
\item $(i_0,j_0) < (i_1,j_1)$ iff equation \eqref{eq_buraq} is true and $q(i_0,i_1,j_0,j_1)=0$
\item  $(i_0,j_0) \sim (i_1,j_1)$ iff equation \eqref{eq_buraq} is true and $q(i_0,i_1,j_0,j_1) \neq 0$
\item $(i_0,j_0) > (i_1,j_1)$ iff equation \eqref{eq_kcq} is true and $\tilde{q}(i_0,i_1,j_0,j_1)=0$
\end{itemize} 
\end{lemma}
\begin{proof} 
With the above definition $(U(S), \leq)$ is a poset. It is immediately clear that between any two index pairs $(i_0,j_0)$ and $(i_1,j_1)$ one and only one of the three relationships $\sim, <, >$ is valid and therefore we have completeness. Transitivity of the distance space follows directly from the multiplicativity of limits in real analysis. Triangle follows from the triangle inequality in $\Z^d$. So $S$ is a distance space.
\end{proof} 
\begin{definition}
Let $z_n \in U_{\frac{1}{2d}}(0)\setminus\{0\}$  a series with $z_n \rightarrow \frac{1}{2d}$ and $s_n := \frac{1}{z_n} - 2d$ and  $^{(n)}Y \in J_r$  a sequence of vectors. We call  $^{(n)}Y$ comparative to $z_n$ if it is comparative and for any $^{(n)}h_{i,j}$ one of the following alternatives is true:
\begin{enumerate}
\item $^{(n)}h_{i,j} \equiv h_{i,j}$ is constant
\item or $\sqrt{s_{n}} \cdot \left|^{(n)}h_{i,j}\right| \rightarrow 0$ and $ \left|^{(n)}h_{i,j}\right| $ is a strictly growing unbounded series 
\item or $\sqrt{s_{n}} \cdot \left|^{(n)}h_{i,j}\right| \rightarrow w \in \C\setminus\{0\}$ 
\item or  $\left|\sqrt{s_{n}} \right|\cdot \left|^{(n)}h_{i,j}\right|$ is a strictly growing unbounded series
\end{enumerate}    
\end{definition}

\begin{lemma}\label{le_comp}
For any  $z_n \in U_{\frac{1}{2d}}(0)\setminus\{0\}$  with $z_n \rightarrow \frac{1}{2d}$ and any infinite sequence $^{(n)}Y \in J_r$ there is a subsequence of $^{(n)}Y$ such that it is comparative to $z_n$.
\end{lemma}
\begin{proof} As noted before $^{(n)}Y$ contains a comparative subsequence. Comparativity is retained even with the same limits in equations  \eqref{eq_buraq} and  \eqref{eq_kcq} by any subsequence of a comparative subsequence. Let us start with a comparative subsequence and for each of the finitely many $(i,j)$ let us subsequentally refine the subsequence of the previous step in the following way (we denote the subsequence of the index sequence $n\in \N$ by $\nu_l$ it potentially changes with every step $l$): \\
If  $\left|\sqrt{s_{\nu_l}} \right|\cdot \left|^{(\nu_l)}h_{i,j}\right|$ is unbound then let us choose 
$\nu_{l+1}$ such as to make it strictly growing. This feature will also not change with a further refinement. \\
If $\sqrt{s_{\nu_l}} \cdot \left|^{(\nu_l)}h_{i,j}\right|$ on the other hand is bounded it contains a convergent subseries with a limit point $w \in \C$. If $w \neq 0$  we choose this subseries as the $\nu_{l+1}$ for the next step. Again convergence and the limit point are not touched by further refinements\\
If $w = 0$ and $  \left|^{(\nu_l)}h_{i,j}\right|$ is unbounded we choose $\nu_{l+1}$ to be a convergent subseries but also to make $  \left|^{(\nu_{l+1})}h_{i,j}\right|$  strictly growing. \\
If $w =0$ and $  \left|^{(\nu_l)}h_{i,j}\right|$ is bounded then we choose  $\nu_{l+1}$ to be a convergent subseries but also to make $  ^{(\nu_{l+1})}h_{i,j} \equiv h_{(i,j)} $ constant. Again these features do not change with refinement and so the Lemma is proven. 
\end{proof}
\begin{definition}
Let  $z_n \in U_{\frac{1}{2d}}(0)\setminus\{0\}$  with $z_n \rightarrow \frac{1}{2d}$  be a sequence and  $^{(n)}Y$ such that it is comparative to $z_n$. Let $S := \{0,\ldots,r\}$ be the index set with the distance space structure induced by $^{(n)}Y$ and $G_S = (V,E,J)$ it's tree. 
For $v \in V \setminus V_0$ we define:
\begin{equation}
^{(n)}d(v) := \frac{1}{\#T(v)} \sum_{(i,j) \in T(v)}  \left|^{(n)}h_{(i,j)} \right|
\end{equation}
We define the the vector 
\begin{equation}
w_v := \sum_{i \in S(v)} e_i
\end{equation}
\end{definition}
\emph{Note:} according to Lemma \ref{DS_sim} for each $(i,j) \in T(v)$ we find 
\begin{equation}\label{eq_dv}
\frac{\left|^{(n)}h_{(i,j)}\right|}{^{(n)}d(v)}  \rightarrow q_{i,j} \in \R\setminus\{0\}
\end{equation}
\begin{definition}
With the data of the previous definition we define as subsets of $V\setminus V_0$
\begin{equation}
V_1:= \{ v \in V\setminus V_0, ^{(n)}d(v) \equiv d(v)\; \textrm{constant}\} 
\end{equation}
\begin{equation}
V_2 :=\{ v \in V\setminus V_0, \sqrt{s_{n}} \cdot^{(n)} d(v)\rightarrow 0 \wedge ^{(n)}d(v) \; \textrm{unbounded}\}
\end{equation}
\begin{equation}
V_3:=\{ v \in V \setminus V_0, \sqrt{s_{n}} \cdot ^{(n)}d(v) \rightarrow w \in \C\setminus\{0\}\}
\end{equation}
\begin{equation}
V_4 := \{v \in V \setminus V_0, \left|\sqrt{s_{n}} \right|\cdot ^{(n)}d(v) \; \textrm{unbounded}\}
\end{equation}
For $v \in V \setminus V_0$ we also define
\begin{equation}
X(v) := \{\hat v \in V(v) \cap V_0; [\hat v] \notin V_1\} \cup (V(v) \cap V_1)
\end{equation}
\end{definition}
\begin{definition}
With the data of the previous definition for vertices $v \in V_1$ we define:
the real matrix $G_v$ which is zero in the diagonal and for $(i \neq j);  i,j \in S(v) $ is 
\begin{equation}
(G_v)_{i,j} :=  G_0(h_{i,j},2) - \frac{4}{\pi}\ln(2)
\end{equation}
and is zero for $i \notin S(v) \vee j \notin S(v)$.
\end{definition}
\subsection{The asymptotic first hit matrix in light of its distance space structure}
\begin{lemma}\label{MA_l1}
Let $d=2$ and  $z_n \in U_{\frac{1}{2d}}(0)\setminus\{0\}$  with $z_n \rightarrow \frac{1}{2d}$  be a sequence and  $^{(n)}Y$ such that it is comparative to $z_n$. Let $S := \{0,\ldots,r\}$ be the index set with the distance space structure induced by $^{(n)}Y$ and $G_S = (V,E,J)$ it's tree.  Using the notation 
\begin{equation}
y_n := \frac{1}{1 + h(0,2,z_n)}
\end{equation}
the following decomposition is true:
\begin{equation}\label{eq_dec}
1_{r\times r} + U(^{(n)}Y,z_n) = 1_{r\times r} + M_0(^{(n)}Y,z_n) + y_n \cdot  M_1(^{(n)}Y,z_n)
\end{equation}
with 
\begin{multline}\label{eq_m0}
M_0(^{(n)}Y,z_n) :=  \sum_{v \in V_2} (1-\frac{2}{\pi}\cdot y_n \cdot  \ln(^{(n)}d(v)))\cdot (w_v^{t} \cdot w_v - \sum_{\overline{v} \in I(v)} w_{\overline{v}}^{t} \cdot w_{\overline{v}}) \\
+  \sum_{v \in V_1} (- y_n \cdot G_v) + w_v^{t} \cdot w_v - \sum_{\overline{v} \in I(v)} w_{\overline{v}}^{t} \cdot w_{\overline{v}}) )
\end{multline}
and $M_1$ has a zero diagonal and for $i \neq j$ converges towards a constant complex matrix $\Gamma_{i,j} \in \C$
\begin{equation}
M_1(^{(n)}Y,z_n)  \rightarrow \Gamma
\end{equation}
Specifically for $(i,j) \in T(v)$ with $v \in V_1 \cup V_4$ we have $\Gamma_{i,j} = 0$
\end{lemma}
\begin{proof} Let $i,j$ be two distinct indices. Then the matrix element on the left hand side of equation \eqref{eq_dec} is given by $y_n \cdot h(^{(n)}h_{i,j},2,z_n)$.  
According to Lemma \ref{GS_l} there is exactly one vertex $v_0 \in V$ such that $(i,j) \in T(v_0)$.  According to Lemma \ref{GS_l2} this vertex is unique in that there is no vertex $\overline{v} \in I(v)$  such that $i,j \in S(\overline{v})$. Therefore the projector sum 
\begin{equation}
P(v) :=  (w_{v}^{t} \cdot w_{v} - \sum_{\overline{v} \in I(v)} w_{\overline{v}}^{t} \cdot w_{\overline{v}})
\end{equation}
fulfils the relation
\begin{equation}  
P(v)\cdot e_i^{t} \cdot e_j = e_i^{t} \cdot e_j \cdot \delta_{v_0,v} 
\end{equation}
So if $v_0 \in V_2$ the right hand side of equation \eqref{eq_dec} is $y_n \cdot (-\frac{2}{\pi}\ln(^{(n)}d(v_0))   + \Gamma_{i,j} + \Lambda_{i,j}^{(n)})$ where $\Lambda_{i,j} ^{(n)} \rightarrow 0$ and so both sides are identical for a suitable constant $\Gamma_{i,j}$ according to Lemma \ref{B3_d2} and equation \eqref{eq_dv}.\\
If $v_0 \in V_4$ the right hand side of equation \eqref{eq_dec} is $y_n\cdot  \Lambda_{i,j}^{(n)})$ where $\Lambda_{i,j} ^{(n)} \rightarrow 0$ which is identical to the right hand side according to Lemma \ref{B3_d2}\\
If $v_0 \in V_3$ the right hand side of equation \eqref{eq_dec} is $y_n\cdot (\Gamma_{i,j} +   \Lambda_{i,j}^{(n)})$ where $\Lambda_{i,j} ^{(n)} \rightarrow 0$ which is identical to the right hand side according to Lemma \ref{B3_d2}\\
If $v_0 \in V_1$ the right hand side of equation \eqref{eq_dec} is $y_n \cdot (G_0(h_{i,j},2) - \frac{4}{\pi}\ln(2) + \Lambda_{i,j}^{(n)})$ where $\Lambda_{i,j} ^{(n)} \rightarrow 0$ which is identical to the right hand side according to Lemma \ref{B3_d2} and so the Lemma is proven. 
\end{proof}
\subsection{Calculating the Determinant with the Matrix Determinant formula and Sherman Morisson }
We now calculate the determinant and the inverse of $1 + M_0$.
We begin with the following preparation:
\begin{lemma}\label{Blambda}
Let $d= 2$ and  $z_n \in U_{\frac{1}{2d}}(0)\setminus\{0\}$  with $z_n \rightarrow \frac{1}{2d}$  be a sequence and  $^{(n)}Y \equiv Y \in J_r$ be constant. Then it is comparative to $z_n$. Let $S := \{0,\ldots,r\}$ be the index set with the distance space structure induced by $^{(n)}Y$ and $G_S = (V,E,J)$ it's tree which of course only has one nontrivial vertex $v \in V_1$. Then for the matrix
\begin{equation}\label{eq_detg}
B_v(\lambda) := -y_n\cdot G_v + \lambda \cdot w_v^{t}\cdot w_v
\end{equation}
we note that  $B(1) = 1_{r \times r} + M_0(Y,z_n)$ and arrive at
\begin{equation}\label{det_Bl}
det(B_v(\lambda) =  y_n^{r-1} \cdot p_v \cdot (\lambda - q_v \cdot y_n)
\end{equation}
where $p_v, q_v \in \R$ and $ p_v \geq \frac{1}{2^{r-1}}$ and $q_v = \frac{det(G_v)\cdot(-1)^r}{p_v}$.
\end{lemma}
\begin{proof} Using simple algebra it is immediately clear that as a polynomial in $y$ and $\lambda$ \eqref{det_Bl} is true with $p_v = (-1)^{r+1}\sum_{i,j} (-1)^{(i+j)} Adj(G_v)_{i,j}$  and $p_v\cdot q_v = det(G_v)\cdot(-1)^r$. From \eqref{eq_big} and Lemma \ref{PD_l} we see that $p_v \geq \frac{1}{2^{r-1}}$.
\end{proof}

\begin{lemma}\label{le_binv}
With the data of Lemma \ref{Blambda} and $y_n \neq 0$: If $\lambda - q_v\cdot y_n \neq 0$ and $q_v \neq 0$ 
\begin{equation}
B_v(\lambda)^{-1} = \frac{1}{y_n}\left(-G_v^{-1} + q_v\cdot \lambda\cdot \frac{u_v^{t}\cdot u_v}{\lambda - q_v\cdot y_n} \right)
\end{equation} with $u_v := G_v^{-1}\cdot w_v$ \\  
 If $\lambda - q_v\cdot y_n \neq 0$ and $q_v = 0$ 
\begin{equation}
B_v(\lambda)^{-1} = \frac{1}{y_n} \left( -H_v^{-1} +  (\lambda + y_n)\cdot \frac{\overline{u_v}^{t}\cdot \overline{u_v}}{\lambda} \right)
\end{equation} with $H_v := G_v + w_v^{t}\cdot w_v$ and $\overline{u_v} := H_v^{-1}\cdot w_v$. In both cases we have 
\begin{equation}\label{eq_v1_xi}
w_v^{t}B_v(\lambda)^{-1}w_v = \frac{1}{\lambda - q_v\cdot y_n}
\end{equation}  
\end{lemma}
\begin{proof} For $\lambda - q_v \cdot y_n \neq 0$ and $y_n \neq 0$ because of equation  
\ref{det_Bl} and $p_v \neq 0$ the matrix $B(\lambda)$ is invertible. If $q_v \neq 0$ $G_v$ is invertible. Otherwise $det(G_v + w_v^{t} \cdot w_v) = (-1)^{r-1} p_v$ which is always nonzero and so $H_v$ is invertible. The above formulas for the inverse then follow from Sherman-Morisson \citep{shmo}.
\end{proof}
\begin{lemma}\label{MA_det}
With the data of Lemma \ref{MA_l1} there is an $N\in \N$ such that for  $n \geq N$ and for all $v \in V_2$ the equations
\begin{equation}\label{eq_dv1}
^{(n)}d(v) > 2
\end{equation}
and for $v \in V_2 \cap V_t$ the equation
\begin{equation}\label{eq_dv2}
\frac{2}{\pi}\cdot \left(\ln(^{(n)}d(\partial v) - \ln(^{(n)}d(v)) \right) > 1
\end{equation}
and for $v \in V_1 \cap V_t \wedge \partial v \in V_2$ 
\begin{equation}\label{eq_dv3}
\left| \frac{q_v}{\frac{2}{\pi} \ln(^{(n)}d(\partial v))}\right| 
< \frac{1}{\#I(\partial v) + 1}
\end{equation}
and for $v \in V_1$  the equation 
\begin{equation}\label{eq_dv4}
\left| q_v \cdot y_n \right| < \frac{1}{\#S(v) + 1}
\end{equation}
are fulfilled.\\
Then for all $n \geq N$ we have the following formula for the determinant 
\begin{multline}\label{eq_det}
\det \left(1_{r \times r} + M_0\left(^{(n)}Y,z_n \right) \right) = \prod_{v \in V_1} y_n^{\#S(v) -1} p_v\cdot (1-\phi_v^{(n)}) \\
\prod_{v \in V_2} g\left(\#X(v),\phi_v^{(n)},\xi_v^{(n)}\right)
\end{multline}
with the notations:
\begin{equation}
g\left(m,x,\beta\right) := x^{(m-1)}\cdot (\beta + (1-\beta)\cdot x)
\end{equation}
and $\xi_v^{(n)}$ defined for $v \in V_2$ recursively by
\begin{equation}\label{eq_xi}
\xi_v^{(n)} := 
 \sum_{\hat{v} \in I(v) \cap V_2 }\frac{\xi_{\hat{v}}^{(n)}}{g(1,\phi_{\hat{v}}^{(n)},\xi_{\hat{v}}^{(n)})}
+ \sum_{\hat{v} \in I(v)\cap (V_1 \cup V_0)}\frac{1}{1 - \phi^{(n)}_{\hat{v}}}
\end{equation}
and the expression
\begin{equation}\label{eq_phiv}
\phi_v^{(n)} := \begin{cases} \frac{\ln(^{(n)}d(v)}{\ln(^{(n)}d(\partial v)} \Longleftarrow v \in V_2 \cap V_t \wedge \partial v \in V_2 \\
y_n\cdot\ln(^{(n)}d(v)) \Longleftarrow v \in V_2, v \notin V_t \vee \partial v \notin V_2 \\
\frac{q_v}{\frac{2}{\pi}\cdot\ln(^{(n)}d(\partial v))} \Longleftarrow v \in V_1 \cap V_t \wedge \partial v \in V_2 \\
y_n\cdot q_v \Longleftarrow v \in V_1, v \notin V_t \vee \partial v  \notin V_2 \\
0 \Longleftarrow \textrm{otherwise}
\end{cases}
\end{equation}
\end{lemma} 
\begin{proof} 
As for $v \in V_2$ the sequence $^{(n)}d(v)$ is unbounded then for sufficiently big $n$ equation \eqref{eq_dv1} and \eqref{eq_dv3} are fulfilled. As $y_n \rightarrow 0$ equation \eqref{eq_dv4} is also fulfilled for sufficiently big $n$. For $v \in V_2 \cap V_t$ the order of $\partial v$ is higher than that of $v$ and so $\frac{^{(n)}d(v)}{^{(n)}d(\partial v)} \rightarrow 0$ and therefore \eqref{eq_dv2} must be fulfilled for sufficiently big $n$. We will assume their validity in the sequel.  \\
We define 
\begin{equation}
\lambda_v^{(n)} := 
\begin{cases}
0 \Longleftarrow v \in V_0 \cup V_1\\
y_n \cdot \frac{2}{\pi} \cdot \ln(^{(n)}d(v))) \Longleftarrow v \in V_2 \\
1 \Longleftarrow v \in V_3 \cup V_4
\end{cases}
\end{equation}
and 
\begin{equation}
\widehat{\lambda_{\partial v}^{(n)}} := 
\begin{cases}
\lambda_{\partial v}^{(n)} \Longleftarrow v \in V_t\\
1 \Longleftarrow v \notin V_t
\end{cases}
\end{equation}
We can rewrite equation \eqref{eq_m0} by collecting terms that have the same projectors $w^{t}\cdot w$ and reach 
\begin{multline}
1_{r \times r} + M_0(^{(n)}Y,z_n) = 
\sum_{v \in V_1} (-G(v)\cdot y_n) +\\  \sum_{v \in (V_0 \cup V_1 \cup V_2)}(\widehat{\lambda_{\partial v}^{(n)}} - \lambda_v^{(n)}) \cdot  w_v^{t} \cdot w_v
\end{multline}
For $v \in (V_1 \cup V_2)$ we define the matrix
\begin{equation}
B(v,n) := \sum_{\hat{v} \in V_1 \cap V(v) } (-G(\hat{v})\cdot y_n) +\\  \sum_{\hat{v} \in (V_0 \cup V_1 \cup V_2) \cap V(v)}(\widehat{\lambda_{\partial \hat{v}}^{(n)}} - \lambda_{\hat{v}}^{(n)}) \cdot  w_{\hat{v}}^{t} \cdot w_{\hat{v}}
\end{equation}
For those $v,n$ for which $B(v,n)$ is invertible we define 
\begin{equation}
\tilde\xi_{v}^{(n)} := \lambda_v^{(n)}\cdot w_v^{t}B(v,n)^{-1}w_v
\end{equation}
If for all $\hat{v} \in I(v)$ the matrices $B(\hat{v},n) + (\lambda_v^{(n)} - \lambda_{\hat{v}}^{(n)}) \cdot w_{\hat{v}}^{t}\cdot w_{\hat{v}}$ are invertible then 
\begin{equation}\label{eq_xi_v}
\tilde\xi_v^{(n)} = \lambda_v^{(n)} \cdot \sum_{\hat{v} \in I(v)} w^{t}_{\hat{v}}\left( B(\hat{v},n) + (\lambda_v^{(n)} - \lambda_{\hat{v}}^{(n)}) \cdot w_{\hat{v}}^{t}\cdot w_{\hat{v}}\right)^{-1}w_{\hat{v}}
\end{equation}
We now start at the base $V_0$ of the tree $G_S$ and move upward.
For $\hat{v} \in V_0 \wedge [\hat{v}] \notin V_1$ we have 
\begin{equation}\label{eq_xi_v0}
\lambda_v^{(n)} \cdot w^{t}_{\hat{v}}\left( B(\hat{v},n) + (\lambda_v^{(n)} - \lambda_{\hat{v}}^{(n)}) \cdot w_{\hat{v}}^{t}\cdot w_{\hat{v}}\right)^{-1}w_{\hat{v}} = 1
\end{equation}
For $\hat{v} \in V_1$ we have 
\begin{equation}\label{eq_xi_v1}
\lambda_v^{(n)} \cdot w^{t}_{\hat{v}}\left( B(\hat{v},n) + (\lambda_v^{(n)} - \lambda_{\hat{v}}^{(n)}) \cdot w_{\hat{v}}^{t}\cdot w_{\hat{v}}\right)^{-1}w_{\hat{v}} = \frac{1}{1-\phi_v^{n}}
\end{equation}
according to equation \eqref{eq_v1_xi}. 
If for a $\hat{v} \in V_2$ the equation 
\begin{equation}\label{eq_sart_v2}
1 + (\lambda_v^{(n)} - \lambda_{\hat{v}}^{(n)})\cdot \frac{\tilde\xi_{\hat{v}}^{(n)}}{\lambda_{\hat{v}}^{(n)}} \neq 0
\end{equation} 
is true then according to Shermann Morisson
\begin{multline}\label{eq_xi_v2}
\lambda_v^{(n)} \cdot w^{t}_{\hat{v}}\left( B(\hat{v},n) + (\lambda_v^{(n)} - \lambda_{\hat{v}}^{(n)}) \cdot w_{\hat{v}}^{t}\cdot w_{\hat{v}}\right)^{-1}w_{\hat{v}} = \\ \lambda_v^{(n)}\left( \frac{\tilde\xi_{\hat{v}}^{(n)}}{\lambda_{\hat{v}}^{(n)}} -  
\frac{ (\lambda_v^{(n)} - \lambda_{\hat{v}}^{(n)})\cdot \left( \frac{\tilde \xi_{\hat{v}}^{(n)}}{\lambda_{\hat{v}}^{(n)}} \right)^2}{1 + (\lambda_v^{(n)} - \lambda_{\hat{v}}^{(n)})\cdot  \frac{\tilde \xi_{\hat{v}}^{(n)}}{\lambda_{\hat{v}}^{(n)}}   } \right) \\
= \frac{\tilde\xi_{\hat{v}}^{(n)}}{g(1,\phi_{\hat{v}}^{(n)},\tilde \xi_{\hat{v}}^{(n)})}
\end{multline}
From equation \eqref{eq_xi_v} and  \eqref{eq_xi_v2} it is clear that $\tilde\xi_v^{(n)}$ is a nonnegative real number for any $v \in V_2$ if this was true for all the $\tilde\xi_{\hat{v}}^{(n)}$ with $\hat{v} \in I(v)$. But because of equations  \eqref{eq_xi_v0} and \eqref{eq_xi_v1} as this is true for all vertices $\hat v \in V_1 \cup V_0$ it must be true for the vertices $v \in V_2$ with $I(v) \subset V_1 \cup V_0$ and therefore by induction in $o(v)$ for any $v \in V_2$ as $G_S$ is a tree.  Equation \eqref{eq_sart_v2} is always fulfilled because of this and equation \eqref{eq_dv2}. Equation \eqref{eq_det} for $ \xi_{\hat{v}}^{(n)} := \tilde \xi_{\hat{v}}^{(n)}$ now follows from a repeated application of the determinant matrix formula for matrices of the form $\left( B(\hat{v},n) + (\lambda_v^{(n)} - \lambda_{\hat{v}}^{(n)}) \cdot w_{\hat{v}}^{t}\cdot w_{\hat{v}}\right)$ using 
equation \eqref{det_Bl} as a staring point for $\hat{v} \in V_1$ and the fact that for $\hat{v} \in V_0 \wedge [\hat{v}] \notin V_1$ the determinant is trivially $\lambda_{ v}^{(n)}$.
\end{proof}
\emph{Corollary:}
Going along the same lines using Sherman Morrison \citep{shmo} repeatedly the inverse of $1 + M_0$ under the same conditions for $n$ is given  by:
\begin{multline}\label{eq_inv}
\left(1_{r\times r}+ M_0(^{(n)}Y_r,z_n) \right)^{-1} = 1_{r \times r} + \sum_{v \in V_0} \left(\frac{1}{\widehat{\lambda^{(n)}_{\partial v}}}-1 \right)\cdot w_v^{t}\cdot w_v \\ + \sum_{v \in V_1} \left(B_v(\lambda_v^{(n)})^{-1} - 1_{\#S(v) \times \#S(v)} \right) - 
\sum_{v \in V_2} \frac{1}{\lambda_v^{(n)}}\frac{1-\phi_v^{(n)}}{g(1,\phi_v^{(n)},\xi_v^{(n)})} \cdot u_v^{t} \cdot u_v
\end{multline}
with 
\begin{equation}
u_v := \sum_{v_0 \in X(v)} \left( \left( \prod_{\hat{v} \in ]v_0,v[}  \frac{1}{g\left(1,\phi_{\hat{v}}^{(n)},\xi_{\hat{v}}^{(n)}\right)} \right) u_{v_0} \right)
\end{equation}
with 
\begin{equation}
u_{v_0} := \begin{cases}
w_{v_0} \Leftarrow v_0 \in V_0 \\
\frac{1}{1-\phi_{v_0}^{(n)}}\cdot \left(-G_{v_0}^{-1} w_{v_0} \right) \Leftarrow v_0 \in V_1 \wedge q_{v_0} \neq 0 \\
\left(-H_{v_0}\right)^{-1} w_{v_0} \Leftarrow v_0 \in V_1 \wedge q_{v_0} = 0 \\
\end{cases}
\end{equation}
\begin{lemma}\label{le_delt}
With the data of Lemma \ref{MA_l1} and an integer $f \geq 1$ there is an $N_0(f)\in \N$ such that for  $n \geq N_0(f)$ 
\begin{equation}
\left|1 - \Delta_r(^{(n)}Y,z_n)^f \right| < 1
\end{equation} 
\end{lemma}
\begin{proof} We assume that with Lemma \ref{MA_det} that $n > N$. From equation \eqref{eq_xi} we see that $\xi_v^{(n)} \in \R \wedge \xi_v^{(n)} > 0$.  For $I(v) \cap V_2 = \emptyset$ we see that 
\begin{equation}\label{eq_delta_xi}
\#I(v) - \delta^{(0)}_n(v) < \xi_v^{(n)} < \#X(v) + \delta^{(1)}_n(v)
\end{equation}
where the $\delta$'s are suitable chosen nonnegative sequences converging towards $0$. This is obvious in the following way:  For any $v \in V_2$ and $\hat{v} \in I(v)$ the $\phi_{\hat{v}}^{(n)} \in ]0,1[$ are real numbers according to equation \eqref{eq_dv2} and \eqref{eq_dv3} and therefore $1 \leq g(1,\phi_{\hat{v}}^{(n)},\xi_{\hat{v}}^{(n)}) < \xi_{\hat{v}}^{(n)}$.  Because of equations \eqref{eq_dv2}, \eqref{eq_dv3} and \eqref{eq_xi}  for sufficiently big $n$ 
using $\#I(v) \geq 2$ the lower boundary on the left hand side of equation \eqref{eq_delta_xi} then is always bigger than $\frac{3}{2}$. Therefore equation \eqref{eq_delta_xi} can be proven inductively for any $v \in V_2$ for sufficiently big $n$. For sufficiently big $n$ we also know that $\left| s_n \right| < \frac{1}{2}$ and we can write with equation \eqref{eq_A3_even}:
\begin{equation}
y_n = \frac{1}{-\frac{1}{\pi}\ln(\left|s_n \right|) + \zeta_n } = A_n \cdot (1 + \delta^{(3)}_n)
\end{equation}
with $A_n := \frac{\pi}{ -\ln(\left| s_n \right|)}$ being a strictly positive real series and $\zeta_n \in \C$ a bounded series and therefore $\delta^{(3)}_n \in \C$ being a sequence converging towards $0$ and for $v \in V_2$ the series 
\begin{equation}\label{eq_phil}
y_n \cdot \frac{2}{\pi}\ln(^{(n)}d(v)) = E_n\cdot (1 + \delta^{(4)}_n)
\end{equation}
with real $E_n$ and $E_n \in ]0,1]$ for sufficiently big $n$ because of the definition of $V_2$. $\delta^{(4)}_n \in \C$ is a sequence converging towards $0$. 
For real $x \in ]0,1]$ real $m \geq 2$  and real $\beta \in [1,m]$ it is easy to calculate that $0 < g(m,x,\beta) \leq 1$ and that the relative deviation propagation module describing how the relative deviation $\delta^{(4)}_n$ is reflected in first order in the relative deviation of $g(m,x,\beta)$ 
\begin{equation}
\left| \frac{x}{g(m,x,\beta)} \frac{\partial}{\partial x} g(m,x,\beta) \right| = 
\left| m-1 + \frac{1 - \beta}{\beta + x\cdot (1-\beta)} \right| \leq m + \beta - 2
%\frac{(m-1)\cdot \beta +  x\cdot m \cdot (1- \beta)}{\beta + x\cdot (1-\beta)} < (m-1)\cdot \beta 
\end{equation} 
is bounded.
The relative deviation propagation module for $\beta$
\begin{equation}
\left| \frac{\beta}{g(m,x,\beta)} \frac{\partial}{\partial \beta} g(m,x,\beta) \right| = 
\left| \frac{\beta\cdot (1-x)}{\beta \cdot(1-x) + x} \right| \leq 1
\end{equation} 
Therefore with equation \eqref{eq_det}, noticing that $p_v \geq \frac{1}{2^{\#S(v)-1}} $ and $\phi^{(n)}_v \rightarrow 0$ for $v \in V_1$ we can write
\begin{equation}
det(1_{r \times r} + M_0(^{(n)}Y,z_n)) = B_n \cdot (1 + \epsilon^{(1)}_n)
\end{equation}
with a real $B_n \in ]0,1]$ and $\epsilon^{(1)}_n \in \C $ a sequence converging towards $0$. \\
We notice that 
\begin{multline}
\Delta_r(^{(n)}Y,z_n) = \det(1_{r \times r} + M_0(^{(n)}Y,z_n))\cdot \\  \det\left( 1_{r \times r} + \left( 1_{r \times r} + M_0(^{(n)}Y,z_n) \right)^{-1} \cdot y_n \cdot  M_1(^{(n)}Y,z_n) \right)
\end{multline}
To prove the Lemma it is therefore sufficient to show that the complex matrix 
\begin{equation}
\left( 1_{r \times r} + M_0(^{(n)}Y,z_n) \right)^{-1} \cdot y_n \cdot  M_1(^{(n)}Y,z_n)
\end{equation}
has matrix elements which vanish for $n \rightarrow \infty$. We do this for the individual summands in equation \eqref{eq_inv} and begin with the unity matrix. 
\begin{equation}
1_{r \times r} \cdot y_n \cdot M_1(^{(n)}Y,z_n) \rightarrow 0
\end{equation} 
as $M_1$ converges towards a constant matrix. 
For $v \in V_2$ we note: 
\begin{equation}\label{eq_lc}
 \frac{1}{\widehat{\lambda_v^{(n)}}}\cdot y_n = \frac{\pi}{2\cdot \ln(^{(n)}d(v))} \rightarrow 0
\end{equation}
and therefore for $v_0 \in V_0$ we see that $\frac{y_n}{\widehat{\lambda_{\partial v_0}^{(n)}}}$ converges to $0$ if $v_0 \in V_t \wedge \partial v_0 \in V_2$. If $v_0 \notin V_t \vee \partial v_0 \notin V_2$ the expression  $\frac{y_n}{\widehat{\lambda_{\partial v_0}^{(n)}}} = y_n \rightarrow 0$ and so all sumands in equation \eqref{eq_inv} from the set $V_0$ multiplied with $y_n \cdot M_1$ converge towards $0$. 
For $v \in V_1$ the matrix 
\begin{equation}
B_v(\widehat{\lambda_v^{(n)}})^{-1} \cdot y_n 
\end{equation}
according to Lemma \ref{le_binv} converges towards a constant matrix for indices $i,j \in S(v)$ and is $0$ if $i \notin S(v) \vee j \notin S(v)$. But according to Lemma \ref{MA_l1} for $i,j \in S(v)$ we know that $(M_1)_{i,j} \rightarrow 0$, as in this case $(i,j) \in T(v)$. 
For $v \in V_2$ we first of all note that with equation \eqref{eq_phil} and equation \eqref{eq_phiv}
\begin{equation}
\phi_v^{(n)} = \Re(\phi_v^{(n)})\cdot (1 + \delta^{(5)}_n(v))
\end{equation}
with $\Re(\phi_v^{(n)}) \in ]0,1]$ and $ \delta^{(5)}_n(v) \in \C$ a sequence converging towards $0$.
The factors $1 \leq g(1,\Re(\phi_v^{(n)}),\xi_v^{(n)}) \leq \xi_v^{(n)} < \#X(v) + 1$  for sufficiently big $n$ and therefore 
\begin{equation}
\frac{1-\phi_v^{(n)}}{g(1,\phi_v^{(n)},\xi_v^{(n)})} \cdot u_v^{t} \cdot u_v
\end{equation} 
is a bounded matrix. With equation \eqref{eq_lc} we see that therefore all the summands in \eqref{eq_inv} over vertices in $V_2$ when multiplied with $y_n \cdot M_1$ also converge towards $0$. 
So we can now finally for sufficiently big $n$ write 
\begin{equation}
\Delta_r(^{(n)}Y,z_n) = B_n\cdot (1 + \epsilon_n^{(2)})
\end{equation}
with real $B_n \in ]0,1]$ and $\epsilon_n^{(2)} \in \C$ a sequence converging towards $0$. 
But that also means that we can write
\begin{equation}\label{eq_deb}
\Delta_r(^{(n)}Y,z_n)^f = Z_n\cdot (1 + \epsilon_n^{(3)})
\end{equation}
with real $Z_n \in ]0,1]$ and $\epsilon_n^{(3)} \in \C$ a sequence converging towards $0$. 
So there is an $N_0(f)$ such that for all $n > N_0(f)$ equation \eqref{eq_deb} is true with $\left|\epsilon_n^{(3)} \right| < \frac{1}{2}$ and therefore 
\begin{equation}
\left|1 -  \Delta_r(^{(n)}Y,z_n)^f  \right| \leq 1 - \frac{Z_n}{2} < 1
\end{equation}  
\end{proof}
\begin{theorem}\label{th_de_1}
For $d \geq 2$ and an integer $f \geq 1$ there is an $\epsilon(f) > 0$ such that for any $z \in U_{\epsilon(f)}(\pm \frac{1}{2d}) \cap U_{\frac{1}{2d}}(0)$ and any $Y \in J_r$ the equation
\begin{equation}
\left|1 - \Delta_r(Y,z)^f\right| < 1 
\end{equation} 
is true.
\end{theorem}
\begin{proof} For $d \geq 3$ this follows from Lemma \ref{le_d3}. For $d=2$ let us assume that it was not true. Then there must be an infinite sequence $^{(n)}Y \in J_r$ and an infinite sequence $z_n \in U_{\frac{1}{2d}}(0)\setminus\{0\}$ with $z_n \rightarrow \frac{1}{2d}$ such that 
\begin{equation}\label{eq_viold}
\left|1- \Delta_r(^{(n)}Y,z_n)^f \right| \geq 1 
\end{equation}
But according to Lemma \ref{le_comp} there is a subsequence of $^{(n)}Y$ comparative to $z_n$ and according to Lemma \ref{le_delt} this subsequence violates equation \eqref{eq_viold} for sufficiently big $n$. Because of Lemma \ref{le_deven} the result is also true for a vicinity of $z = -1/2d$.
\end{proof}
\section{The moments and their Euler graph contributions}  
We are now ready to define the moments as holomorphic functions on an open vicinity of $U_{\frac{1}{2d}}(0)$. 
\subsection{The moments as analytic functions}
\begin{definition}
Let $K_r$ be the complete undirected graph with $r$ vertices from the set $\{1,\ldots,r\}$. Let $\tilde T$ be a spanning tree of $K_r$. Then there are $2^{(r-1)}$ possible orderings of it's edges as ordered pairs of integers, each edge can be ordered independantly of the other. Let $Tr(K_r)$ denote the set of all ordered spanning trees. For an ordered spanning tree $T \in Tr(K_r)$ let its edges be defined by the ordered index pair $(i,j)$ for an edge between $i$ and $j$, and let us denote the set of edges of $T$ with $E(T)$. For variables $X_{i,j}$ we define the differetial operator $O_T$ as 
\begin{equation}
O_T := \prod_{(i,j) \in E(T)} \frac{\partial}{\partial X_{i,j}}
\end{equation}
\end{definition}
\begin{definition}\label{df_pt}
For a tree $T \in Tr(K_r)$ and $ k \in (\N\setminus\{0\})^r$ using the notation of \citep[Lemma 2.2.]{dhoef1} we define the polynomial  $p_{T,k}$ in the variables $\omega = (\omega_1,\ldots,\omega_r)$ and $X_{i,j}$ with $X_{i,i} \equiv 0$ by
\begin{multline}
p_{T,k}(\omega,X) := \frac{1}{r!}\cdot \left(\det(1 + X)\right)^{s(k)} \cdot O_T \\
 \sum_{\vartheta \in D(k_1,\ldots,k_r)}
\left( \prod_{i=1}^r  (1 + \omega_i)^{k_i} \hat K(x_i,k_{\vartheta (i)},\omega_i) \right)\\
\left( \prod_{j=1}^r \frac {1} {(1-x_j)^2} \right) 
\frac {1} {det(W^{-1} + X)} \Bigr\rvert_{\forall a: x_a = 0}
\end{multline} 
with $s(k) := \sum_{i=0}^r k_i$.
\end{definition}
\begin{lemma}\label{le_T_Euler}
For any $\ell \in \N_0$ and any $f \in \N$ the polynomial
\begin{equation}\label{eq_mono_F}
\left(1-  \det(1 + X)^f  \right)^{\ell} \cdot p_{T,k}(\omega,X) \cdot \prod_{(i,j) \in E(T)} X_{i,j} = \sum_{F} \psi_{F,T,k,\ell,f}(\omega) \prod_{i,j} X_{i,j}^{F_{i,j}}
\end{equation}
has monomials $\prod_{i,j} X_{i,j}^{F_{i,j}}$ with nonzero coefficients $\psi$ corresponding to matrices $F \in \tilde H_r$ only, where 
\begin{equation}
\tilde H_r = \bigcup_{h_1 \geq 1,...,h_r \geq 1} \tilde H_r(h_1,...,h_r)
\end{equation}
\end{lemma}
\begin{proof} From the Leibniz formula for determinants and the decomposition of permutations in cycles we find
\begin{equation}\label{eq_det_1x}
\det(\lambda + X) = \sum_{\sigma \in S_r} (-1)^{sign(\sigma)}  \left(\prod_{\zeta \in Cycl(\sigma)} X_{\zeta} \right) \cdot \left( \prod_{k \in Fix(\sigma)}\lambda_k \right) 
\end{equation}
where $\lambda$ is the diagonal matrix $\lambda = diag(\lambda_1,\ldots,\lambda_r)$ and  $Fix(\sigma) = \{j: \sigma(j) = j\}$ 
and $Cycl(\sigma)$ denotes the cycles of $\sigma$ and $X_{\zeta} := \prod_{(i,j) \in E(G(\zeta))} X_{i,j}$ where $\zeta$ is identified with its subgraph $G(\zeta)$ in $DK_r$ (the complete digraph with $r$ vertices). Each such subgraph $G(\zeta)$ is a directed balanced  graph, i.e. a directed graph where for each vertex the number of ingoing edges equals the number of outgoing edges. Moreover 
\begin{equation}\label{eq_diffdet}
X_{i,j} \frac{\partial}{\partial X_{i,j}} \det(\lambda + X) = \sum_{\sigma \in S_r; \sigma(i) = j} (-1)^{sign(\sigma)} \left( \prod_{\zeta \in Cycl(\sigma)} X_{\zeta } \right)\cdot  \left( \prod_{k \in Fix(\sigma)}\lambda_k \right)
\end{equation}
The subgraph $G(\zeta)$  of $DK_r$ for any  cycle $\zeta$ on the right hand side of equation \eqref{eq_diffdet} contains the indices $i$ and $j$ and also contains a directed edge walk from $i$  to $j$ and from $j$ to $i$. Further applications of operators of the form 
\begin{equation}
X_{l,k} \frac{\partial}{\partial X_{l,k}}
\end{equation}
to the right hand side of \eqref{eq_diffdet} adds further conditions on the contributing permutations $\sigma \in S_r$ which make sure that the contributions belong to cycles $\zeta$  whose graphs $G(\zeta)$ contain an edge trail from $l$ to $k$ and from $k$ to $l$ for each such operator. 
Therefore the monomials on the right hand side of equation \eqref{eq_mono_F} 
\begin{equation}
\prod_{i,j} X_{i,j}^{F_{i,j}}
\end{equation}
have matrices $F$ which corresponds to graphs $G(F)$ which consist of a product of cyclic directed subgraphs of different copies of $DK_r$. These graphs therefore have the same number of ingoing edges as outgoing edges for each vertex as any cycle in that product has that property. For each edge $(i,j) \in E(T)$ they contain a directed edge path from $i$ to $j$ and from $j$ to $i$ and therefore by appropriate concatenation a directed edge trail in both directions between any two distinct vertices as T is a maximal tree of $K_r$. $G(F)$ therefore is a strongly connected balanced digraph and therefore is Eulerian. Therefore $cof(A-F) \neq 0$ and therefore $F \in \tilde H_r$.   
\end{proof}
\begin{lemma}\label{le_treec}
Let $T \in Tr(K_r)$ be an ordered spanning tree and for $(i,j) \in E(T)$ let there be functions 
\begin{equation}
\begin{array}{rccl}
f_{i,j} &\Z^d\setminus\{0\} & \mapsto & \C \\
&x& \mapsto & f_{i,j}(x)
\end{array}
\end{equation}
such that the infinite sum
\begin{equation}
\Phi_{i,j} := \sum_{x \in \Z^d \setminus\{0\}} f_{i,j}(x) 
\end{equation}
converges absolutely. Then the infinite sum
\begin{equation}
\sum_{Y \in J_r} \prod_{(i,j) \in E(T)} f_{i,j}(y_{i,j}) = \prod_{(i,j) \in E(T)} \Phi_{i,j}
\end{equation}
is well defined and converges absolutely too.
\end{lemma}
\begin{proof} As $T \in Tr(K_r)$ is a tree, the variables $y_{i,j}$ for $(i,j) \in E(T)$ are $r-1$ independant variables which cannot take on the value $0$ by definition of $J_r$ and therefore the Lemma holds.
\end{proof}
\begin{corollary}
The Lemma \ref{le_treec} is analogously true for absolutely integrable functions $g_{i,j}$ on $\R^d$ and the corresponding integrals instead of sums.
\end{corollary}
\begin{lemma}\label{le_Tconv}
For any tree $T \in Tr(K_r)$ we define  the product
\begin{equation}
\hat \Pi_T := \prod_{(i,j) \in E(T)} X_{i,j}
\end{equation}
and  for any $Y \in J_r$ and any $z \in D$  we derive
\begin{equation}
\Omega_T(Y,z) := (1 + h(0,d,z))^{r-1} \cdot \hat \Pi_T \rvert_{X_{i,j} = U_{i,j}(Y,z)} =  \prod_{(i,j) \in E(T) } h(y_{i,j},d,z)
\end{equation}
Then  $\Omega_T(Y,z)$ is holomorphic and with the notations of Lemma \ref{A2}
\begin{equation}
\left| \Omega_T(Y,z) \right| \leq  C(z)^{(r-1)} \cdot \exp\left( - \frac{2 \cdot \lambda(z)}{r\cdot(r-1)}  \sum_{1\leq m < l \leq r}\left\| y_{m,l} \right\| \right)
\end{equation}
is true. We here have used the norm $\left\| x \right\| := \sum_{a=1}^d \left| x_a \right|$ for $x \in \R^d$
\end{lemma}
\begin{proof} With Lemma \ref{A2} we only have to prove that for any ordered spanning tree $T$
\begin{equation}\label{eq_yT}
\sum_{(i,j) \in E(T)} \left\|y_{i,j} \right\| \geq \frac{2}{r\cdot(r-1)} \sum_{1\leq m <l \leq r} \left\|y_{m,l} \right\|
\end{equation}
To prove this we look at an index pair $(m,l)$. As $T$ is a spanning tree there is an edge path (possibly against the order of $T$ which is irrelevant for this discussion) between $m$ and $l$ with edges in a set $\kappa(m,k)$ . From the triangle inequality and the fact that $y_{i,j} = -y_{j,i}$ we have 
\begin{equation}
\left\| y_{m,l} \right\| \leq \sum_{(i,j) \in \kappa(m,k)} \left\| y_{i,j} \right\| \leq  \sum_{(i,j) \in E(T)} \left\|y_{i,j} \right\|
\end{equation}
from which equation \eqref{eq_yT} follows. 
\end{proof}
\begin{lemma}\label{le_grconv}
Let $G$ be a graph with vertices $1,\ldots,r$ which contains an ordered spanning tree $T \in Tr(K_r)$.Then for $z \in D$ 
\begin{equation}
\sum_{Y \in J_r} \prod_{(i,j) \in E(G)} h(Y_i - Y_j,d,z) 
\end{equation}
\end{lemma}
converges absolutely.
\begin{proof}
With Lemma \ref{A2} c) we can write 
\begin{equation}
\left| \prod_{(i,j) \in E(G)} h(Y_i - Y_j,d,z) \right| =  C(z)^{\#E(G) - \#E(T)}
\left| \prod_{(i,j) \in E(T)} h(Y_i - Y_j,d,z)  \right|
\end{equation}
The proof now follows from \ref{le_Tconv}.
\end{proof}
\begin{theorem}\label{th_defp}
For $d \geq 2$ there is an open neighborhood $U \supset \overline{U_{\frac{1}{2d}}(0)} \cap D$ such that the function 
\begin{equation}
\begin{array}{rccl}
S_k : &  U  & \mapsto & \C \\
& z & \mapsto &\underset{N \rightarrow \infty} {\lim}\left( \sum_{w \in W_{2N}} P(N_{2k_1}(w),\ldots,N_{2k_r}(w)) \cdot z^{length(w)} \right)
\end{array}
\end{equation}
is well defined and a holomorphic function on $U$. 
\end{theorem}
\begin{proof}
From Lemma \ref{A3}, Lemma \ref{le_Tconv} and Theorem \ref{Th_delta} we realize that for any spanning tree $T \in Tr(K_r)$ and any $\delta > 0$ and $\gamma_r = \frac{1}{2}$ there is an $\epsilon(\gamma_r)$ such that  $z \in M_{\frac{1}{2d} + \epsilon(\gamma_r),\frac{1}{2d}-\delta}$ and $Y\in J_r$ the expression 
\begin{multline}
C_{T,k}(Y,z) :=  \frac{\Omega_T(Y,z)}{\Delta_r(Y,z)^{s(k)}\cdot (1 + h(0,d,z))^{s(k) + r-1}} \\ \cdot p_{T,k}\left( h(0,d,z)\cdot u(r),\frac{h(y_{i,j},d,z))}{1 + h(0,d,z)}\right) 
\end{multline}
($u(r) =  (1,\ldots,1)$, i.e. $u(r)_j = 1$ for any $j=1,\ldots,r$) is well defined and holomorphic and fulfils the following inequality:
\begin{multline}
\left| C_{T,k}(Y,z) \right| \leq\frac{\left(2\cdot (1 + \left|h(0,d,z)\right|)\right)^{s(k)\cdot(r-1)}\cdot 2^{s(k)}}{\left| 1 + h(0,d,z) \right|^{s(k) + r -1}} \\
 C(z)^{(r-1)} \cdot \exp\left( - \frac{2 \cdot \lambda(z)}{r\cdot(r-1)}  \sum_{1\leq i < j \leq r}\left\| y_{i,j} \right\| \right)\\
 \left| p_{T,k}\left(h(0,d,z)\cdot u(r),\frac{h(y_{i,j},d,z))}{1 + h(0,d,z)}\right)  \right|
\end{multline}
But that means that for a given $z$ the sum 
\begin{equation}
\hat S_{T,k}(z) := \sum_{Y\in J_r} C_{T,k}(Y,z)
\end{equation}
converges absolutely and locally uniform and therefore is a holomorphic function. 
From  \citep[Lemma 2.2.]{dhoef1} using the matrix tree theorem for the Kirchhoff polynomial \citep{aar} and its generalizations by Tutte for multidigraphs \citep{tutte} we know that there is a subset $Cof_r \subset Tr(Kr)$ of the set of ordered spanning trees such that
\begin{equation}
cof (\hat A - \hat F) = \sum_{T \in Cof_r} \hat \Pi_T \cdot O_T
\end{equation} 
and therefore 
\begin{equation}
 \sum_{w \in W_{2N}} P(N_{2k_1}(w),\ldots,N_{2k_r}(w)) \cdot z^{length(w)} = z\cdot \frac{\partial}{\partial z}\left(\sum_{T \in Cof_r} \Bigl[\hat S_{T,k}(z) \Bigr]_{z \diamond 2N}\right)
\end{equation}
From the identity theorem for power series we therefore can define 
\begin{equation}
\hat S_k(z) := \sum_{T\in Cof_r} \hat S_{T,k}(z)
\end{equation}
\begin{equation}\label{eq_hats}
S_{k}(z) :=  z\cdot \frac{\partial}{\partial z}\left( \hat S_k(z) \right)
\end{equation}
Now for any $z \in  \overline{U_{\frac{1}{2d}}(0)} \cap D$ we can find an open neighborhood $U(z)  \subset \C$ such that for a suitably chosen $\delta$,  $ U(z) \subset M_{\frac{1}{2d} + \epsilon(\gamma_r),\frac{1}{2d}-\delta}$. And so the theorem is proven.
\end{proof} 
\begin{remark}
As $S_k$ does not have any singularities on $U$ the behaviour of its Taylor coefficients for large index are completely determined by the asymptotic behaviour of $S_k$ in the vicinity of the points $z = \pm \frac{1}{2d}$. 
\end{remark}
\subsection{Defining the sum over matrices}
We will now work on a proper definition of the sum over matrices from the geometric expansion of powers of the first hit determinant. 
\begin{definition}
For any vector $k = (k_1,... k_r)$ of strictly positive integers and any $\ell \in \N$ we define
\begin{equation}
D_{\ell,k}(\omega,X) := \frac{\left(1 - \det(1 + X)^{s(k)}  \right)^{\ell}}{\prod_{m=1}^{r}(1 + \omega_m)^{k_m}} 
\sum_{T \in Cof_r} \hat \Pi_T \cdot p_{T,k}(\omega,X) 
\end{equation}and 
\begin{equation}
C_{\ell,k}(Y,z) := D_{\ell,k}\left(h(0,d,z) \cdot u(r),\frac{h(y_{i,j},d,z)}{1+ h(0,d,z)} \right)
\end{equation} 
\end{definition}
\begin{definition}
For  $F \in \tilde H_r(h_1,...,h_r)$ we define
\begin{equation}
D_{F}(X) :=  \prod_{\substack{1 \leq a \leq r \\ 1 \leq b \leq r}} X_{a,b}^{F_{a,b}} 
\end{equation}
\begin{equation}
C_F(Y,z) := D_F\left(\frac{h(y_{i,j},d,z)}{1+ h(0,d,z)} \right)
\end{equation}
\begin{multline}\label{eq_dfk1}
D_{F,k}(\omega,X) := cof(A-F) \cdot M(F)  \cdot  D_F(X)  \cdot \\ 
\left[ \frac{1}{r\,!}\sum_{\sigma \in D(k_1,\ldots,k_r)} 
\prod_{j=1}^r (-1)^{h_j} \cdot h_j\,!\cdot K(h_j,k_{\sigma (j)},\omega_j)
)
)
)
\right]
\end{multline} 
\begin{equation}\label{eq_dfk2}
C_{F,k}(Y,z) := D_{F,k}\left(h(0,d,z)\cdot u(r),\frac{h(y_{i,j},d,z)}{1+ h(0,d,z)} \right)
\end{equation}
\end{definition}
\begin{lemma}\label{le_ceven}
For $z \in D\setminus P_0$ and $Y\in J_r$ the functions $C_F(Y,z)$ and $C_{F,k}(Y,z)$ are even functions in $z$.
\end{lemma}
\begin{proof}
As the matrix $F$ belongs to an Euler graph taking a Euler trail $tr$ and summing along the edge sequence $E(tr)$  in it  
\begin{equation}
0 = \sum_{(i,j) \in E(tr)} Y_i - Y_j 
\end{equation} 
because $tr$ returns to its beginning. But this sum also is the sum over all edges of the graph $G(F)$ because $tr$ is Eulerian. Because of Remark \ref{re_heven} and equation \eqref{eq_rho} then the Lemma is true.
\end{proof}
\begin{lemma}\label{le_l_F}
For a given $k$ and $\ell$ there are uniquely determined coefficients $\gamma(\ell,k,F,\omega)$ such that 
\begin{equation}\label{eq_l_F}
D_{\ell,k}(\omega,X)  = \sum_{F \in \tilde H_r} \gamma(\ell,k,F,\omega) \cdot D_{F}(X)
\end{equation}
The sum in equation \eqref{eq_l_F} is finite, i.e. only finitely many of the coefficients $\gamma(\ell,k,F,\omega) \neq 0_f$ (where $0_f$ is the zero function) for a given $\ell$ and $k$. For a given $F$ and $k$ the number of $\ell$ for which $\gamma(\ell,k,F,\omega) \neq 0_f$ is also finite. Moreover 
\begin{equation} \label{eq_gammadeg}
\gamma(\ell,k,F,\omega) \cdot \prod_{m=1}^{r} (1 + \omega_m)^{k_m}  
\end{equation}
is a polynomial in $\omega_i$ where the sum of the degrees in the $\omega_i$ is smaller or equal to $s(k) - r$ and the identity
\begin{equation}\label{eq_sumga}
D_{F,k}(\omega,X) = \sum_{\ell = 0}^{\infty} \gamma(\ell,k,F,\omega)\cdot D_F(X) 
\end{equation}
is true.
\end{lemma}
\begin{proof} $D_{\ell,k}(\omega,X)$ is a polynomial in $X$ whose degree matrices $F$ fullfil $F \in \tilde H_r$ according to Lemma  \ref{le_T_Euler}. Equation \eqref{eq_l_F} therefore is just writing a polynomial as sum of monomials and therefore the sum in the equation is finite. The coefficients $\gamma(\ell,k,F,\omega)$ therefore are also uniquely defined by the equation. As $p_T(\omega,X)$ is a polynomial in $\omega_i$ with the sum of the degrees in $\omega_i$ is smaller or equal to $ s(k) - r$ by the definition \ref{df_pt} the expression in equation \eqref{eq_gammadeg} must have the same highest degree. The term $1-\det(1 + X)^{s(k)} $ is a polynomial in $X_{i,j}$ where each term has at least degree $2$. Therefore $(1 - \det(1 + X)^{s(k)})^{\ell}$ is a polynomial where each term has at least degree $2\cdot\ell$. For a given $F \in \tilde H_r(h_1,...,h_r)$ and $2\cdot \ell > h_1 + \ldots + h_r$ therefore $\gamma(\ell,k,F,\omega) = 0_f$. Equation \eqref{eq_sumga} follows from the defintion of $p_T$ in \ref{df_pt},  \citep[Theorem 2.3, Lemma 2.2.]{dhoef1}, especially \citep[eq. 2.26]{dhoef1} and the identity theorem for polynomials in $\omega_i$. 
\end{proof}
\begin{definition}\label{de_hlk}
For a given $k$ and $\ell$ we define the finite set 
\begin{equation}
\tilde H_{\ell,k} := \{ F \in \tilde H_r; \gamma(\ell,k,F,\omega) \neq 0_f\}
\end{equation}
\end{definition}
\begin{theorem} \label{t_feyn}
\textbf{Feynman Graph Summation}\\
For $z \in D\setminus P_0$ the sums
\begin{equation}
\hat S_{\ell,k}(z) := \sum_{Y\in J_r} C_{\ell,k}(Y,z) 
\end{equation}
and 
\begin{equation}
\hat S_{F}(z) := \sum_{Y\in J_r}C_F(Y,z)
\end{equation}
\begin{equation}\label{eq_sfk}
\hat S_{F,k} (z) :=   \sum_{Y\in J_r} C_{F,k}(Y,z) 
\end{equation}
converge absolutely and therefore are holomorphic on $D\setminus P_0$.
Moreover 
\begin{equation}\label{eq_lk_ell}
\hat S_{\ell,k}(z) := \sum_{F \in \tilde H_{\ell,k}}\hat  S_{F}(z)  \cdot  \gamma(\ell,k,F,h(0,d,z)\cdot u(r))
\end{equation}
and 
\begin{equation}\label{eq_lk_F}
\hat S_{F,k}(z) := \sum_{\ell = 0}^{\infty} \gamma(\ell,k,F,h(0,d,z)\cdot u(r)) \cdot \hat S_{F}(z)
\end{equation}

and for $z \in U_{\frac{1}{2d}}(0) \cap U_{\epsilon(s(k))}(\pm \frac{1}{2d})$ 
\begin{equation}\label{eq_feynman}
\hat S_{k}(z) = \sum_{\ell \geq 0} \hat S_{\ell,k}(z) = \sum_{\ell \geq 0}  \left(  \sum_{F \in \tilde H_{\ell,k}} \gamma(\ell,k,F,h(0,d,z)\cdot u(r)) \cdot \hat S_{F}(z) \right)
\end{equation}
The sum over $\ell$ converges absolutely.
\end{theorem}
\begin{proof} The convergence of the sums over $Y \in J_r$ follows from Lemma \ref{le_Tconv}. Equation \eqref{eq_lk_F} follows from Lemma \ref{le_l_F}. The absolute convergence of the sum over $\ell$ follows from Theorem \ref{th_de_1} from the absolute convergence of the geometrical series.
\end{proof} 
\begin{remark}
 As a note of caution it should be mentioned however that the sum over $F$  in equation \eqref{eq_feynman} cannot be interchanged with the sum over $\ell$ because the sum over $F \in \tilde H_r$ does not converge (absolutely) in general. It seems that the factors $\gamma(\ell,k,F,\omega)$ come from a general inclusion exclusion principle, mediated by the factors 
$1 - \Delta^{s(k)}$
\end{remark}

\section{The 2-dimensional case}
We will now turn to the case of two dimensions. 
\subsection{Replacing the sum over $Y \in J_r$ by integration}
As a first step we will replace the sum over $Y \in J_r$ by an integral in the asymptotic regime. In the case $d=2$ this is a nontrivial task because a special generalization of the Euler--Maclaurin expansion has to be used. 
\begin{lemma} \label{le_ht}
Let
\begin{equation}
h_1(x,z) := \frac{4 + s}{2\pi} K_0(\left|x\right|\sqrt{s})
\end{equation}
and 
\begin{multline}
h_2(x,z)= - \frac{4+s}{8\pi\left|x\right|^2}\cdot\Biggl(K_1(\left|x\right|\sqrt{s})\cdot  (\left|x\right|\sqrt{s})  - K_2(\left|x\right|\sqrt{s}) \cdot(\left|x\right|\sqrt{s})^2 \\ - \frac{x_4}{6\left|x\right|^4}\cdot K_3(\left|x\right|\sqrt{s})\cdot (\left|x\right|\sqrt{s})^3 \Biggr)
\end{multline}
For $\lambda < \frac{\pi}{2}$ and $z \in V_{\frac{1}{4}} \cap W_{\lambda}$ there is a constant $C$ such that  for 
\begin{equation}\label{eq_2d_h}
 h_{t}(x,z) := h(x,2,z) - h_1(x,z)  - h_2(x,z)
\end{equation}
the equation
\begin{equation}\label{eq_2dht}
\left| h_{t}(x,z) \right|  \leq \frac{C}{\left|x\right|^4}
\end{equation}
is true. 
\end{lemma}
\begin{proof} According to Lemma \ref{B2} there is an $x_0  \in \R$ such that for $\left|x\right| \geq x_0$ equation \eqref{eq_2dht} is true for a given $C_1$. So we only need to prove that for the finitely many $x \in \Z^2 \wedge \left|x\right| \leq x_0$ there exists a $C_2 \in \R $ such that 
\begin{equation}
\left|h_t(x,z) \right| \leq C_2
\end{equation}
From the asymptotic expansion in equation \eqref{eq_A3_even} it is easy to calculate that 
the leading singularity of $h(x,2,z)$ proportional to $\ln(1-16z^2)$
for $z \rightarrow \pm \frac{1}{4}$ is absent in $h_t(x,z)$ . But that means that for $z \in  V_{\frac{1}{4}} \cap W_{\lambda} $ the function $h_t(x,z)$ has a continuous continuation for $z = \pm \frac{1}{4}$ and therefore onto the whole border of $V_{\frac{1}{4}} \cap W_{\lambda} $ and therefore its absolute value has a maximum for a given $x$ which proves the Lemma for $C = \max(C_1,C_2\cdot x_0^4)$. 
\end{proof}
\begin{lemma}\label{le_integr}
Lef $F \in \tilde H_r(h_1,\ldots,h_r)$. Then according to Theorem  \ref{t_feyn} the sum
\begin{equation}
I_{\Sigma}(F) := \sum_{Y \in J_r} \prod_{i \neq j} h(y_{i,j},2,z)^{F_{i,j}}
\end{equation}
converges absolutely for $z \in D$. The integral
\begin{equation}
I(F) :=\int_{(\R^2)^r} \prod_{1 \leq k \leq r} d^2 y_k \prod_{i \neq j} \left( \frac{2}{\pi} \cdot K_0(\left|y_i - y_j\right|) \right)^{F_{i,j}}\cdot \delta(y_1) 
\end{equation}
converges absolutely too and we have the asymptotic expansion 
\begin{equation}
I_{\Sigma}(F) = \frac{1}{s^{(r-1)}}\cdot I(F) \cdot \left(1 + o\left(s\cdot \ln(s)^{K(F)+1} \right)\right)
\end{equation}
around $z = \frac{1}{4}$  (and a corresponding asymptotic expansion around $z = -\frac{1}{4}$) with 
\begin{equation}
K(F) := \sum_{i \neq j} F_{i,j} = \sum_{i=1}^r h_i 
\end{equation}
\end{lemma}
\begin{proof} As for $n \in \N_0$  \citep[eq. 9.7.2]{danraf}
\begin{equation}\label{eq_knasy}
K_n(z) \sim \sqrt{\frac{\pi}{2z}} \exp(-z) \cdot \left(1 + O\left(\frac{1}{z}\right)\right)
\end{equation} 
for large $z$ and the graph corresponding to  $F \in \tilde H_r(h_1,\ldots,h_r)$ contains a spanning tree we see from the Corollary to Lemma \ref{le_treec} together with Lemma \ref{B3} that the integral $ I(F)$ converges absolutely (the logarithmic singularities of the $K_0$ for $y_i = y_j$ are of course integrable). \\
We further write
 \begin{equation}\label{eq_multi}
I_{\Sigma}(F) := \sum_{Y \in J_r} \prod_{i \neq j} (h_1(y_{i,j},z) + h_2(y_{i,j},z) +h_t(y_{i,j},z))^{F_{i,j}}
\end{equation}
and expand multinomially. Then there is the term 
\begin{equation}
I_{\Sigma,1}(F) := \sum_{Y \in J_r} \prod_{i \neq j} h_1(y_{i,j},z)^{F_{i,j}}
\end{equation} 
and 
\begin{equation}
I_{\Sigma,t}(F) := \sum_{Y \in J_r} \prod_{i \neq j} h_t(y_{i,j},z)^{F_{i,j}}
\end{equation} 
and finitely many other terms which contain at least one factor $h_2$ or $h_t$. All these sums converge absolutely as the graph belonging to $F$ contains a maximal tree and therefore Lemma \ref{le_treec} together with equation \eqref{eq_knasy} ensures the absolute convergence.   \\
Because of Lemma \ref{le_ht} $I_{\Sigma,t}(F) \sim O(1)$. \\
$I_{\Sigma,1}$ can be interpreted  as a modified multidimensional trapezoidal summation with stepwidth $ h = \left|\sqrt{s}\right|$. We use the generalization of the Euler Maclaurin formula by Sidi \citep{sidi} and apply it in each of the $2(r-1)$ dimensions consecutively. We realize from the expansions \citep[eq. 9.6.13]{danraf}  that integrating piecewise towards and from the logarithmic singularities in each dimension the resulting integrals and coefficients in the generalized Euler Maclaurin expansion all exist without Hadamard regularisation. The integrals over given variables and the product of factors which contain these given variables are continous functions in the remaining variables as we only integrate over logarithmic singularities. Therefore we have only logarithmic singularities in the remaining variables after multiplying with the rest of the factors and can again apply Sidis formulas. Moreover the function to be integrated over a given dimension with variable $x_k$ does not have a singularity stronger than $\ln(x_k-a)^{K(F)}$ in a point $a$ and from the form of the coefficients \citep[eq. 1.2]{sidi} in Sidi's generalized Euler Maclaurin expansion we see that a coefficient multiplying $h\cdot\ln(h)^m$  does not have a singularity stronger than $\ln(x_k-a)^{K(F)-m}$ in a point $a$. As any singularity only occurs when two dimensional conditions $y_i - y_j = 0$ are met an additional factor $h$ multiplies the leading terms of the asymptotic expansion in $h$. Therefore the leading behaviour of the asymptotic expansion according to stepwidth is 
\begin{equation}
\left|s\right|^{r-1} \cdot I_{\Sigma,1}(F) = \left(\frac{s}{\left|s \right|}\right)^{1-r}\cdot I(F)\cdot (1 + o(s \cdot \ln(\left|s\right|)^{K(F) + 1})
\end{equation}
The finitely many other terms in the multinomial expansion of \eqref{eq_multi} have the form
\begin{equation}
I_{\Sigma,b}(F) := \sum_{Y \in J_r} \prod_{i \neq j} h_{b,i,j}(y_{i,j},z)^{F_{i,j}}
\end{equation}
where $h_{b,i,j}$ can be either $h_1$ or $h_2$ or $h_t$. 
Now we define
\begin{equation}
I_{\Sigma,c}(F) := \sum_{Y \in J_r} \prod_{i \neq j} h_{c,i,j}(y_{i,j},z)^{F_{i,j}}
\end{equation}
where $h_{c,i,j}(x,z) = \frac{C}{\left|x\right|^4}$ if $h_{b,i,j}(x,z) = h_t(x,z)$ and $h_{c,i,j}(x,z) = \left|h_{b,i,j}(x,z) \right|$ otherwise.   Then according to Lemma \ref{le_treec} this sum again converges absolutely. Moreover
\begin{equation}
\left| I_{\Sigma,b} \right| \leq I_{\Sigma,c} 
\end{equation}
We now again interprete $I_{\Sigma,c}$ as a modified multidimensional trapezoidal sum according to Sidi with a stepwidth of $\left|\sqrt{s}\right|$. Factors $h_{c,i,j}$ which are not $\left|h_1\right|$ have an additional scaling factor $\left|s\right|$ for $\left|h_2\right|$ and $\left|s\right|^2$ for $C/\left|x\right|^4$ in the generalized Euler Maclaurin expansion of the integrals. Integrating and Hadamard regularizing according to \citep[eq. 1.8]{sidi} in each dimension leads to terms in the remaining variables which are less singular in the resulting variables by one order. So they again lead to integrands in the resulting dimensions which have algebro logarithmic singularities in the resulting variables and so Sidi generalized Euler Maclaurin expansion can be applied to them too.  
Integrals in the expansion  according to the the formula for the $K_n$ in \citep[eq. 9.6.11]{danraf} for $h_2$ and obviously for $C/\left| x \right|^4$ of course have to be Hadamard regularized in the appropriate cases and therefore are finite and multiplied with those additional scaling factors and therefore the resulting multidimensional integral is scaled down by at least a factor $s$. The asymptotic corrections for stepwidth on the other hand after multiplying with the scaling factors then according to Sidis formulas scale like those for $I_{\Sigma,1}$ and so we are left with 
\begin{equation}
\left|s\right|^{r-1}I_{\Sigma,c} \sim o( s \ln(\left|s\right|)^{K(F)+1})
\end{equation} 
and therefore the Lemma is proven. 
\end{proof}
\begin{corollary}
The sum 
\begin{equation}
I_{\Sigma,abs}(F) := \sum_{Y \in J_r} \prod_{i \neq j} \left| h(y_{i,j},2,z) \right|^{F_{i,j}}
\end{equation}
fulfils 
\begin{equation}
I_{\Sigma,abs}(F) =  \frac{1}{\left| s \right|^{(r-1)}}\cdot I(F) \cdot \left(1 + o\left(s\cdot \ln(s)^{K(F)+1} \right)\right)
\end{equation}
\end{corollary}
\subsection{Defining a standard asymptotic expansion}
After solving the summation over $Y \in J_r$ we now see that for $d=2$ we can get all logarithmic corrections even if we drop terms which scale down by a factor $s$. We will make ample use of this to get simple, managable equations.  
\begin{theorem}\label{th_nge}
\textbf{(Naive Graph evaluation in two dimensions allowed)}\\
Let $d=2$ and $M \in \N$ be a natural number. Then for $k=\{k_1,\ldots,k_r\}$ and $ \hat S_k(z)$ there is a finite set $\tilde H^{(M)}_{k} \subset \tilde H_r$ such that  
\begin{equation}
\hat S_k(z) - \sum_{F \in \tilde H^{(M)}_{k}} \hat S_{F,k}(z) = o\left(\frac{1}{s^{r-1}\cdot \ln(s)^M}\right)
\end{equation}  
\end{theorem}
\begin{proof} From Theorem \ref{t_feyn} we know that for $z \in U_{\frac{1}{2d}}(0) \cap U_{\epsilon(s(k))}(\frac{1}{2d}) $ and a given $m \in \N$  we can define the absolutely converging sum 
\begin{equation}
\hat S_{\geq m,k}(z) := \sum_{\ell \geq m} \hat S_{\ell,k}(z)
\end{equation}
Then from the geometric series we know that 
\begin{multline}
\hat S_{\geq m,k}(z) =  \sum_{Y\in J_r} \frac{1}{\Delta_r(Y,z)^{s(k)}}\cdot C_{m,k}(Y,z)\\
= \sum_{F \in \tilde H_{m,k}} \gamma(m,k,F,h(0,2,z) \cdot u(r)) 
\left(\sum_{Y \in J_r} \frac{1}{\Delta_r(Y,z)^{s(k)}} \cdot C_{F}(Y,z) \right) 
\end{multline}
Now according to equation \eqref{eq_A4_3} of Corollary \ref{co_A3} we note that for $z \in U_{\frac{1}{2d}}(0) $
\begin{equation}
\left| \sum_{Y \in J_r} \frac{1}{\Delta_r(Y,z)^{s(k)}} \cdot C_{F}(Y,z) \right| \leq \left( 2(1 +\left| h(0,2,z) \right|)\right)^{s(k) \cdot (r-1)} \sum_{Y \in J_r} \left| C_{F}(Y,z) \right| 
\end{equation}
On the other hand according to Lemma \ref{le_integr} and its Corollary we know that for $z \rightarrow \frac{1}{4}$ 
\begin{equation}
\sum_{Y \in J_r} \left| C_{F}(Y,z) \right| = \frac{I(F)}{\left|s \right|^{(r-1)}\left| \ln(s) \right|^{ K(F)}} \cdot \left( 1 + O\left(\frac{1}{\ln(s)}\right)\right)
\end{equation}  
and from Lemma \ref{le_l_F} we know that there is a constant $\eta(m,k,F)$ such that in a vicinity of $z = 1/4$ 
\begin{equation}\label{eq_gammask}
\left| \gamma(m,k,F,h(0,2,z))  \right|  \leq \frac{ \eta(m,k,F)}{  \left| \ln(s)^r \right|}
\end{equation}
Noticing that $(1 + h(0,2,z)) \sim O\left( \ln(s) \right) $ and that for $F \in \tilde H_{m,k}$ we have $K(F) \geq 2\cdot m$ by choosing $m$ such that  $ 2m   > M + s(k) \cdot r + 1$ we notice that 
\begin{equation}\label{eq_sgem}
\left| \hat S_{\geq m, k} \right| \sim o\left(\frac{1}{ s^{r-1} \ln(s)^M} \right)
\end{equation} 
We define $m(M) := \left[\frac{1}{2} (M + s(k) \cdot r +1)\right] + 1$ (the bracket $[]$ here denoting the biggest integer smaller than the expression it is around) and define the finite sum
\begin{equation}
\hat S_{\leq m(M),k}(z) := \sum_{\ell = 0}^{m(M)} \gamma(\ell,k,F,h(0,2,z)\cdot u(r))\left(\sum_{F \in H_{\ell,k}}\hat S_F(z) \right) 
\end{equation}
We have already shown in \eqref{eq_sgem} that 
\begin{equation}
\hat S_k(z) - \hat S_{\leq m(M),k}(z) \sim  o\left(\frac{1}{ s^{r-1} \ln(s)^M} \right)
\end{equation}
The sum $\hat S_{\leq m(M),k}(z) $ as it is finite can be reordered. 
We know that for a given matrix $F$ the coefficient $\gamma(\ell,k,F,\omega) = 0_f$ (the zero function) if $2 \ell > K(F)$. Let us define
\begin{equation}
\tilde H_{\leq q,k} := \bigcup_{0\leq \ell \leq q } \tilde H_{\ell,k}
\end{equation}
Then if $F \in \tilde H_{\leq m(M),k}$ and $K(F)  \leq M -r $ we know that $\hat S_{\leq m(M),k}(z)$ contains the sum 
\begin{equation}
\hat S_{F,k}(z) = \sum_{\ell = 0}^{\infty} \gamma(\ell,k,F,h(0,2,z)\cdot u(r))\left(\hat S_F(z)\right)
\end{equation}
and no other terms for this given matrix $F$.
We also know from equation \eqref{eq_gammask} that for $F \in  \tilde H_{\leq m(M),k} $ with $K(F) > M-r $ the term 
\begin{equation}
\gamma(\ell,k,F,h(0,2,z)\cdot u(r)) \cdot \hat S_F(z) \sim  o\left(\frac{1}{ s^{r-1} \ln(s)^{M }} \right)
\end{equation}
From simple algebra 
\begin{equation}
\hat S_{F,k}(z) \sim  O\left(\frac{1}{ s^{r-1} \ln(s)^{K(F) + r}} \right)
\end{equation}
So if we define
\begin{equation}
\tilde  H^{(M)}_k := \left\{ F \in \bigcup_{0 \leq \ell \leq m(M) } \tilde H_{\ell,k};  K(F) \leq M - r \right\}
\end{equation}
the Theorem is true. This choice of $\tilde H_k^{(M)}$ is pretty much what one would get from naive power counting and therefore we have given the Theorem its title.
\end{proof} 
\begin{definition}
For a matrix $F \in \tilde H_r$ we define
\begin{equation}
M(F) := \prod_{\substack{1 \leq a \leq r \\ 1 \leq b \leq r}} \frac {1} 
{ F_{a,b}\,!}
\end{equation}
 and the functions 
\begin{multline}\label{eq_skf}
\widetilde{S}_{F,k}(z) := cof(A-F) \cdot \frac{I(F) \cdot M(F) }{2^{r-1}\cdot (1-16z^2)^{r-1}} \cdot \\ \left[ \frac{1}{r\,!} \sum_{\sigma \in D(k_1,\ldots,k_r)} 
\prod_{j=1}^r  \widetilde{K}(h_j,k_{\sigma (j)},\widetilde{h}(0,2,z))
\right]
\end{multline}
and (to simplify equations later)
\begin{multline}\label{eq_sckf}
\mathcal{T}_{F,k}(z) := cof(A-F) \cdot \frac{I(F) \cdot M(F) }{2^{r-1}\cdot (1-16z^2)^{r-1}} \cdot \\ \left[ \frac{1}{r\,!} \sum_{\sigma \in S_r} 
\prod_{j=1}^r  \widetilde{K}(h_j,k_{\sigma (j)},\widetilde{h}(0,2,z))
\right]
\end{multline} 
with the definitions
\begin{equation}
 \widetilde{K}(m,k,\omega) := \frac{(-1)^{m}\cdot m\,!\cdot K(m,k,\omega)}{(1 + \omega)^{m}}
\end{equation}
and 
\begin{equation}\label{eq_h0ti}
\widetilde{h}(0,2,z) := -\frac{1}{\pi}\left(\ln(1-16z^2)\right) +\frac{4}{\pi}\ln(2) -1
\end{equation}
\end{definition}
\begin{remark}
For $m \geq 0$ we have 
\begin{equation}\label{eq_komega}
\frac{\partial}{\partial \omega}\widetilde{K}(m,k,\omega) = \widetilde{K}(m+1,k,\omega)
\end{equation}
which is easily verified by basic algebra.
\end{remark}
\begin{lemma}\label{le_sfk}
For $F \in \tilde H_r$ and $z \rightarrow \frac{1}{4}$ the equation
\begin{equation}
\hat S_{F,k}(z) - \widetilde{S}_{F,k}(z) \sim o\left(\frac{1}{s^{r-1}}\cdot s \cdot \ln(s)^{K(F) + 1}\right)
\end{equation}
and the corresponding equation for $z \rightarrow -\frac{1}{4}$ are true and therefore $\widetilde{S}_{F,k}(z)$ contains the full information for the asymptotic behaviour in $n$ of the Taylor coefficients of $\hat S_{F,k}(z)$ around $z = 0$ with all corrections up to order $n^{r-3}\ln(n)^{K(F)+1}$.
\end{lemma}
\begin{proof} The Lemma for $z \rightarrow \frac{1}{4}$ follows from Lemma \ref{le_integr} for the factor $I_{\Sigma}(F)$ and equation \eqref{eq_A3_even} for replacing $h(0,2,z)$ with $\widetilde{h}(0,2,z)$ in the vertex factors in the definition of $ \hat S_{F,k}(z)$ in equations  \eqref{eq_sfk}, \eqref{eq_dfk2} and \eqref{eq_dfk1}\\
For $z \rightarrow -\frac{1}{4}$ we observe: From Lemma \ref{le_ceven} $C_{F,k}(Y,z)$ is even in $z$ and therefore also $\hat S_{F,k}(z)$. $\tilde S_{F,k}(z)$ is even in $z$ by definition.
So the Lemma follows also for $z \rightarrow -\frac{1}{4}$. 
\end{proof}
\begin{lemma}\label{le_momfunc}
We define the moment functions 
\begin{equation}
\hat T_k(z) = \underset{N \rightarrow \infty}{\lim}\left( \sum_{w \in W_{2N}}\left(\prod_{i = 1}^r N_{2k_i}(w) \right)\cdot \frac{z^{length(w)}}{lenght(w)} \right)
\end{equation}
and 
\begin{equation}\label{eq_deftk}
T_k(z) := z\frac{\partial}{\partial z} \hat T_k(z) = \underset{N \rightarrow \infty}{\lim}\left( \sum_{w \in W_{2N}}\left(\prod_{i = 1}^r N_{2k_i}(w) \right)\cdot z^{length(w)} \right)
\end{equation}
Then for $z \rightarrow 1/4$ the equation
\begin{equation}\label{eq_defsum}
\hat T_k(z) -  \left(\sum_{F \in \tilde H_k^{(M)}} \mathcal{T}_{F,k}(z) \right) \sim o\left(\frac{1}{s^{r-1} \ln(s)^M} \right)
\end{equation}
is true and a corresponding equation for $z \rightarrow -1/4$. Termwise differentiation of the asymptotic expansion \eqref{eq_defsum} in any order is allowed for $z \in U_{\frac{1}{2d}}(0)$ and interchangeable with asymptotic expansion.  
\end{lemma}
\begin{proof}
The function $T_k(z)$ as $S_k(z)$ is an even function in $z$ as the walks $w$ contributing to it are closed. The functions $\mathcal{T}_{F,k}(z)$ are also even following the same reasoning as in Lemma \ref{le_ceven}. We therefore can limit the discussion to $z \rightarrow 1/4$. From the definition of $P$ we know that  
\begin{multline}
P(N_{2k_1}(w),\ldots,N_{2k_r}(w)) = \frac{\#D(k_1,\ldots,k_r)}{r\,!}\left(\prod_{i=0}^r N_{2k_i}(w) \right) \\ + \textrm{terms with $r-1$ or less factors $N_{2k_j}(w)$ }
\end{multline}
For $\bar k$ containing $r-1$ or less factors $k_j$  $\hat S_{\bar k}(z)$ for $z \rightarrow 1/4$ is $o\left(\frac{1}{s^{(r-2)}} \right)$ according to Lemma \ref{le_integr} and Theorem \ref{th_nge}. So for $d=2$ we see that 
\begin{equation}
\hat S_k(z) - \frac{\#D(k_1,\ldots,k_r)}{r\,!}\cdot \hat T_k(z) \sim o\left(\frac{1}{s^{(r-2)}}\right)
\end{equation}
But on the other hand it is immediately clear that 
\begin{equation}
\tilde S_{F,k}(z) =  \frac{\#D(k_1,\ldots,k_r)}{r\,!}\cdot \mathcal{T}_{F,k}(z)
\end{equation}
and so the \eqref{eq_defsum} follows from Theorem \ref{th_nge} and Lemma \ref{le_sfk}. The functions $\mathcal{T}_{F,k}(z)$ are holomorphic on the convex open set $U_{\frac{1}{2d}}(0)$ and 
\begin{equation}
\frac{\mathcal{T}_{F,k}(z)}{\frac{d}{dz} \mathcal{T}_{F,k}(z)}  = \frac{1- 16 z^2}{8\cdot (r-1)}   \left( 1 + O \left(\frac{1}{\ln(1-16z^2)} \right) \right)
\end{equation}
and $T_k(z)$ is holomorphic on $U_{\frac{1}{2d}}(0)$ and therefore termwise differentiation is allowed and interchangable with asymptotic expansion.
\end{proof}
\subsection{Reducing the graphs to be summed over}
In this subsection we partition the set of all Eulerian graphs to the subset of the dam free graphs (graphs which do not have vertices of indegree $1$ and outdegree $1$), from which the others can be deduced both in the sense of graph theory and in the sense of the graph contributions in our calculations. This reduction is the key to getting the leading behaviour of the centralized moments. Centralizing the moments and summing over the dams leads to the important cancellations which result in a consistent rescaling in leading order for all centralized moments.
\begin{definition}\label{de_dam}
We define 
\begin{equation}
\tilde H := \underset{r \geq 2 }\cup \tilde H_r
\end{equation}
We now define graphs with dams (dam) and those with no dams (nd). 
\begin{equation}
\begin{array}{rcl}
\tilde H_{nd} &= & \{F \in \tilde H;  \forall_{j}: \sum_{i\geq1}F_{i,j} = h_j \neq 1\}\\
\tilde H_{dam} &= & \{F \in \tilde H;  \exists_{j}: \sum_{i\geq1}F_{i,j} = h_j = 1\}
\end{array}
\end{equation}
and therefore 
$\tilde H = \tilde H_{nd} \dot{\cup} \tilde H_{dam}$.\\
For $f \in \N\setminus\{0\}$ we define
\begin{equation}\label{eq_ff}
F_f := \bigl(\begin{smallmatrix}
0&f\\
f&0
\end{smallmatrix}\bigr)
\end{equation}
Then obviously $\tilde H_2 =\{F_f, f \in \N\setminus \{0\}\}$ and $\{F_1\} = \tilde H_2 \cap \tilde H_{dam}$
\end{definition}
\begin{definition}
Let $F \in \tilde H_{dam} \cap \tilde H_r$ be an $r \times r $ matrix with $r > 2$. Then there is a smallest index $i_0 = min(i: \sum_{j\geq1} F_{i,j} = 1)$. We define the $(r-1)\times (r-1)$ less dam matrix $F_{ld}$ by 
\begin{equation}\label{eq_ltpdef}
(F_{ld})_{i,j} := F_{i+ \theta(i,i_0),j+ \theta(j,i_0)} + F_{i+ \theta(i,i_0),i_0}\cdot F_{i_0,j+ \theta(j,i_0)}
\end{equation} 
where 
\begin{equation}
\theta(a,b) = 
\begin{cases} 0 \Longleftarrow a < b\\
1 \Longleftarrow a \geq b \end{cases} \\
\end{equation}
\end{definition}
\begin{lemma}\label{le_ltp}
Let $F \in \tilde H_{dam}\setminus \{F_1\}$ and $F_{ld}$ its less dam matrix. Then $F_{ld} \in \tilde H$ and 
\begin{equation}\label{eq_tpcof}
cof(A-F) = cof(A_{ld} - F_{ld})
\end{equation}
\end{lemma}
\begin{proof} 
 We denote with $i_1$ the index for which $F_{i_1,i_0} = 1$ and with $j_1$ the index for which  $F_{i_0,j_1} = 1$ and define the notation that $i_0$ is a dam between $i_1$ and $j_1$. We say that a dam is of class A if $j_1 = i_1$ and of class B otherwise. 
 We define $i_2$ such that $i_2 + \theta(i_2,i_0) = i_1$ and $j_2 + \theta(j_2,i_0) = j_1$.
 We first proove that $F_{ld}$ is balanced. We do this seperately for class A and class B. \\ We  start with class A. $F$ was balanced. $F_{i_0,j} = 0$ for $j \neq i_1$  and $F_{i,i_0} =  0$ if $i \neq i_1$. So any row and column other than the $i_1 $ th of $F$ has the same sum  as the corresponding row and column of $F_{ld}$. The $i_1$ th row and column of $F$ both have a  sum which is bigger by $1$ than the corresponding one of $F_{ld}$ and so $F_{ld}$ is  balanced.\\
 For class B the rowsum of the $i + \theta(i,i_0)$ th row of $F$ is the same as that of the $i$th row  of $F_{ld}$, which is trivial for $i + \theta(i,i_0) \neq i_1$.  For the $i_2$th row of $F_{ld}$  because of  $F_{i_1,i_0} =1$ there is an additional $1$ in the rowsum of $F$ but there is also  an additional $1$ in the rowsum of the corresponding row of $F_{ld}$ as for $j=j_2$ we find  from equation \eqref{eq_ltpdef}
 \begin{equation}
 (F_{ld})_{i_2,j_2} = F_{i_1,j_1} +1 
 \end{equation}
 By the same argument the columnsum of corresponding columns of $F$ and $F_{ld}$ again is  the same and so $F_{ld}$ is balanced for class B.\\
 If we  delete the $i_1$ th row and $j_1$ column of $U = A - F$ we get a matrix $U_1$ which  has only one $1$ in the row inherited from the $i_0$ th row of $U$ before deleting the $i_1$  row and a column which only has the same $1$ and the rest $0$s which was  inherited from the  $i_0$ th column before deleting the $j_1$ th column of $U$. So this row and column can be  deleted without changing the  value of the cofactor according to Laplace. Let us call the matrix which  results from $U_1$ by this operation $U_2$. For the matrix $V=A_{ld} - F_{ld}$ we  delete the $i_2$th row and $j_2$th column and call the result $V_1$. We then see $V_1 =  U_2$ and so equation \eqref{eq_tpcof} is true as $cof(A_{ld} - F_{ld}) = \det(V_1)\cdot (-1)^{i_2 + j_2}$.  
\end{proof}
\begin{definition}
We define the mapping 
\begin{equation}
\begin{array}{rcccl}
g_{dam}: & \tilde  H_{dam} \setminus\{F_1\}  & \mapsto & \tilde H \\
&F & \mapsto & F_{ld} 
\end{array}
\end{equation}
and inductively the sets $\tilde H_1(F) := g_{dam}^{-1}(F)$ and $\tilde H_m(F) := g_{dam}^{-1}(H_{m-1}(F))$ \\
%We also define $\tilde H_{m,B}(F)$ to be the subset of $\tilde H_{m}(F)$ such that each of the %$m$ dams is of type $B$ 
\end{definition}
\begin{lemma}\label{le_nontp}
For any matrix $F \in \tilde H_{dam}$ there is a number $m(F)$ and either a matrix $F_{nd} \in \tilde H_{nd}$ such that $F \in \tilde H_{m(F)}(F_{nd})$ or $F \in \tilde H_{m(F)}(F_1)$.   
\end{lemma}
\begin{proof}
For any $r \times r$ matrix $F \in \tilde H_{dam}$ we apply $g_{dam}$ as often as possible (e.g. $m$ times). Then the matrix $g_{dam}^{(m)}(F)$ is either a matrix without dams, or a $2 \times 2$ matrix with dams and therefore $F_1$.   
\end{proof}
\begin{definition}
Let $F \in  ( \tilde H_{nd} \cup \{F_1\}) \cap \tilde H_r $  and $\tilde F \in \tilde H_k(F)$. For $j \in \{1,\ldots,r\}$ we define the function $ind(i_j,k) := j$ if $g_{dam}^{(k)}$ maps the remainder of the $i_j$ th column and row of $\tilde F$ onto the $j$ th column and row of $F$.  
\end{definition}
\begin{lemma}\label{le_step}
Let $F \in  ( \tilde H_{nd}  \cup \{F_1\}) \cap \tilde H_r $ and $\tilde F \in \tilde H_m(F)$. Then $\tilde F$ is uniquely defined by $m$ steps starting from $F$. \\
The $k$ th step is defined inductively by a number $i_k \in \{1,\ldots,r+k\}$  and a pair of numbers $(i,j)$, where $i,j \in \{1,\ldots,r+k-1\}$ are the indices $ind(\alpha_1,k-1)$ and $ind(\alpha_2,k-1)$ if the $k$ th step is a dam between $\alpha_1$ and $\alpha_2$ (not necessarily different). We also define $ind(i_k,k) := r+k$, the $k$ th added dam. For any $i \neq i_k$ we define $ind(i +\theta(i,i_k),k) = ind(i,k-1)$.
\end{lemma}
\begin{proof}
Any dam is inserted between two uniquely defined dimensions $\alpha_1$ and $\alpha_2$. We have just given them a unique index which does not change with the application of $g_{dam}$. The definition of the index is obviously consistent with $g_{dam}$.
\end{proof}
\begin{definition}
Let $F \in  ( \tilde H_{nd}  \cup \{F_1\}) \cap \tilde H_r $ and $\tilde F \in \tilde H_m(F)$. Inductively we define a label function $lab$ on $i \in \{r+1,\ldots,r+m\}$ in the following way: If in the $k$ th step in Lemma \ref{le_step} $i_k$ was a dam of class A between $\alpha_1$ and $\alpha_1$, then we define $lab(r+k) := (ind(\alpha_1,k-1),ind(\alpha_1,k-1))$.\\ 
If the dam was of class $B$ between $\alpha_1$ and $\alpha_2$ we define $lab(r+k) := (ind(\alpha_1,k-1),ind(\alpha_2,k-1))$ if $ind(\alpha_1,k-1) \leq r$ and $ind(\alpha_2,k-1) \leq r$, i.e. the dam is between two dimensions of $F$ loosely speaking. In case  $ind(\alpha_1,k-1) > r$ or $ind(\alpha_2,k-1) > r$  we define $lab(r+k) := lab(max(ind(\alpha_1,k-1), ind(\alpha_2,k-1))).$
\end{definition}
\begin{lemma}\label{le_label}
Let $F \in (\tilde H_{nd}  \cup \{F_1\})\cap \tilde H_r $ and $\tilde F \in \tilde H_m(F)$. Let $G(\tilde F)$ be the multidigraph which has the indices $ind(i,m)$ as vertices. Then one of the following alternatives is true:  
\begin{enumerate}
\item if $lab(r+k) = (i,j)$ with $i \leq r$ and $j \leq r$ then the vertex $r + k$ is on a subgraph which is a path graph from $i$ to $j$ and all the vertices of this subgraph also have the label $(i,j)$
\item if $lab(r+k) = (i,i)$ then the vertex $r+k$ is on a subgraph which is a cycle starting and ending in $i$ and all vertices of this subgraph other than $i$ also have the same label $(i,i)$. If $i >r$ then $r+k  > i$.
\end{enumerate}
Moreover  if there is an edge between vertex $r+k$ and vertex $j$ and $lab(r+k) \neq lab(j)$ then if $r +k > j$ we have $lab(r+k) = (j,j)$ or (not exclusive) $j <  r$ and $lab(r+k) = (a,b)$ with $a = j$ or $b=j$.
\end{lemma}
\begin{proof} 
By induction. For $m = 1$ the Lemma is obviously true. Let it now be true for $m-1$. If in the $m$ th step we have a dam of class A, then obviously the Lemma again is true as the new vertex $r+m$ gets a label $(i,i)$ for some $i < r+m$. If in the $m$th step we have a dam of class $B$ inserted into an edge between the vertices $i$ and $j$ (assuming $i \geq j$) then:
\begin{enumerate}
\item if $i \leq r$ and $j \leq r$ then the Lemma is obviously true as $lab(r+m) = (i,j)$.
\item if $i > r$ and $lab(i) = (a,b)$ with $a \leq r$ and $b \leq r$ then according to induction:
Either $lab(j) = (a,b)$ and therefore both are on a subgraph which is a path between $a$ and $b$ and therefore $r+m$ which gets the label $(a,b)$ too fulfils the Lemma. 
Or $j < r$ and therefore $b = j$ according to induction and $lab(r+m) = (a,b)$ and again the Lemma is true.
\item the case $i \leq r$ and $j > r$ was exclude by the assumption $i \geq j$ (it is equivalent to the last case).
\item if $i > r$ and $lab(i) = (c,c)$ and $lab(i) \neq lab(j)$ then according to induction $lab(i) = (j,j)$ or $i = j$. But then $lab(r+m) = (j,j)$ and obviously again $r+m$ is on a subgraph which is a cycle with starting and ending point $j$ and the correct labels according to the Lemma. The case $lab(j) = lab(i)$ obviously fulfils the Lemma too.
\end{enumerate}
\end{proof}
\begin{definition}
Let $F \in \tilde (H_{nd} \cup \{F_1\}) \cap \tilde H_r $ and $\tilde F \in H_m(F)$. For $F \neq F_1$ we define a digraph $Tr(\tilde F) = (V,E)$ in the following way: Let $V = \{0,1,\ldots,m\}$. We define $e \in E$ by the following rule: If in $G(\tilde F)$ a vertex $r+k$ has  $lab(r + k) = (i,j)$ with $i \leq r$ and $j \leq r$ then there is and edge from $0$ to $k$ in $Tr(\tilde F)$. If a vertex $r+k$ has $lab(r+k) = (i,i)$ with $i \leq r$ then there is also a vertex from $0$ to $k$. If a vertex $r+k$ has $lab(r+k) = (i,i)$ with $i > r$ then there is an edge from $i-r$ to $k$. \\
For $F = F_1$ we define a  digraph $Tr(\tilde F) = (V,E)$ in the following way: Let $V = \{1,\ldots,m+2\}$.We define $e \in E$ by the following rules:  There is an edge from vertex $1$ to vertex $2$ in $Tr(\tilde F)$. If in $G(\tilde F)$ a vertex $k+2$ has $lab(k+2) = (i,j)$ with $i \leq 2$ and $j \leq 2$ then there is an edge from $1$ to $k+2$ in $Tr(\tilde F)$. If a vertex $k+2$ has $lab(k + 2) = (i,i)$ then there is a vertex from $i$ to $k+2$.\\
\end{definition}
\begin{lemma}\label{le_trlab}
$Tr(\tilde F)$ is a tree. 
\end{lemma}
\begin{proof}
The graph $Tr(\tilde F)$ is connected by induction: For $m = 1$ this is trivial. Putting in a new dam in the $m$ th step connects the vertex $m$ of $Tr(\tilde F)$ with a vertex $k \in \{0,\ldots,m-1\}$ for $F \neq F_1$ and $k \in \{1,\ldots,m+1\}$ for $F = F_1$  determined by the label of the vertex $r+m$ in $G(\tilde F)$. Therefore by induction the vertex  $m$ of $Tr(\tilde F)$ is connected  with a connected graph.\\
For $F \neq F_1$ any vertex other than $0$ and for $F = F_1$ any vertex other than $1$ has just one ingoing edge well defined by the label. Therefore $\#E = \#V -1$ and as $Tr(\tilde F)$ is connected it is a tree.   
\end{proof}
\begin{definition}
We define: \\
A path dressing of an edge $e$ of a graph with an ordered sequence of vertices $v_1,\ldots,v_k$ (which are not vertices of the graph) is the replacement of $e$ from vertex $v_a$ to vertex $v_b$ in the graph with new edges $e_1,\ldots,e_{k+1}$ such that $e_i$ is an edge from $v_{i-1}$ to $v_i$ for $i=2,\ldots,k$ and $e_1$ is an edge from $v_a$ to $v_1$ and $e_{k+1}$ is an edge from $v_k$ to $v_b$. For completeness sake we define the dressing of an edge $e$ with the empty sequence as not changing anything.\\
A cycle dressing of a vertex $v_a$ of a graph with an ordered sequence of vertices $v_1,\ldots,v_k$ (which are not vertices of the graph) is defined as the addition of a cycle graph which starts and ends in $v_a$ with vertices $v_1,\ldots,v_k$ in this order. Again a cycle dressing of a vertex $v_a$ with the empty sequence is just defined as not changing anything.\\
A dressing of vertex $v_a$ with nonoverlapping sequences of vertices is defined as the combined cycle dressing of $v_a$ with the  sequences.  \\
A path dressing of a graph with nonoverlapping (not necessarily nonempty) sequences of vertices for each edge is defined as the combined path dressing of the edges with the corresponding sequences.\\
A dressing of a graph with ordered (not necessarily nonempty) sequences of vertices for each vertex of the graph (more than one nonempty sequence allowed for a vertex) and (not necessarily nonempty) sequences of vertices for each edge (only one nonempty sequence allowed for an edge) is defined as the combined vertex and path dressing of those sequences.  The sequences are supposed to be nonoverlapping. \\
\end{definition}
\begin{definition}
Let first $G$ be a graph which is not an isolated vertex and $\{v_1,\ldots,v_m\}$ a set of vertices which are not vertices of the graph. Let $T$ be a tree with  the vertices $0,\ldots,m$.
Let us have a partitioning of the set $\{v_1,\ldots,v_m\}$ into $m + 1 $ subsets $A_k$ with $\#A_k = deg(k) -1 + \delta_{k,0}$ where $deg(k)$ here is the degree of vertex $k$ in the Tree $T$.\\
If $G$ is an isolated vertex $v_1$ on the other hand let there be a set of other vertices $\{v_2,\ldots,v_{m+2}\}$. Let $T$ in this case be a tree with the vertices $1,\ldots,m+2$. 
Let us have a partitioning of the set $\{v_2,\ldots,v_m+2\}$ into $m + 2 $ subsets $A_k$ with $\#A_k = deg(k) -1 + \delta_{k,1}$  where $deg(k)$ here is the degree of vertex $k$ in the Tree $T$. Let also $v_2 \in A_1$.
\\ 
  Let each of the subsets $A_k$ be completely partitioned into ordered sequences. We then call the graph $\tilde G$ of  $G$ hierarchically dressed with these ordered sequences according to $T$ if we do the following steps inductively with G:
\begin{enumerate}
\item If $G$ is not an isolated vertex dress $G$ with the sequences belonging to the subset $A_0$. If $G$ is an isolated vertex $v_1$ dress it with the sequences belonging to the subset $A_1$. 
\item dress the resulting graph in the vertex $v_i$  ($v_i \neq v_1$ for $G$ being the isolated vertex) with the seqences belonging to the subset $A_i$. 
\item continue the last step with the resulting graph and the vertices $v_l$ added in the step before till only empty sets $A_l$ are left. 
\end{enumerate}  
\end{definition}
\begin{lemma}\label{le_hmchar}
Let $F \in \tilde H_{nd} \cap \tilde H_r$. Then any $\tilde F \in H_m(F)$ is completely characterized by the following elements: 
\begin{enumerate}
\item a selection of dimensions $i_1,\ldots,i_m$ out of $r+m$ dimensions as dimensions of the dams.
\item A tree $T$ on the set $\{0,\ldots,m\}$.
\item A hierarchical dressing of $G(F)$ according to $T$ with a partitioning of the set $\{i_1,\ldots,i_m\}$ as set of new vertices.
\end{enumerate} 
I.e. any set of selected dimensions, a tree of the given form and a hierarchical dressing results in a graph $G(\tilde F)$ with $\tilde F \in H_m(F)$ and any $\tilde F \in H_m(F)$ leads to a set of selected dimensions, a tree $T$ of the given form and a hierarchical dressing according to $T$.  
\end{lemma}
\begin{proof}
We first of all prove that the selection of the dimensions, a tree $T$ of the given form and a hierarchical dressing results in a graph $G(\tilde F)$ with $\tilde F \in H_m(F)$. We prove this by induction in $m$. For $m=1$ it is trivially true. So let $m > 1$. The tree $T$ contains one or more leafs. Each leaf in $T$ by construction corresponds to a vertex in the graph $G(\tilde F)$ which is a dam. We choose the leaf $k$  of $T$ such that the corresponding vertex $i_k$ of $G(\tilde F)$ is minimal among all leafs of $T$. Then this $i_k$ is minimal among all dimensions of $\tilde F$ such that the rowsum $h_{i_k} = 1$ as by construction any vertex $p$ in the tree $T$ which is not a leaf corresponds with a vertex $i_p$ in $G(\tilde F)$ which has at least one cycle added and is in a cycle or a path subgraph which was added to $G(F)$ and therefore has a degree which is bigger than $1$ by construction. So $g_{dam}(\tilde F)$ is a matrix with row and column $i_k$ removed. To it belongs  a tree $\hat T$ which after the change of the indices according to $g_{dam}$ has the leaf corresponding to $i_k$ removed and a hierarchical dressing according to the new data. Therefore according to induction $g_{dam}(\tilde F) \in H_{m-1}(F)$. \\
On the other hand with Lemma \ref{le_label} and \ref{le_trlab} for any $\tilde F \in H_m(F)$ we have defined a choice of dimensions and constructed the additional cycle and path graph additions to $G(F)$ and the tree $T$ such that the Lemma is true. With this Lemma we have completely characterized $H_m(F)$ because any choice of dam dimensions, a tree $T$ and dressing with ordered sequences defines a different matrix $\tilde F$. 
\end{proof}
\begin{lemma}\label{le_f1hmchar}
Let $F = F_1$. Then any $\tilde F \in H_m(F)$ is completely characterized by the following elements: 
\begin{enumerate}
\item A tree $T$ on the set $\{1,\ldots,m+2\}$.
\item A hierarchical dressing of the point graph with vertex $1$ according to $T$ with a partitioning of the set $\{2,\ldots,m+2\}$ as set of new vertices.
\end{enumerate} 
I.e. any tree of the given form and a hierarchical dressing results in a graph $G(\tilde F)$ with $\tilde F \in H_m(F_1)$ and any $\tilde F \in H_m(F_1)$ leads to a tree $T$ of the given form and a hierarchical dressing according to $T$. 
\end{lemma}
\begin{proof}
We first of all prove that a tree $T$ of the given form and a hierarchical dressing results in a graph $G(\tilde F)$ with $\tilde F \in H_m(F_1)$.  We prove this by induction in $m$. For $m=1$ it is trivially true. So let $m > 1$. The tree $T$ contains one or more leafs. Each leaf in $T$ by construction corresponds to a vertex in the graph $G(\tilde F)$ which is a dam. We choose among the leafs $k$  of $T$ such that it is minimal among all leafs of $T$. Then this $k$ is minimal among all dimensions of $\tilde F$ such that the rowsum $h_{k} = 1$ as by construction any vertex $p$ in the tree $T$ which is not a leaf corresponds with a vertex $p$ in $G(\tilde F)$ which has at least one cycle added and is itself in a cycle. So $g_{dam}(\tilde F)$ is a graph with $k$ removed. So after the removing of the leaf $k$ from $T$ we are left with a tree $\hat T$ which after the change of the indices according to $g_{dam}$ and a hierarchical dressing according to the new data. Therefore according to induction $g_{dam}(\tilde F) \in H_{m-1}(F)$. We have also seen that the tree $T$ completely determines the sequence of indices $i_p$ which correspond to the dimensions removed in the $p$ th application of $g_{dam}$ and therefore we have no free choice of dimensions as we had in the case of $F \neq F_1$ for any vertex other than the first two.  But it is also easy to see that $g_{dam}$ cannot remove the vertex with the highest index, as a tree $T$ always has at least two leafs, so the vertex with the highest index is mapped from $m$ to $m-1$ etc. and finally onto $2$ in the process of applying $g_{dam}$ and reordering the dimensions of $g_{dam}^{(k)}(\tilde F)$ and therefore the whole sequence is completely determined by $T$.
\\
On the other hand with Lemma \ref{le_label} and \ref{le_trlab} for any $\tilde F \in H_m(F_1)$ we have defined and constructed the additional cycle and path graph additions to the point graph $1$ and the tree $T$ such that the Lemma is true. With this Lemma we have completely characterized $H_m(F_1)$ because any tree $T$ and dressing with ordered sequences defines a different matrix $\tilde F$ and each matrix $\tilde F$ is completely characterized by $T$ and the hierachical dressing according to it.  
\end{proof}
\begin{remark}
 By the Markov feature 
\begin{equation}\label{eq_markov}
z \cdot \frac{\partial}{\partial z} \left( h(x,d,z) \right) = \sum_{y \in \Z^d} h( y,d,z) \cdot  h(x - y,d,z)
\end{equation}
and therefore by induction 
\begin{multline}\label{eq_hxmarkov}
\left(z \cdot \frac{\partial}{\partial z} \right)^k \left( h(x,d,z) \right) = \sum_{y_i \in \Z^d} \sum_{\sigma \in S_k} h(y_{\sigma(1)},d,z) \cdot h(x - y_{\sigma(k)},d,z) \cdot \\ \prod_{j=2}^k h(y_{\sigma(j)} - y_{\sigma(j-1)},d,z)
\end{multline} 
\end{remark}
\begin{lemma}\label{le_rd1}
Let $F \in \tilde H_{nd}\cap \tilde H_r$ and $k = (k_1,\ldots,k_{r+m})$ a vector of strictly positive integers.
Then 
\begin{multline}\label{eq_rd1}
\sum_{\tilde F \in \tilde H_{m}(F)} \mathcal{T}_{\tilde F,k}(z) = \binom{r+m}{m} \cdot cof(A-F) \cdot
\\  \left( z\cdot \frac{\partial}{\partial z}\right)^{m-1}
 \Biggl[ \frac{1}{(r+m)\,!} \sum_{\sigma \in S_{r+m}} 
\left( \prod_{j=r+1}^{r+m} \widetilde{K}(1,k_{\sigma(j)},\widetilde{h}(0,2,z) \right)   \\
 \cdot \left( z\cdot \frac{\partial}{\partial z}\right)\left( \frac{I(F) \cdot M(F) }{2^{r-1}\cdot (1-16z^2)^{r-1}} \cdot \left(
\prod_{j=1}^r  \widetilde{K}(h_j,k_{\sigma (j)},\widetilde{h}(0,2,z))\right) \right)\Biggr] \\
+ o\left(\frac{1}{s^{r+m-2}} \right)
\end{multline}
\end{lemma}
\begin{proof}
From equation \eqref{eq_hxmarkov} we see that the sum of the contributions of all matrices which we get from a matrix $F$ by applying all possible path dressings of an edge between $i$ and $j$ in $G(F)$ is equivalent to taking the contribution of $F$ and replacing $h(Y_i - Y_j,2,z)$ with $(z\cdot \partial/\partial z)^k h(Y_i - Y_j,2,z)$ before the summation over $Y_i$ and $Y_j$, up to terms scaled down by at least one factor $s$ because of the summation constraint $y_i \neq y_j$ for the variables of the dressing vertices relative to the free summation in equation \eqref{eq_hxmarkov} . (Of course we have to also multiply it with the corresponding vertex factors $\tilde K$ and average over the bigger set of permutations, we will in the sequel assume this).\\ 
The sum of the contributions of all matrices which we get from a matrix $F$ by applying all possible cycle dressings of a vertex $i$ in $G(F)$ with $k$ vertices is given by replacing the factor $\tilde K(h_i,k_{\sigma(i)},h(0,2,z))$ with 
\begin{equation}
\left(\frac{\partial}{\partial z} \right)^k \tilde K(h_i,k_{\sigma(i)},h(0,2,z)) 
\end{equation}
up to contributions which are scaled down by a factor $s$ because of summing constraints. 
This is true because of equation \eqref{eq_komega} and equation \eqref{eq_hxmarkov} for $x = 0$ which describes the contribution of a cycle dressing in this case.\\ 
Therefore the dressing of a graph $G(F)$ with $k$ vertices leads to a contribution 
\begin{equation}
\left(\frac{\partial}{\partial z} \right)^k (gf(\sigma,z))
\end{equation}
up to terms scaled down by at least a factor $s$, where
\begin{multline}
gf(\sigma,z) = 
cof(A-F)  \cdot \\
 \left( \frac{I(F) \cdot M(F) }{2^{r-1}\cdot (1-16z^2)^{r-1}} \cdot \left(
\prod_{j=1}^r  \widetilde{K}(h_j,k_{\sigma (j)},\widetilde{h}(0,2,z))\right) \right)
\end{multline}
The hierarchical dressing according to a tree $T$ with vertex degrees $\lambda_0,\lambda_1+1,\ldots,\lambda_m+ 1$ where $\lambda_0 > 0$  therefore up to contributions scaled down at least  by a factor $s$ from summing constraints is given by 
\begin{equation}
\left(\frac{\partial}{\partial z} \right)^{\lambda_0} (gf(\sigma,z)) \prod_{j=1}^{m} \left(\frac{\partial}{\partial z} \right)^{\lambda_j} K(1,k_{\sigma(r+j)},h(0,2,z)) 
\end{equation}
Now according to Lemma \ref{le_hmchar} we have to multiply this with a factor for the choice of the dimensions which are dimensions of dams
\begin{equation}
\binom{r+m}{m}
\end{equation}
and the number of trees with this given set of degrees such that  
\begin{equation} 
\sum_{i=0}^m \left( \lambda_i + 1  - \delta_{0,i} \right) = 2m
\end{equation} 
The calculation of the number of the trees with these constraints fortunately also comes up in a standard proof of Cayley's theorem on the number of trees \citep{bryant}. It is given by
\begin{equation}
\frac{(m-1)\,!}{(\lambda_0 - 1)\,!\prod_{i=1}^m \lambda_i\,!}
\end{equation}
But according to the general Leibniz rule for multiple differentiations of products we then get equation \eqref{eq_rd1}. Differentiation of the asymptotic expansion in any order is allowed, as we have discussed in Lemma \ref{le_momfunc} as each term in the finite sum on the left hand side of equation \eqref{eq_rd1} is a holomorphic function on $U_{\frac{1}{2d}}(0)$. 
\end{proof}
\begin{lemma}\label{le_rd2}
Let $F_1$ be defined according to equation \eqref{eq_ff} and $k = (k_1,\ldots,k_{m+2})$ a vector of strictly positive integers.
Then 
\begin{multline}\label{eq_rd2}
\sum_{\widetilde{F} \in \tilde H_{m}(F_1)} \mathcal{T}_{\tilde F,k}(z) =
\\ \left( z\cdot \frac{\partial}{\partial z}\right)^{m}
\Biggl[ \frac{1}{(m+2)\,!}\sum_{\sigma \in S_{m+2}} 
\left( \prod_{j=1}^{m+2} \widetilde{K}\left(1,k_{\sigma(j)},\widetilde{h}(0,2,z)\right) \right) \\
 \cdot \left( z\cdot \frac{\partial}{\partial z}\right)\widetilde{h}(0,2,z)  \Biggr] 
+ o\left(\frac{1}{s^{m}}  \right)
\end{multline}
\end{lemma}    
\begin{proof} With Lemma \ref{le_f1hmchar} the proof is analogous to the proof of Lemma \ref{le_rd1}, realizing the we do not have the binomial choice factor and taking into account that $M(F_1) = 1$  and $cof(A_1 - F_1) = 1$. 
\end{proof}
\section{The moments}
In this section we will calculate the expectation value of the moments with all its logarithmic corrections and give examples for the first and second moment. We will then calculate the leading behaviour for the centralized moments and all its logarighmic corrections for any $r$ and get the joint distribution for the closed simple random walk from that. 
\subsection{The asymptotic expansion}
We observe from Lemma \ref{le_momfunc} that the constituent functions in equation \eqref{eq_sckf} and \eqref{eq_h0ti} all have a specific form which leads to a general scheme for calculating the asymptotic expansion of them. Their sum in \eqref{eq_defsum} differentiated as shown in \eqref{eq_deftk} amounts to the expectation value up to a division by the number of all closed walks of a given length. 
\begin{lemma}\label{le_yexp}
Let $g(x) = \sum_{j=0}^{M} b_j \cdot x^{j}$ be a polynomial of degree $M$. Then the function \begin{equation}
f_m(\varpi) := \frac{1}{(1-\varpi)^m}\cdot g\left(\frac{1}{\ln(1-\varpi)}\right)
\end{equation}
around $\varpi = 0$ has a Laurent series with poles of highest order $M$ and the sum  
\begin{equation}
f_m(\varpi) = \sum_{n=-M}^{\infty}c_n(m) \cdot \varpi^n 
\end{equation}
converges for $\varpi \in U_1(0)\setminus \{0\}$ and for $c_n(m)$ we have the asymptotic expansion
\begin{multline}\label{eq_coeff}
c_n(m) =  n^{m-1} \cdot \Biggl[ \sum_{i=0}^{M} \gamma^{(m)}_i  y^{i+1} \left(\frac{d}{dy}\right)^i \left(y^{i-1} g(y) \right)\Biggr|_{y =- \frac{1}{\ln(n)}} \\ + o\left(\frac{1}{\ln(n)^M} \right)\Biggr]
\end{multline}
where the coefficients $\gamma^{(m)}_i$ belong to the expansion 
\begin{equation}
\frac{1}{\Gamma(m + \tau)} = \sum_{i=0}^{\infty}  \gamma^{(m)}_i \cdot \tau^i
\end{equation}
(which has an infinte radius of convergence \citep[eq. 6.1.34]{danraf}). 
\end{lemma}
\emph{Proof:} Because of the power series of $\ln(z)$ around $z = 1$ it is clear that $f_m$ has the Laurent series with a pole of highest order $M$ in $\varpi = 0$ and converges for  $\varpi \in U_1(0)\setminus \{0\}$. \\
For $m \in \N \setminus\{0\}$ and $j \in \N_0$ we now consider the functions
\begin{equation}
\beta_{m,j}(x) := \frac{1}{(1- x)^m \cdot \ln(1-x)^j} 
\end{equation}
They have a Laurent series around $x=0$ with 
\begin{equation}
\beta_{m,j}(x) = \sum_{n=-j}^{\infty} \theta_{j,m,n} \cdot x^n
\end{equation}
For $\beta_{m,0}(x)$ we realize that $\theta_{0,m,n} = n^{m-1}/\Gamma(m) (1 + O(1/n))$ and so equation \eqref{eq_coeff} is true for the constant term in the polynomial $g$ as in the sum over $i$ only the contribution with $i=0$ is different from $0$. 
We then note that for $j > 0$  
\begin{equation}\label{eq_funct}
\frac{d}{dx} \beta_{m,j}(x) = m\cdot \beta_{m+1,j}(x) + j\cdot \beta_{m+1,j+1}(x)
\end{equation}
To find the asymptotic behaviour of the coefficients of $\beta_{m,1}(x)$ for $x > 0$ we write for $k \in \N$ and $k > 1$ 
\begin{equation}
\beta_{m,1}(x) - \beta_{-k,1}(x) = - \int_0^{m+k}  e^{-\ln(1-x)\cdot (m-t)} \cdot dt = -\int_0^{m+k} (1-x)^{-(m-t)} \cdot dt
\end{equation}
Of course $\beta_{-k,1}(x)$ has Taylor coefficients $o\left(n^{-(k+1)} \right)$.
We expand the coefficients of $(1-x)^{-(m-t)}$ asymptotically for large order $n$ and perform the integration over $t$  using Watsons Lemma \citep[Proposition 2.1, p. 53]{miller} for the fixed upper boundary $m+k$ to calculate
\begin{equation}
\int_0^{m+k} \frac{e^{-\ln(n) t}}{\Gamma(m-t)}\cdot dt \sim \sum_{i=0}^{\infty} \frac{\gamma^{(m)}_i (-1)^i \cdot i\,!}{\ln(n)^{i+1}}
\end{equation}
 From there we find 
\begin{equation}
\theta_{1,m,n} =  n^{m-1} \left(\sum_{i=0}^{M} \frac{\gamma^{(m)}_i \cdot (-1)^{i+1} \cdot i\,!}{\ln(n)^{i+1}}  + o\left(\frac{1}{\ln(n)^{M+1}} \right) \right)
\end{equation}
By a repeated use of equation \eqref{eq_funct} we find
\begin{equation}
 \theta_{j,m,n} =   n^{m-1} \left( \sum_{i=0}^{M} \frac{\gamma^{(m)}_i \cdot (-1)^{i + j}}{\ln(n)^{j+i}} \cdot \frac{(j+i-1)\,!}{(j-1)\,!}  + o\left(\frac{1}{\ln(n)^{M+j}} \right) \right)
\end{equation}
But from here it is only simple algebra to reach \eqref{eq_coeff}.
\begin{definition}\label{de_fgy}
We define the function 
\begin{multline}
\bar f_{m,k}(y,C) := \widetilde{K}(m,k,-\frac{1}{\pi y} + C) \\=m! \cdot (\pi y)^{m+1}\frac{\left(1-C\cdot\pi\cdot y \right)^{k-1}}{\left(1-(C+1)\cdot \pi \cdot y \right)^{k+m}} \cdot \bar G_{m,k}(y,C)
\end{multline}
and 
\begin{equation}
\bar G_{m,k}(y,C) := 1 + \frac{1}{m}\sum_{j=1}^{m-1} (\pi \cdot y)^j  \binom{k-1}{ j} \cdot \binom{m}{ j+1} \cdot  \left(\frac{1}{1- C\cdot \pi \cdot y} \right)^j 
\end{equation}
and $f_{m,k)}(y) := \bar f_{m,k}(y,C^{(c)})$.
and $G_{m,k}(y) := \bar G_{m,k}(y,C^{(c)})$ with
\begin{equation}\label{eq_cc}
C^{(c)} = 4 \cdot \frac{\ln(2)}{\pi}-1
\end{equation}
\end{definition}
\begin{theorem}\label{th_momr}
Let $M \in \N$ be a natural number. The $r$th moment of the planar closed simple random walk of lenght $2n$ is given together with all its corrections up to order $1/\ln(n)^M$  by 
\begin{multline}\label{eq_enmom}
E_{n}(N_{2k_1}(w)\cdot \ldots \cdot N_{2k_r}(w)) =  2 \cdot \pi \cdot n^{r} 
\cdot \sum_{F \in \tilde H_k^{(M)}} \Biggl(    cof(A-F) \cdot  \frac{I(F) \cdot M(F) }{2^{r-1}} \\
 \Biggl[\sum_{i=0}^{M} \gamma^{(r-1)}_i y^{i+1} \left(\frac{d}{dy} \right)^i \left(y^{i-1}
\frac{1}{r\,!} \sum_{\sigma \in S_r} \prod_{j=1}^r  f_{h_j,k_{\sigma(j)}}(y) \right)\Biggr|_{y = -\frac{1}{\ln(n)}}\Biggr]\Biggr) \\ +  o\left(\frac{n^r}{\ln(n)^M} \right) 
\end{multline}
where where $E_n(.)$ is the expectation value for closed walks of length $2n$ and $\tilde H_k^{(M)}$ is the finite set defined in Theorem \ref{th_nge}.
\end{theorem}
\begin{proof} We start with Lemma \ref{le_momfunc} and equation \eqref{eq_sckf} and \eqref{eq_h0ti}. We then see that each function $\mathcal{T}_{F,k}(z)$ has the form of Lemma \ref{le_yexp} with $\varpi = 16z^2$. We also use the fact that the number of closed random walks of length $2n$ in $2$ dimensions, by which we have to divide to get the expectation value, is given by 
\begin{equation}
4^{2n} \cdot \frac{1}{\pi \cdot n} \left( 1 - \frac{1}{4n} +O\left(\frac{1}{n^2}\right) \right)
\end{equation}
The differentiating of equation \eqref{eq_deftk} gives an additional factor $2n$. Putting all this together we reach equation  \eqref{eq_enmom}.
\end{proof}
\begin{remark}
From equation \eqref{eq_enmom} one can (with any good computer algebra system) compute any moment with all its logarithmic corrections from a summation over contributions from a finite set of integer valued matrices (or equivalent Eulerian multidigraphs).
\end{remark}
\begin{lemma}
Let $M \in \N$ be a natural number. Then the first moment of the expectation value of $N_{2k}(w)$ for a closed simple random walk of length $2n$ in $2$ dimensions has the asymptotic expansion 
\begin{equation}\label{eq_fmom}
E_n(N_{2k}(w)) = 2 \cdot n \cdot \sum_{i=0}^M \gamma^{(1)}_i \cdot y^{i+1}\left(\frac{d}{dy}\right)^i \left(y^{i-1} f_{1,k}(y) \right)\Biggr|_{y = -\frac{1}{\ln(n)}} + o\left(\frac{n}{\ln(n)^M}\right)
\end{equation} 
\end{lemma}
\begin{proof}
As we have shown in \citep{dhoef1} and can now extend with Theorem \ref{th_defp}
\begin{multline}
\underset{N \rightarrow \infty}{\lim}\sum_{w \in W_{2N} } N_{2k}(w) z^{length(w)}  = z \cdot \frac{\partial}{\partial z}\frac{1}{k}\left( \frac{h(0,2,z)}{1+h(0,2,z)} \right)^k \\ = \widetilde{K}(1,k,h(0,2,z)) \cdot z \cdot  \frac{\partial}{\partial z} h(0,2,z)
\end{multline}
from which equation \eqref{eq_fmom} follows using Lemma \ref{le_yexp}.
\end{proof}
\begin{corollary}
For concreteness sake:
\begin{multline}\label{eq_fmo}
E_n(N_{2k}(w)) = \frac{2 \pi^2 n}{\ln(n)^2}\Biggl(1 - \frac{1}{\ln(n)}(k\pi + 2\cdot \gamma + \pi(1+ 2\cdot C^{(c)})) + \\ O\left(\frac{1}{\ln(n)^2}\right) \Biggr)
\end{multline}
\end{corollary}
\begin{corollary}
The second moment is given by:
\begin{multline}
E_n(N_{2k_1}(w)\cdot N_{2k_2}(w)) - E_n(N_{2k_1}(w))\cdot E_n(N_{2k_2}(w))) = \\
\left(\frac{2\pi^2 n}{\ln(n)^3}\right)^2 \Biggl(\left(\frac{1}{2} H_4 \pi^3 - 4 \zeta(2) \right) - \\ \frac{3}{\ln(n)}\left(  \left( \frac{1}{2} H_4 \pi^3 - 4 \zeta(2) \right)\cdot \left( (k_1 + k_2)\frac{\pi}{2}  + 2\cdot \gamma + \pi(1+ 2\cdot C^{(c)})\right)  + 8 \zeta(3) \right) \\
+ O\left(\frac{1}{\ln(n)^2} \right) \Biggr)
\end{multline}
where $H_4 = I(F_2)$ and $C^{(c)}$ given in equation \eqref{eq_cc}. 
As we would expect from \citep[Theorem 3.5]{hamana_ann} (which deals with the non restricted random walk which is different and has a second moment differing from that of the closed random walk) the second centralized moment has a leading order $n^2/\ln(n)^6$. The cancellations which lead to this high order are not accidental. We will explore them in all orders in the next subsection and will see that they have to do with dam graphs.
\end{corollary}
We will now prepare the way to calculate the leading order for the centralized moments. 
To calculate the leading order of the $p+r$ th centralized moment, for a given $F \in \tilde H_{nd} \cap \tilde H_r$ we sum the contributions of 
\begin{equation}
E_{n}\left(\prod_{\ell=1}^{p+r} N_{2k_{\ell}}(w) - E_n( N_{2k_{\ell}}(w)) \right)
\end{equation} which come from $\widetilde{F} \in \tilde H_{p-\nu}(F)$ with $0 \leq \nu \leq p$ and do the same  for $\tilde F \in \tilde H_{p+r -\nu}(F_1)$ for $0 \leq \nu \leq p+r$. 
We have summed over all appropriate matrices in Theorem \ref{th_momr} so far. According to Definition \ref{de_dam} and  Lemma \ref{le_nontp} we can reduce the summing to one over an appropriate subset of $\tilde H_{nd}$ plus $F_1$ loosely speaking.
We define
\begin{multline}\label{eq_fullas}
E_{n}\left(\prod_{\ell=1}^{m} N_{2k_{\ell}}(w) - E_n( N_{2k_{\ell}}(w)) \right) = \\
 \widehat{\mathcal{E}}_{n,F_1,m} + \sum_{r=2}^m \left( 
\sum_{F \in \tilde H_{nd} \cap \tilde H_r \cap \tilde H^{(M)}_k} \mathcal{E}_{n,F,r,m-r}
\right)   
\end{multline}
We will completely calculate the leading behaviour of $\mathcal{E}_{n,F,r,p}$ and $\widehat{\mathcal{E}}_{n,F_1,m}$ in this section together with all logarithmic corrections of any order. The leading behaviour then can be summed to give a closed formula for the leading behaviour of the characteristic function. \\
\subsection{Generalized Leibniz rules of differentiation}
To calculate the centralized moments along the lines outlined above we first have to derive a couple of generalizations of the Leibniz rule for differentiations. 
\begin{lemma}\label{le_stonebox}
Let $f_1(x), \ldots f_p(x)$ be infinitely many times differetiable functions and $g(x)$ also be an infinitely many times differentiable function.
For a number $\mu$ let us define $A_{\mu,p}$ to be the set of the $\mu$ element subsets of $ A_p = \{1,\ldots,p\}$. Let us also define 
\begin{equation}\label{eq_psidef}
\Psi^{(p)}_{\mu} = \sum_{B \in A_{\mu,p}} \prod_{j \in B}f_j(x) \left(\frac{d}{dx} \right)^q \left(\prod_{i \in A_p \setminus B} f_i(x) \cdot g(x) \right)
\end{equation}
Then
\begin{multline}\label{eq_inclu} 
\sum_{\substack{q_1 \geq 1,\ldots,q_p \geq 1 \\ q_{g} \geq 0 \\ q_1 + \ldots + q_p + q_{g} = q}} \left( \frac{q\,!}{q_1\,! \cdot \ldots \cdot q_p\,! \cdot q_{g}\,!}\right) f_1^{(q_1)}(x) \cdot \ldots \cdot f_{p}^{(q_p)}(x) \cdot g^{(q_{g})}(x) =\\
 \sum_{\mu = 0}^{p} \Psi^{(p)}_{\mu} (-1)^{\mu} 
\end{multline}
This formula is true also for $p > q$ where the left hand side is to be interpreted to be $0$. 
\end{lemma}
\begin{proof}
By the general Leibniz rule using $f^{(q)}(x)$  as a shorthand for the $q$ th derivative of an infinitably many times differentiable function $f(x)$  we know
\begin{multline}
\left(\frac{d}{dx} \right)^q (f_1(x)\cdot \ldots \cdot f_p(x) \cdot g(x)) = \\ \sum_{\substack{q_1 \geq 0,\ldots, q_{p+1} \geq 0 \\ q_1 + \ldots + q_p +  q_{g} = q}} \left( \frac{q\,!}{q_1\,! \cdot \ldots \cdot q_{p} \cdot  q_{g}\,!}\right) f_1^{(q_1)}(x) \cdot \ldots \cdot f_{p}^{(q_p)}(x) \cdot g^{(q_{q})}(x)
\end{multline} 
We can now use the inclusion exclusion principle:
Having at least one derivative at any of the functions $f_j$ is equivalent to the occupancy problem of having at least one stone in each of $p$ boxes when distributing $q$ stones. We prove by induction in $p$: For $p =1$ equation \eqref{eq_inclu} is obviously true. For a given $p$ let's assume it is true and write $g(x) = f_{p+1}(x) \cdot \tilde g(x)$. Using \eqref{eq_inclu} we therefore get 
\begin{multline}\label{eq_inclup}
\sum_{\substack{q_1 \geq 1,\ldots,q_p \geq 1 \\ q_{g} \geq 0 \\ q_1 + \ldots + q_p + q_{g} = q}} \left( \frac{q\,!}{q_1\,! \cdot \ldots \cdot q_p\,! \cdot q_{g}\,!}\right) f_1^{(q_1)}(x) \cdot \ldots \cdot f_{p}^{(q_p)}(x) \cdot \\
\left(\sum_{\substack{q_{p+1} \geq 0 \\ q_{\tilde g} \geq 0 \\ q_{p+1} + q_{\tilde g} = q_g}} f^{(q_{p+1})}(x) \cdot \tilde g^{(q_{\tilde g})}(x) \cdot \binom{q_g}{q_{\tilde g}} \right) \\
= \sum_{\nu = 0}^{p} \tilde \Psi^{(p)}_{\mu} (-1)^{\mu}
\end{multline}
where 
\begin{equation}
\tilde \Psi^{(p)}_{\mu} =  \sum_{B \in A_{\nu,p}} \prod_{j \in B}f_j(x) \left(\frac{d}{dx} \right)^q \left(\prod_{i \in A_p \setminus B} f_i(x) \cdot  f_{p+1}(x) \cdot\tilde g(x) \right)
\end{equation}
Now we can write 
\begin{equation}\label{eq_yildiz}
\Psi^{(p+1)}_{\mu} = \tilde \Psi^{(p)}_{\mu} + f_{p+1}(x) \cdot \Psi^{(p)}_{\mu-1}
\end{equation}
the first term on the right hand side of equation \eqref{eq_yildiz} are the contributions to $\Psi^{(p+1)}_{\mu}$ in equation \eqref{eq_psidef} where $p+1 \notin B$ the second term those where $p+1 \in B$ and we have suppressed showing $\tilde g$ instead of $g$ as it is irrelevant for our discussion.  From equation \eqref{eq_inclup} we can write 
\begin{multline}\label{eq_inclu1}
\sum_{\substack{q_1 \geq 1,\ldots,q_{p+1} \geq 1 \\ q_{\tilde g} \geq 0 \\ q_1 + \ldots + q_{p+1} + q_{\tilde g} = q}} \left( \frac{q\,!}{q_1\,! \cdot \ldots \cdot q_{p+1}\,! \cdot q_{\tilde g}\,!}\right) f_1^{(q_1)}(x) \cdot \ldots \cdot f_{p+1}^{(q_{p+1})}(x) \cdot \tilde g^{(q_{\tilde g})}(x) = \\
\sum_{\mu = 0}^{p} (-1)^{\mu}\tilde \Psi_{\mu}^{(p)} - f_{p+1}(x) \cdot \sum_{\mu = 0}^{p} (-1)^{\mu} \Psi_{\mu}^{(p)}
\end{multline}
where we have just subtracted all terms with $f^{(q_{p+1})}_{p+1}(x)$ with $q_{p+1} = 0$  from the left hand side of equation \eqref{eq_inclup}. If $p = q$ this subraction yields $0$ as one sees from equation \eqref{eq_inclup} and so the interpretation of the left hand side to be $0$ for $p > q$ is valid. Using equation \eqref{eq_yildiz} in equation \eqref{eq_inclu1} we get 
\begin{multline}
\sum_{\substack{q_1 \geq 1,\ldots,q_{p+1} \geq 1 \\ q_{\tilde g} \geq 0 \\ q_1 + \ldots + q_{p+1} + q_{\tilde g} = q}} \left( \frac{q\,!}{q_1\,! \cdot \ldots \cdot q_{p+1}\,! \cdot q_{\tilde g}\,!}\right) f_1^{(q_1)}(x) \cdot \ldots \cdot f_{p+1}^{(q_{p+1})}(x) \cdot \tilde g^{(q_{\tilde g})}(x) = \\
\sum_{\mu=0}^{p} (-1)^{\mu} \Psi^{(p+1)}_{\mu} + (-1)^{p+1}\Psi^{(p+1)}_{p+1}
\end{multline}
as 
\begin{equation}
\Psi^{(p+1)}_{p+1} = -f_{p+1}(x) \cdot (-1)^p \Psi_{p}^{(p)}
\end{equation}
which proves the induction.
\end{proof}
\begin{lemma}\label{le_psifak}
With the definitions of Lemma \ref{le_stonebox} 
\begin{multline}\label{eq_psifak}
\Psi^{(p)}_{\mu} = \sum_{B \in A_{\mu,p}} \prod_{j \in B}f_j(x) \left(\frac{d}{dx} \right)^q \left(\prod_{i \in A_p \setminus B} f_i(x) \cdot g(x) \right) =\\
\frac{1}{(p-\mu)\,! \cdot \mu \,!} \sum_{\sigma \in S_p} \Psi^{(p)}_{\mu}(\sigma)
\end{multline} 
where
\begin{equation}
 \Psi^{(p)}_{\mu}(\sigma) =  \prod_{j=1}^{\mu} f_{\sigma(j)}(x) \left(\frac{d}{dx} \right)^q \left( \prod_{i =\mu+1}^{p} f_{\sigma(i)}(x) \cdot g(x)  \right)
\end{equation}
\end{lemma}
\begin{proof}
The permutations are the bijections of $A_p$. They map $A_{\mu} = \{1,\ldots,\mu\}$ onto any set $B$ with $\mu$ elements in exactly $\mu\,!$ ways and $A_p \setminus A_{\mu} $ then in $(p-\mu)\,!$ ways onto $A_p \setminus B$. 
\end{proof}
\begin{lemma}\label{le_philer}
Let $X_1,\ldots,X_p$ be $p$  random variables whose joint moments in any power are supposed to be finite. Let $E(.)$ denote the expectation value. Then the central moment 
\begin{equation}\label{eq_musterek}
E\left(\prod_{i=1}^p (X_i - E(X_i)) \right) = \sum_{\mu=0}^p (-1)^{\mu} \cdot \Phi^{(p)}_{\mu}
\end{equation}
with 
\begin{multline}\label{eq_philer}
\Phi^{(p)}_{\mu} = \sum_{B \in A_{\mu,p}} \prod_{j \in B}E(X_j) \cdot  E\left(\prod_{i \in A_p \setminus B} X_i \right) =\\
\frac{1}{(p-\mu)\,! \cdot \mu \,!} \sum_{\sigma \in S_p} \prod_{j=1}^{\mu} E(X_{\sigma(j)}) \cdot E\left( \prod_{i =\mu+1}^{p} X_{\sigma(i)}  \right)
\end{multline}
\end{lemma}
\begin{proof}
The first equation in \eqref{eq_philer} just is a regrouping of the left hand side of equation \eqref{eq_musterek} according to the indices of $E(X_j)$ being in the $\mu$ element subsets of $A_p$. Each subset comes up once of course. The second equation in \eqref{eq_philer} then follows from identical reasoning as the one for Lemma \ref{le_psifak} .
\end{proof}
\begin{lemma}\label{le_inclg}
For $p \in \N$ let $p$ infinitely many times differentiable functions $f_i(x)$ and $g_i(x)$ with $i = 1,\ldots,p$ be given and another infinitely many times differentiable function $g(x)$. Let us define
\begin{multline}
\widehat{\Psi}^{(p)}_{\mu} =
\frac{1}{(p-\mu)\,! \cdot \mu \,!} \sum_{\sigma \in S_p} \prod_{j=1}^{\mu} \left( f_{\sigma(j)}(x) + g_{\sigma(j)}(x) \right) \cdot \\ \left(\frac{d}{dx} \right)^q \left( \prod_{i =\mu+1}^{p} f_{\sigma(i)}(x) \cdot g(x)  \right)
\end{multline}
For $\sigma \in S_p$ let us also define
\begin{equation}\label{eq_defm}
 \mathcal{M}_{\nu}^{(p-\nu)}(\sigma) := q\,! \sum_{\substack{q_{\nu+1} \geq 1,\ldots,q_{p} \geq 1 \\ q_{g} \geq 0 \\ q_{\nu + 1} + \ldots + q_{p} + q_{g} = q}} \frac{g^{(q_g)}(x)}{q_g\,!} \prod_{i=\nu +1}^{p} \frac{f^{(q_{i})}_{\sigma(i)}(x)}{q_{i}\,!}
\end{equation}
and 
\begin{equation}
\widetilde{\Psi}^{(p)}_{\nu} =
\frac{1}{(p-\nu)\,! \cdot \nu \,!} \sum_{\sigma \in S_p}\mathcal{M}_{\nu}^{(p-\nu)}(\sigma) \prod_{j=1}^{\nu} \left( g_{\sigma(j)}(x) \right) 
\end{equation}
Then
\begin{equation}\label{eq_psiler}
\sum_{\mu = 0}^{p} (-1)^{\mu} \widehat{\Psi}^p_{\mu} = \sum_{\nu = 0}^{p} (-1)^{\nu} \widetilde{\Psi}^p_{\nu}
\end{equation}
\end{lemma}
\begin{proof}
For a given $\nu$ we look at the contributions of $\widehat{\Psi}^{(p)}_{\mu}$ with $\mu \geq \nu$ which contain the factor 
\begin{equation}
 \prod_{j=1}^{\nu} \left( g_{\sigma(j)}(x) \right)
\end{equation}
and no other factors $g_{\sigma(k)}(x)$. That amounts to choosing $\nu$ factors from the product over $f_{\sigma(j)}(x) + g_{\sigma(j)}(x)$ and so we get as contributions
\begin{multline}
\frac{1}{(p-\mu)\,! \cdot \mu \,!} \cdot \binom{\mu}{\nu} \sum_{\sigma \in S_p} \left(
 \prod_{j=1}^{\nu} \left( g_{\sigma(j)}(x) \right)
 \prod_{k=\nu + 1}^{\mu}  f_{\sigma(k)}(x) \right) \cdot \\ \left(\frac{d}{dx} \right)^q \left( \prod_{i =\mu+1}^{p} f_{\sigma(i)}(x) \cdot g(x)  \right)
\end{multline}
Defining $\tilde p = p - \nu$ and $\tilde \mu = \mu - \nu$ we see that the contributions are 
\begin{equation}
\frac{1}{(\tilde p- \tilde \mu)\,! \cdot \tilde \mu \,! \cdot \nu \, !} \sum_{\sigma \in S_p} \left(
 \prod_{j=1}^{\nu} \left( g_{\sigma(j)}(x) \right)
 \right) \cdot \Psi^{(\tilde p)}_{\tilde \mu}(\sigma) 
\end{equation}
with a shift of indices. Summing over $\tilde \mu$ with the appropriate signs we see from Lemma \ref{le_stonebox} together with Lemma \ref{le_psifak} and equation \eqref{eq_defm}
\begin{equation}
\sum_{\tilde \mu = 0}^{\tilde p}\frac{(-1)^{\tilde \mu}}{(\tilde p- \tilde \mu)\,! \cdot \tilde \mu \,! } 
\Psi^{(\tilde p)}_{\tilde \mu}(\sigma)  = \frac{1}{\tilde p\,!} \cdot \mathcal{M}_{\nu}^{\tilde p}(\sigma)
\end{equation}
and therefore equation \eqref{eq_psiler} is true which proves the Lemma.
\end{proof}
\subsection{The centralized moments}
We can now calculate the centralized moments with all logarithmic corrections. We start with
\begin{theorem}\label{th_lc1}
Let $F \in \tilde H_{nd} \cap \tilde H_r $. Then the contribution of $F$ together with all $\widetilde{F} \in \tilde H_{p-\nu}(F)$ with $0 \leq \nu \leq p$ to the centralized moment as given in equation \eqref{eq_fullas} together with all logarithmic corrections up to order $M$ is given by
\begin{multline}\label{eq_contras1}
\mathcal{E}_{n,F,r,p} = \frac{\pi \cdot n^{p+r}}{r\,!} \sum_{\sigma \in S_{p+r}}  \sum_{\nu = 0}^p \Biggl(  \left( 2^{\nu} \cdot  \prod_{j= r+ p - \nu +1}^{r + p} g_{1,k_{\sigma(j)}}(y)  \right) \cdot \frac{2^{p-\nu} \cdot (-1)^{\nu}}{\nu\,! (p-\nu)\,!} \\
\sum_{q\geq p - \nu}^{M} q\,! \cdot \gamma_q^{(r)} \cdot y^{q+1} 
\sum_{\substack{q_{r+1} \geq 1,\ldots,q_{r+ p - \nu} \geq 1 \\ q_{g} \geq 0 \\ q_{r+1} + \ldots + q_{r+ p - \nu} + q_{g} = q}} \frac{gs^{(q_g)}(y)}{q_g\,!} \prod_{i=r+1}^{r + p - \nu} \frac{f^{(q_{i})}_{1,k_{\sigma(i)}}(y)}{q_{i}\,!} \Biggr)  \Biggr|_{y = -\frac{1}{\ln(n)}} \\
+ o\left( \frac{n^{p+r}}{\ln(n)^M} \right)
\end{multline}
with
\begin{equation}
gs(y) := y^{q-1} g_{k,\sigma,p}(y)
\end{equation} 
where for a permutation $\sigma \in S_{p+r}$ 
\begin{multline}
g_{k,\sigma,p}(y) := \frac{I(F) \cdot M(F) \cdot cof (A-F)}{2^{r-1}} \\ \left(\sum_{i=1}^{r}
\left(\frac{ 2 \cdot (r-1)}{r} \cdot f_{h_i,k_{\sigma(i)}}(y)   + \frac{2}{\pi} \cdot f_{h_i+1,k_{\sigma(i)}}(y)\right) \cdot 
 \left( \prod_{\substack{j\geq1 \\ j \neq i}}^{r}f_{h_j,k_{\sigma(j)}}(y) \right)\right)
\end{multline}
and 
\begin{equation}
 g_{1,k}(y) :=  \sum_{i = 1}^M \gamma^{(1)}_i \cdot y^{i+1}\left(\frac{d}{dy}\right)^i \left(y^{i-1} f_{1,k}(y) \right)
\end{equation}
$\mathcal{E}_{n,F,r,p}$ has the leading behaviour
\begin{multline}\label{eq_momf}
\left(\frac{2 \pi^2 n}{\ln(n)^3}\right)^{p+r} (p+r)\,!  \cdot   \left(\frac{(-1)\cdot\pi}{\ln(n)}  \right)^{\sum_{j=1}^r (h_j -2)} \cdot \prod_{j=1}^r \frac{h_j\,!}{2\,!} \cdot \\
\Biggl( 4 \pi \cdot I(F) \cdot M(F) \cdot cof(A-F) \cdot \frac{r-1}{r\,!} \left( -\frac{\pi}{2} \right)^r
2^p (-1)^p\sum_{\nu = 0}^p \gamma^{(r)}_{p-\nu} \frac{ (- \gamma^{(1)}_1)^{\nu}}{\nu\,!} \Biggr)
\end{multline}
\end{theorem}
\begin{proof} We start at equation \eqref{eq_rd1} and  realize that the function corresponding to $g_{k,\sigma,m}(y)$ as a function in $z$ in \eqref{eq_rd1} was multiplied with 
\begin{equation}
\frac{1}{(1-16z^2)^r}
\end{equation}
because of the differentiation in $z$ according to equation \eqref{eq_deftk}.
We then can calculate again with \eqref{eq_rd1}
\begin{multline}\label{eq_erwa}
\sum_{\widetilde{F} \in \tilde H_{p-\nu}(F)} E_{n,\widetilde{F}}(N_{2k_1}(w)\cdot \ldots \cdot N_{2k_{r+p-\nu}(w)}) = \pi n \cdot (2n)^{p-\nu} \cdot \frac{ n^{r-1}}{r\,! \cdot (p-\nu)\,!} \\ \sum_{q = 0}^M \gamma_q^{(r)} y^{q+1}\left(\frac{d}{dy} \right)^q \left( y^{q-1} \sum_{\sigma \in S_{r+p-\nu}} g_{k,\sigma,p-\nu}(y) \cdot \prod_{j=r+1}^{r+ p- \nu} f_{1,k_{\sigma(j)}}(y)\right)
\\ + o\left(\frac{n^{r + p - \nu}}{\ln(n)^M}\right) 
\end{multline}
In the contribution to the centralized $p+r$ th moment this term has to be multiplied with 
\begin{equation}\label{eq_prer}
\prod_{j=r+p-\nu + 1}^{r+p} E_n(N_{2k_{j}}(w))\cdot (-1)^{\nu}
\end{equation}
and summed over all possible partitions of the set of the $p+r$ variables $k_1,\ldots, k_{r+p}$ into a set of $r+p-\nu$ and a set of $\nu$ variables. 
As $\gamma_0^{(1)} = 1$ we get 
\begin{equation}\label{eq_fmom_mod}
E_n(N_{2k}(w)) = 2 \cdot n \Biggl( f_{1,k}(y) + g_{1,k}(y) \Biggr|_{y = -\frac{1}{\ln(n)}} + o\left(\frac{1}{\ln(n)^M}\right) \Biggr)
\end{equation}
Putting this together with Lemma \ref{le_psifak},  Lemma \ref{le_philer} and  Lemma \ref{le_inclg} we get equation \eqref{eq_contras1}.
It is now easy to see that the leading term in equation \eqref{eq_contras1} comes from $q = p - \nu$ as all terms with $q < p -\nu$ vanish according to Lemma \ref{le_stonebox} and any increase beyond $q_g = 0$ or the $q_i =1$ for $i=r+1,\ldots,r+p-\nu$ will lead to higher order terms in $y$ because of the factor $y^{q+1}$ in the sum over $q$ before the differentiations.  The Lemma now follows from simple analysis and algebra. 
\end{proof}
\begin{theorem}\label{th_lc2}
The leading contribution of all $\tilde H_{p -2 -  \mu}(F_1)$ for $\mu = 0,\ldots,p-2$ to the $p$ th centralized moment with all logarithmic corrections up to order $M$  is 
\begin{multline}\label{eq_contras2}
\widehat{\mathcal{E}}_{n,F_1,p} = \pi \cdot n^{p} \sum_{\sigma \in S_{p}}  \sum_{\nu = 0}^{p} \Biggl(  \left( 2^{\nu} \cdot  \prod_{j=  p  -\nu +1}^{ p} g_{1,k_{\sigma(j)}}(y)  \right) \cdot \frac{2^{p-\nu-1}\cdot (-1)^{\nu}}{\nu\,! (p -\nu)\,!} \\
\sum_{q\geq p-\nu}^{M} q\,! \cdot \gamma_{q}^{(1)} \cdot y^{q+1} 
\sum_{\substack{q_{1} \geq 1,\ldots,q_{p  - \nu} \geq 1 \\ q_{g} \geq 0 \\ q_{1} + \ldots + q_{p -\nu} + q_{g} = q}} \frac{gt^{(q_g)}(y)}{q_g\,!} \prod_{i=1}^{ p  - \nu} \frac{f^{(q_{i})}_{1,k_{\sigma(i)}}(y)}{q_{i}\,!} \Biggr)  \Biggr|_{y = -\frac{1}{\ln(n)}} \\ + o\left( \frac{n^p}{\ln(n)^{M}}\right)
\end{multline}
where 
\begin{equation}
gt(y) := y^{q-1} \frac{2}{\pi}
\end{equation}
(for the term with $\nu=p$ the sum over $q$ at the right hand side of equation \eqref{eq_contras2} only has one nonzero contribution $2/\pi$ for $q =0$ as has been discussed in the proof of Lemma \ref{le_yexp})\\
The leading behaviour of $\widehat{\mathcal{E}}_{n,F_1,p}$ is given by
\begin{equation}
\left(\frac{2 \pi^2 n}{\ln(n)^3}\right)^{p} p\,!  \cdot 2^p (-1)^p\sum_{\nu = 0}^{p} \gamma^{(1)}_{p -\nu} \frac{ (- \gamma^{(1)}_1)^{\nu}}{\nu\,!}
\end{equation}
\end{theorem}
\begin{proof} The proof is identical to the one of Lemma \ref{th_lc1}, we just need to start from equation \eqref{eq_rd2}.
\end{proof}
\begin{theorem}\label{th_closedch}
Let $\Lambda_k$ be complex variables of which only finitely many are different from $0$. Then the characteristic function of the variable
\begin{equation}\label{eq_betavar}
\beta(w,\Lambda):=\frac{\ln(length(w))^3}{4\pi^3 length(w)}\sum_{k} \Lambda_k \cdot \left( N_{2k}(w) - E_n(N_{2k}(w) \right)
\end{equation}
(scaled to be comparable with the renormalized intersection local time of the planar Brownian motion by a factor $4 \pi^3$ in the denominator of \eqref{eq_betavar})
for a closed $2$ dimensional random walk is given in leading order of $length(w)$ by 
\begin{multline}\label{eq_gener}
E \left(e^{i t \beta(w,\Lambda)}\right) =e^{\gamma \cdot \frac{i\tau}{2\pi}}\cdot \Biggl[ \frac{1}{\Gamma(1-  \frac{i\tau}{2\pi})} + \\
4 \pi \cdot \sum_{r=2}^{\infty} \left(\frac{- i\tau}{8} \right)^r \frac{r-1}{r\,! \cdot \Gamma(r-  \frac{i\tau}{2\pi})} \cdot \sum_{F \in \tilde H_r(2,\ldots,2)} I(F) \cdot M(F) \cdot cof(A-F)  \Biggr]
\end{multline}
where 
\begin{equation}
\tau =  t \cdot \sum_{k} \Lambda_k
\end{equation}
and $\gamma$ is the Eulerian constant.
\end{theorem}
\begin{proof} With Theorem \ref{th_lc1} we see that the leading contribution overall for the centralized moments comes from $F_1$ and the matrices $F$ with $h_i = 2$ for all $i$ as a matrix with $h_i=1$ has a dam and a matrix with $h_i > 2$ is subleading as $H_r(2,\ldots,2)$ is nonempty. The generating function then follows from simple algebra from Theorem \ref{th_lc1} and from Theorem \ref{th_lc2}. We have used that a rescaling in the logarithm (by a factor of 2 in our case) does not change the leading behaviour.
\end{proof}
\section{The non restricted simple random walk}
In this section we will derive the joint distribution for the variables  $N_{2k}(w)$ in the case of the non restricted simple random walk $w$.  A non restricted random walk is one where starting and ending points of the walk underly no restrictions. As a walk is either closed or non closed, the set of non restricted random walks is the disjoint union of the set of closed and the set of non closed random walks. In the sequel (as before) we will deal in our formulas with simple random walks only and so drop the word simple often.\\ Non restricted random walks are the type of random walk mainly discussed in the vast literature \citep{chenxia}. So this section will give us a point of contact to results from many other angles of research and we will benefit strongly from the results derived by Le Gall \citep{legall} and Hamana \citep{hamana_ann} about the relationship between the joint distribution of the multiple point range and the intersection local time for this type of random walk. To make it visible in the formulas that they belong to a different group of walks than the closed simple random walk we will denote quantities with a ${(u)}$ in the upper right corner.  The concrete class of walks will be given individually in the text. This class is in some appropriate places indicated by $(n.c.)$ (for non closed) and $(n.r.)$ (for non restriced).
\subsection{Transferring the analytic results for $d \geq 2$}
\begin{lemma}
The number of non restricted simple random walks in $d$ dimensions has the generating function
\begin{equation}
h^{(n.r.)}(d,z) := \frac{2dz}{1-2dz}
\end{equation}
\end{lemma}
\begin{proof}
An non restricted random walk of length $n$ has $n$ steps, where each step can independantly go into any of the $2d$ directions. 
\end{proof}
\begin{definition}
For  $F \in \tilde H_r(h_1,...,h_r)$ and $Y_{r+1},Y_{r+2} \in \Z^d$ and $T \in Tr(K_r)$ we define
\begin{multline}
O_{ins}^{(u)} := \Biggl(\sum_{\substack{1 \leq \alpha \leq r \\ 1 \leq \beta \leq r \\ \alpha \neq \beta}} 
h(Y_{r+1} - Y_{\alpha},d,z)\cdot h(Y_{\beta} - Y_{r+2},d,z) \cdot 
\frac{\partial}{\partial H_{\alpha,\beta}}  + \\
\sum_{1 \leq \gamma \leq r} h(Y_{r+1} - Y_{\gamma},d,z)\cdot h(Y_{\gamma} - Y_{r+2},d,z) \cdot \frac{\partial}{\partial \omega_{\gamma}} \Biggr)
\end{multline}
We then use the variable transformation
\begin{equation}
X_{i,j}(\omega,H) = \frac{H_{i,j}}{\sqrt{(1 + \omega_i)\cdot(1 + \omega_j)}}
\end{equation}
to define the following quantities in analogy to the quantities in section six. 
\begin{equation}
p^{(u)}_{T,k,Y_{r+1},Y_{r+2}}(\omega,H) := O_{ins}^{(u)} \left(   p_{T,k}(\omega,X(\omega,H)\right)
\end{equation}
and using the same operations to define the quantities $D^{(u)}_{F,Y_{r+1},Y_{r+2}}(\omega,H),\\ D^{(u)}_{F,k,Y_{r+1},Y_{r+2}}(\omega,H)$, $D^{(u)}_{\ell,k,Y_{r+1},Y_{r+2}}(\omega,H)$ and finally from them\\ $C^{(u)}_{F,Y_{r+1},Y_{r+2}}(Y,z),$,
$C^{(u)}_{F,k,Y_{r+1},Y_{r+2}}(Y,z)$, $C^{(u)}_{\ell,k,Y_{r+1},Y_{r+2}}(Y,z)$ 
by setting 
\begin{equation}
H_{i,j} = h(Y_i-Y_j,d,z)
\end{equation}
and $\omega_j = h(0,d,z)$ for $j=1,\ldots,r$. \\
We also define
\begin{equation}
\gamma^{(u)}(\ell,k,F,\omega,Y_{r+1},Y_{r+2}) := O_{ins}^{(u)}(\gamma(\ell,k,F,\omega))
\end{equation}
\end{definition}
\begin{lemma}
For $z \in D \setminus P_0$
the infinite sums
\begin{equation}
S^{(u)}_{F,Y_{r+1},Y_{r+2}}(z) := \sum_{Y \in J_r\setminus\{Y_{r+1},Y_{r+2}\}} C^{(u)}_{F,Y_{r+1},Y_{r+2}}(Y,z)
\end{equation}
\begin{equation}
S^{(u)}_{F,k,Y_{r+1},Y_{r+2}}(z) := \sum_{Y \in J_r\setminus\{Y_{r+1},Y_{r+2}\}} C^{(u)}_{F,k,Y_{r+1},Y_{r+2}}(Y,z)
\end{equation}
converge absolutely. 
The same is true for 
\begin{equation}
S^{(u)}_{\ell,k,Y_{r+1},Y_{r+2}}(z) := \sum_{Y \in J_r\setminus\{Y_{r+1},Y_{r+2}\}} C^{(u)}_{\ell,k,Y_{r+1},Y_{r+2}}(Y,z)
\end{equation}
Moreover using entities defined in Lemma \ref{le_l_F} and Definition \ref{de_hlk}
\begin{multline}\label{eq_ufl1}
S^{(u)}_{\ell,k,Y_{r+1},Y_{r+2}}(z) = \sum_{F \in \tilde H_{\ell,k} }
\Biggl( \gamma(\ell,k,F,h(0,2,z)\cdot u(r)) \cdot S^{(u)}_{F,Y_{r+1},Y_{r+2}}(z) + \\
\gamma^{(u)}(\ell,k,F,h(0,2,z)\cdot u(r),Y_{r+1},Y_{r+2}) \cdot S_{F}(z) \Biggr)
\end{multline}
and 
\begin{multline}\label{eq_ufl2}
S^{(u)}_{F,k,Y_{r+1},Y_{r+2}}(z) = \sum_{\ell = 0}^{\infty}
\Biggl( \gamma(\ell,k,F,h(0,2,z)\cdot u(r)) \cdot S^{(u)}_{F,Y_{r+1},Y_{r+2}}(z) + \\
\gamma^{(u)}(\ell,k,F,h(0,2,z)\cdot u(r),Y_{r+1},Y_{r+2}) \cdot S_{F}(z) \Biggr)
\end{multline}
\end{lemma}
\begin{proof}
The convergence of the sums follows from Lemma \ref{le_grconv} as any Euler graph belonging to $F$ after removing an edge (which the differentiation in $H_{\alpha,\beta}$ stands for) still contains a spanning tree. For the differentiation in $\omega_{\gamma}$ the Euler graph belonging to $F$ is unchanged and so again we can use Lemma \ref{le_grconv}. Equations \eqref{eq_ufl1} and \eqref{eq_ufl2}  then follow immediately from the absolute convergence. 
\end{proof}
\begin{theorem}\label{th_ufeynman}
For a simple non closed random walk starting in a point $Y_{r+1} \in \Z^d$ and ending in $Y_{r+2} \in \Z^d$ where $Y_{r+1} \neq Y_{r+2}$ we define the set 
\begin{equation}
W_{N,Y_{r+1},Y_{r+2}} := \{w: \textrm{walk from  $Y_{r+1}$ to $Y_{r+2}$, $length(w) \leq N$}\}
\end{equation}
and 
\begin{equation}
W^{(n.c.)}_{N} := \{w: \textrm{not closed; $length(w) \leq N$}\}
\end{equation}
and 
\begin{equation}
W^{(n.r.)}_{N} := \{w: \textrm{ $length(w) \leq N$}\}
\end{equation}
Then the sums
\begin{equation}
S^{(u)}_{k,Y_{r+1},Y_{r+2}}(z) := \underset{N \rightarrow \infty}{\lim} \sum_{w \in W_{N,Y_{r+1},Y_{r+2}}} P(N_{2k_1}(w),\ldots,N_{2k_r}(w)) \cdot z^{length(w)}
\end{equation}
and 
\begin{equation}
S^{(n.c.)}_{k}(z) := \underset{N \rightarrow \infty}{\lim} \sum_{w \in W^{(n.c.)}_{N}} P(N_{2k_1}(w),\ldots,N_{2k_r}(w)) \cdot z^{length(w)}
\end{equation}
are well defined holomorphic functions on the open neighborhood \\ $U \supset \overline{U_{\frac{1}{2d}}(0)} \cap D$ of Theorem \ref{th_defp} and there is an $\epsilon(s(k)) > 0$ such that for 
$z \in U_{\frac{1}{2d}}(0) \cap U_{\epsilon(s(k))}(\pm \frac{1}{2d})$
the following equations are true. 
\begin{equation}\label{eq_ufeynman1}
S^{(u)}_{k,Y_{r+1},Y_{r+2}}(z) = \sum_{\ell \geq 0} S^{(u)}_{\ell,k,Y_{r+1},Y_{r+2}}(z)
\end{equation}
and 
\begin{equation}\label{eq_ufeynman2}
S^{(n.c.)}_{k}(z) := \sum_{Y_{r+1} \neq Y_{r+2}} S^{(u)}_{k,Y_{r+1},Y_{r+2}}(z) 
\end{equation}
where the summations over $Y_{r+1},Y_{r+2}$ and the summation over $\ell$ converge absolutely and therefore are interchangeable and the summation over $Y_{r+1}$ and $Y_{r+2}$ can be done on the matrix terms inside the sums if the substitutions \eqref{eq_ufl1} are done in equation \eqref{eq_ufeynman2}.
The function
\begin{equation}
S^{(n.r.)}_{k}(z) := S^{(n.c.)}_{k}(z)  + S_k(z)
\end{equation}
then is 
\begin{equation}\label{eq_nonres}
S^{(n.r.)}_{k}(z) = \underset{N \rightarrow \infty}{\lim}\sum_{w \in W^{(n.r.)}_N} P(N_{2k_1}(w),\ldots,N_{2k_r}(w))\cdot z^{length(w)}
\end{equation}
\end{theorem}
\begin{proof}
We start at Theorem \ref{t_odet} and observe 
\begin{multline}
\sum_{w \in W_{N,Y_{r+1},Y_{r+2}}} P(N_{2k_1}(w),\ldots,N_{2k_r}(w)) z^{length(w)} =\\
\frac{1}{r\,!} \sum_{\substack{\sigma \in D(k_1,\ldots,k_r)\\ Y \in J_r \setminus\{Y_{r+1},Y_{r+2}\}}}
 \left(\sum_{w \in W_{N,Y_{r+1},Y_{r+2}}(Y,k_{\sigma(1)},\hdots,k_{\sigma(r)})} z^{length(w)}\right)
\end{multline}
as a walk which is not closed cannot have a point of even multiplicity in its starting and ending point. The proof after this is equivalent to that of the closed random walk in section six, as the insertion Operator $O^{(u)}_{ins}$ can be used to create a one to one relationship between any equation there and the corresponding one for the non-closed case: All related summations over $Y \in J_r$ and $Y \in J_r  \setminus\{Y_{r+1},Y_{r+2}\}$ converge absolutely, because the relevant quantities summed belong to Euler graphs for the closed simple random walk and therefore contain a spanning tree even if an edge was removed (Euler graphs contain Euler circuits which go to every vertex and circuits do not contain an isthmus) and the convergence then follows from Lemma \ref{le_grconv}. The final operator $z \cdot \partial / \partial z$ in e.g. equation \eqref{eq_hats} for the case of the closed random walk to choose any point as starting and ending point is in a sense replaced by $O^{(u)}_{ins}$.\\ 
The summations over $Y_{r+1}, Y_{r+2}$ converge absolutely for any matrix $F$  because the operator $O^{(u)}_{ins}$ contains the product $h(Y_{r+1} - Y_{\alpha},d,z)\cdot h(Y_{\beta} - Y_{r+2},d,z)$ for every term  (with $\alpha = \beta$ for the $\omega$ differentiations). The semi-Eulerian graphs it creates (in the case of $\gamma^{(u)}$ the pair of pendant edges is attached to the vertex of the Eulerian graph $F$ to which the factor belongs which $\gamma^{(u)}$ multiplies and then is summed over the vertices) contain a spanning tree which extends to the vertices related to $Y_{r+1}, Y_{r+2}$ and therefore the sums converge absolutely according to Lemma \ref{le_grconv}. Therefore the summation can be extended and interchanged with the summation over $\ell$ so equations \eqref{eq_ufeynman1} and \eqref{eq_ufeynman2} are true. Equation \eqref{eq_nonres} directly follows from the fact that a walk is either closed or nonclosed.  
\end{proof}
We have now discussed the moments of the distribution of the non closed simple random walk for $d \geq 2$ and seen that they again are completely determined by the behaviour for $z = \pm 1 / 2d$.\\
\subsection{Transferring the results of the asymptotic expansion in $d =2$}
We now turn to the case $d=2$ which is the main theme of this paper. 
\begin{definition}
For a matrix $F \in \tilde H_r$ an integer $m \geq 0$ and a permutation $\sigma \in S_{r+m}$ we define 
\begin{multline}\label{eq_usckf1}
\mathcal{T}^{(u,1)}_{F,k,\sigma,m}(z) := cof(A-F) \cdot \frac{\mathcal{I}(F) \cdot M(F) }{4^{r-1}\cdot (1-4z)^{r+1}} \cdot \\ \left[ 
\prod_{j=1}^r  \widetilde{K}(h_j,k_{\sigma (j)},\bar h(0,2,z))
\right]
\end{multline} 
with $\mathcal{I}(F)$ defined in equation \eqref{eq_guzI} and 
\begin{multline}\label{eq_usckf2}
\mathcal{T}^{(u,2)}_{F,k,\sigma,m}(z) := cof(A-F) \cdot \frac{I(F) \cdot M(F) }{4^{r-1}\cdot (1-4z)^{r+1}} \cdot \\ \left[   \sum_{i=1}^r \widetilde{K}(h_i+1,k_{\sigma(i)},\bar h(0,2,z)) \cdot 
\prod_{j \geq 1, j \neq i}^r  \widetilde{K}(h_j,k_{\sigma (j)},\bar h(0,2,z))
\right]
\end{multline} 
with 
\begin{equation}\label{eq_uh0ti}
\bar h(0,2,z) = - \frac{1}{\pi} \cdot \ln(1-4z) + C^{(u)}
\end{equation} 
with
\begin{equation}\label{eq_cu}
C^{(u)} = \frac{3}{\pi} \cdot \ln(2) -1
\end{equation}
\\
We also define 
\begin{equation}
\mathcal{T}^{(u)}_{F,k,\sigma,m}(z) := \mathcal{T}^{(u,1)}_{F,k,\sigma,m}(z) + \mathcal{T}^{(u,2)}_{F,k,\sigma,m}(z)
\end{equation}
and 
\begin{equation}
\mathcal{T}^{(u)}_{F,k}(z) = \frac{1}{r\,!} \sum_{\sigma \in S_r} \mathcal{T}^{(u)}_{F,k,\sigma,0}(z)
\end{equation}
and the moment functions
\begin{equation}
T^{(n.c.)}_k(z) := \underset{N \rightarrow \infty}{\lim}\sum_{w \in W^{(n.c.)}_N} \left(\prod_{i=1}^r N_{2k_i}(w) \right) z^{length(w)}
\end{equation}
\begin{equation}
T^{(n.r.)}_k(z) := \underset{N \rightarrow \infty}{\lim}\sum_{w \in W^{(n.r.)}_N} \left(\prod_{i=1}^r N_{2k_i}(w) \right) z^{length(w)}
\end{equation}
\end{definition}
\begin{theorem}\label{th_unge}
In two dimensions
\begin{equation}\label{eq_2dnc}
T^{(n.c.)}_k(z) - \sum_{F \in \tilde H^{(M)}_k} \mathcal{T}^{(u)}_{F,k}(z) = o\left( \frac{1}{s^{r + 1} \cdot \ln(s)^M} \right)
\end{equation}
and 
\begin{equation}\label{eq_2dnr}
T^{(n.r.)}_k(z) - T^{(n.c.)}_k(z) = o\left( \frac{1}{s^{r}} \right)
\end{equation}
where $\tilde H_k^{(M)}$ was defined in Theorem \ref{th_nge} 
\end{theorem}
\begin{proof}
We start with Theorem \ref{th_ufeynman} and consider
\begin{equation}
\sum_{Y_{r+1} \neq Y_{r+2}} S^{(u)}_{F,k,Y_{r+1},Y_{r+2}}(z) 
\end{equation}
In the same way as shown in Lemma \ref{le_integr} (the argument does not change if an edge is removed) for a discussion of the asymptotic behaviour the summation over $Y \in J_r\setminus\{Y_{r+1},Y_{r+2}\}$ in $S^{(u)}_{F,k,Y_{r+1},Y_{r+2}}(z)$ in $2$ dimensions can be replaced by an integral, the restrictions $Y_i \neq Y_{r+1}$ and $Y_i \neq Y_{r+2}$ lead to correction terms which are scaled down by a factor $s$ and are therefore not relevant to our discussion. For the sum over $Y_{r+1}$ and $Y_{r+2}$  again the restrictions for $Y_{r+1} \neq Y_{r+2}$ lead to correction terms scaled down by a factor $s$ and lead to two factors
\begin{equation}\label{eq_factu}
\sum_{y \in \Z^d} h(y-x,d,z) = \frac{4z}{1-4z} 
\end{equation}
for the two summations.\\
As the factor in \eqref{eq_factu} does not have a singularity for $z \rightarrow -1/4$ we see that terms related to the singularity there are scaled down by two orders in $s$ and are therefore not relevant for our discussion either. 
The argument for Theorem \ref{th_nge} then applies as also 
\begin{equation}
\sum_{Y_{r+1} \neq Y_{r+2}} \gamma^{(u)}(\ell,k,F,h(0,2,z)\cdot u(r),Y_{r+1},Y_{r+2}) \sim 
o\left( \frac{1}{s^2 \cdot \ln(s)^{r}}\right)
\end{equation}
as we can deduce from Lemma \ref{le_l_F}.
We have defined $\mathcal{T}^{(u)}$ analogously to $\mathcal{T}$ where $\mathcal{T}^{(u,1)}$ obviously belongs to the differentiation terms of $O^{(u)}_{ins}$ related to $H$ and  $\mathcal{T}^{(u,2)}$ to those related to $\omega$ where we have used equation \eqref{eq_komega} and so the leading behaviour is encapsulated in those functions by the same arguments as in the case of the closed walk. So equation \eqref{eq_2dnc} holds. Equation \eqref{eq_2dnr} follows from Theorem \ref{th_nge} and equations \eqref{eq_hats} and \eqref{eq_deftk} showing that the contribution of the closed walk is scaled down in $s$. 
\end{proof}
\begin{lemma}\label{le_urd1}
Let $m > 0$ and $F \in \tilde H_{nd}\cap \tilde H_r$ and $k = (k_1,\ldots,k_{r+m})$ a vector of strictly positive integers and $\sigma \in S_{r+m}$ and let us define
\begin{multline}
gf^{(u)}(\sigma,z) = 
cof(A-F)  \cdot \\
 \left( \frac{I(F) \cdot M(F) }{4^{r-1}\cdot (1-4z)^{r-1}} \cdot \left(
\prod_{j=1}^r  \widetilde{K}(h_j,k_{\sigma (j)},\bar h(0,2,z))\right) \right)
\end{multline}
and 
\begin{equation}
gfd^{(u)}(\sigma,z) = 
\mathcal{T}^{(u)}_{F,k,\sigma,m}(z)
\end{equation}
and  
\begin{equation}\label{eq_uvfdef}
vf^{(u)}(\sigma,j,z) := \widetilde{K}\left(1,k_{\sigma(j)},\bar h(0,2,z)\right) 
\end{equation}
and 
\begin{equation}\label{eq_uvfdef2}
vfd^{(u)}(\sigma,j,z) := \frac{1}{(1-4z)^2}\widetilde{K}\left(2,k_{\sigma(j)},\bar h(0,2,z)\right) 
\end{equation}
Then 
\begin{multline}\label{eq_urd1}
\sum_{\tilde F \in \tilde H_{m}(F)} \mathcal{T}_{\tilde F,k}(z) = \binom{r+m}{m} \left( z\cdot \frac{\partial}{\partial z}\right)^{m-1}
\\
\Biggl[ \frac{1}{(r+ m)\,!} \sum_{\sigma \in S_{r+m}} \Biggl(
\left( \prod_{j=r+1}^{r+m} vf^{(u)}(\sigma,j,z) \right) \cdot  \left( z\cdot \frac{\partial}{\partial z} \right) gfd^{(u)}(\sigma,z)  + \\
\sum_{i=r+1}^{r+m} vfd^{(u)}(\sigma,i,z) \left(\prod_{j\geq r+1 \wedge j \neq i}^{r+m} vf^{(u)}(\sigma,j,z) \right) \cdot   \left( z\cdot \frac{\partial}{\partial z} \right) gf^{(u)}(\sigma,z)
\Biggr)
\Biggr] \\
+ o\left(\frac{1}{s^{r+m }} \right)
\end{multline}
\end{lemma}
\begin{proof}
The proof is analogous to the proof of Lemma \ref{le_rd1}. According to Theorem \ref{th_ufeynman} we just have to insert the two pendant edges at any possible vertex  or replacing an edge of the graph $G(\tilde F)$.
In Lemma \ref{le_hmchar} we had completely characterized $\tilde H_m(F)$ with path dressings on the edges of $G(F)$ and cycle dressings of each vertex of $G(F)$ and of the inserted vertices. 
We realize that $gfd^{(u)}$ just is the factor for all insertions of the two pendant edges into $G(F)$ and $vfd^{(u)}$ the corresponding factor for the insertion of the two pendant edges into the vertex with the factor $\tilde K$ of a dam of a graph. Dressing with dams and inserting pendant edges are interchangeable according to the Markov feature of the random walk.  
The proof of the Lemma now is analogous to the proof of Lemma \ref{le_rd1}. 
\end{proof}
\begin{lemma}\label{le_urd2}
Let $m \geq 0$ and $F_1$ be  as defined in equation \eqref{eq_ff} and $k = (k_1,\ldots,k_{m+2})$ a vector of strictly positive integers.
Then 
\begin{multline}\label{eq_urd2}
\sum_{\widetilde{F} \in \tilde H_{m}(F_1)} \mathcal{T}^{(u)}_{\tilde F,k}(z) =
\\ \left( z\cdot \frac{\partial}{\partial z}\right)^{m+1}
\Biggl[ \frac{1}{(m+2)\,!}\sum_{\sigma \in S_{m+2}} 
\left( \prod_{j=1}^{m+2} \widetilde{K}\left(1,k_{\sigma(j)},\bar h(0,2,z)\right) \right) \\
 \cdot \left( z\cdot \frac{\partial}{\partial z}\right)\frac{1}{(1-4z)}  \Biggr] 
+ o\left(\frac{1}{s^{m+2}}\right)
\end{multline}
\end{lemma}    
\begin{proof}
The proof is analogous to the one for Lemma \ref{le_rd2} and Lemma \ref{le_urd1}.
\end{proof}
\begin{definition}\label{de_ufgy}
Using the general definitions \ref{de_fgy} we define the function 
\begin{equation}
f^{(u)}_{m,k}(y) := \bar f_{m,k}(y,C^{(u)})
\end{equation}
with $C^{(u)}$ given in equation \eqref{eq_cu}.
\end{definition}
\begin{theorem}\label{th_ucalc}
Let $M \in \N$ be a natural number. The $r$th moment of the planar non restricted simple random walk of lenght $n$ is given together with all its corrections up to order $1/\ln(n)^M$  by 
\begin{multline}\label{eq_uenmom}
E_{n}^{(n.r.)}(N_{2k_1}(w)\cdot \ldots \cdot N_{2k_r}(w)) =  n^{r} 
\cdot \sum_{F \in \tilde H_k^{(M)}} \Biggl(   cof(A-F) \cdot\frac{   M(F)  }{4^{r-1}} 
\\ \Biggl[ \sum_{i=0}^{M} \gamma^{(r+1)}_i y^{i+1} \left(\frac{d}{dy} \right)^i 
\Biggl(y^{i-1} \frac{1}{r\,!} \sum_{\sigma \in S_r} \Biggl(  \mathcal{I}(F) \prod_{j =1}^r  f^{(u)}_{h_j,k_{\sigma(j)}}(y)  \\+ I(F)
\sum_{q = 1}^r f^{(u)}_{h_q+1,k_{\sigma(q)}}(y) \prod_{j\geq 1 \wedge j \neq q}^r  f^{(u)}_{h_j,k_{\sigma(j)}}(y) 
\Biggr) \Biggr)\Biggr|_{y = -\frac{1}{\ln(n)}}   
\Biggr] \Biggr)\\ +  o\left(\frac{n^r}{\ln(n)^M} \right) 
\end{multline}
where $E_n^{(n.r.)}(.)$ is the expectation value for non restricted walks of length $n$ and $\tilde H_k^{(M)}$ is the finite set defined in Theorem \ref{th_nge}.
\end{theorem}
\begin{proof} We start with Theorem \ref{th_unge} and equations \eqref{eq_usckf1}, \eqref{eq_usckf2} and \eqref{eq_uh0ti}. We then see that any of the functions $\mathcal{T}^{(u)}_{F,k}(z)$ has the form of Lemma \ref{le_yexp} with $\varpi = 4z$. We also use the fact that the number of simple non restricted walks of length $n$ in two dimensions by which we have to divide to get the expectation value is given by $4^n$.
Putting all this together we reach equation  \eqref{eq_uenmom}.
\end{proof}
\begin{lemma}\label{le_fmom}
Let $M \in \N$ be a natural number. Then the first moment of the expectation value of $N_{2k}(w)$ for a simple non restricted random walk of length $n$ in $2$ dimensions has the asymptotic expansion 
\begin{multline}\label{eq_ufmom}
E^{(n.r.)}_n(N_{2k}(w)) = n \cdot \sum_{i=0}^M \gamma^{(2)}_i \cdot y^{i+1}\left(\frac{d}{dy}\right)^i \left(y^{i-1} f_{1,k}^{(u)}(y) \right)\Biggr|_{y = -\frac{1}{\ln(n)}} + \\ o\left(\frac{n}{\ln(n)^M}\right)
\end{multline} 
\end{lemma}
According to Theorem \ref{th_unge} we can use the non closed walks and sum without restrictions. 
So we find from Theorem \ref{t_odet} for $r = 1$:
\begin{equation}
\underset{N \rightarrow \infty}{\lim}\sum_{w \in W^{(n.r.)}_{N}} N_{2k}(w) z^{length(w)} = \frac{1}{(1-4z)^2} \tilde K(1,k,\bar h(0,2,z))+ o\left( \frac{1}{s}\right)
\end{equation}
But then equation \eqref{eq_ufmom} follows from Lemma \ref{le_yexp}.
For concreteness sake
\begin{multline}
E_n^{(n.r.)}(N_{2k}(w)) = \frac{n \pi^2}{\ln(n)^2}\cdot \left(1-\frac{1}{\ln(n)}(k \pi + 2 \cdot \gamma + \pi(1 + 2\cdot C^{(u)})) \right) \\ + O\left(\frac{n}{\ln(n)^2}\right)
\end{multline}
\begin{corollary}
We also give the second centralized moment of the non restricted random walk as calculated from equation \eqref{eq_uenmom}:
\begin{multline}\label{eq_2mom}
E_n^{(n.r.)}(N_{2k_1}(w)\cdot N_{2k_2}(w)) - E_n^{(n.r.)}(N_{2k_1}(w))\cdot E_n^{(n.r.)}(N_{2k_2}(w)) = \\
\left(\frac{\pi^2 n}{\ln(n)^3}\right)^2 \Biggl(8\cdot \left(\frac{H_3 \cdot \pi^2}{8} + \frac{1}{2} - \frac{\pi^2}{12} \right) -  \\ \frac{24}{\ln(n)}\Biggl( \left(\frac{H_3 \cdot \pi^2}{8} + \frac{1}{2} - \frac{\pi^2}{12} \right)\cdot \left( (k_1 + k_2)\frac{\pi}{2}  + 2\cdot \gamma + \pi(1+ 2\cdot C^{(u)}) -4 \right)  +\\
 \zeta(3) - \zeta(2) + \frac{H_3 \cdot \pi^2}{8} + \frac{H_4 \cdot \pi^3}{24} \Biggr)
+ O\left(\frac{1}{\ln(n)^2} \right) \Biggr)
\end{multline}
where $H_4= I(F_2)$ and $H_3 = 1/4 \cdot \mathcal{I}(F_2)$:
\begin{equation}
H_3 = \left(\frac{2}{\pi}\right)^3 \int_{\R^2} d^2y \cdot  K_0(\lvert y \rvert)^3 = \frac{16}{\pi^2} \int_0^{\infty} dx \cdot  K_0(x)^3  \cdot x
\end{equation} 
\end{corollary}
With \citep[eq. 3.477 4.]{rg} we realize that the Laplace transform of the Gauss distribution is proportional to $K_0$ and therefore using \citep[eq. 4.3]{jain} we get  
\begin{equation}
 2\cdot \int_0^{\infty} dx \cdot  K_0(x)^3\cdot x  = -\int_0^1 \frac{\ln(x)\cdot dx}{1-x+x^2} = \frac{3}{2}\left[ \sum_{p=0}^{\infty} \frac{1}{(1 + 3p)^2} - \frac{1}{(2+3p)^2} \right] 
\end{equation}
where the second equation is from \citep[eq. 9-133, p. 455]{itzykson}.
The leading term at the right hand side of \eqref{eq_2mom} can be compared to the classical result of Jain and Pruitt \citep[Theorem 4.2]{jain} for the second centralized moment of the range. For this we use  Hamanas relation  \citep[Theorem 3.5]{hamana_ann} of $\ln(n)^3/n$ times the distribution of the multiple point range of a random walk converging asymptotically in the distribution sense towards $c_1$ times the distribution of the intersection local time of the planar Brownian motion. $c_1$ is given by 
\begin{equation}\label{eq_c1}
c_1 = -16\cdot \pi^3 \det(\Xi^2)
\end{equation}
And we use Le Galls result \citep[Theorem 6.1]{legall} of $\ln(n)^2/n$ times the distribution of the range of the random walk  converging asymptotically in the distribution sense towards $c_2$ times the distribution of the intersection local time of the planar Brownian motion where 
\begin{equation}
c_2 = -4 \cdot \pi^2 \det(\Xi)
\end{equation}
In the case of the simple random walk $\det(\Xi) = 1/2$ and so the result of Jain and Pruitt for the rescaled centralized second moment of the range 
\begin{equation}
8 \pi^2 \cdot \det(\Xi^2)  \cdot \left(\frac{H_3 \cdot \pi^2}{8} + \frac{1}{2} - \frac{\pi^2}{12} \right)
\end{equation}
has to be multiplied by a factor 
\begin{equation}
\frac{c_1^2}{c_2^2} = 4 \pi^2
\end{equation}
to yield our result. 
We have used the fact here that the variables $Q^{(p)}_n$  of Hamana compare to our definitions with 
\begin{equation}
Q_n^{(p)} = N_{2p}(w) + N_{2p-1}(w)
\end{equation}
but that the contributions of the multiple points with odd order, as there can only be two of them for a given random walk, are irrelevant for our result. 
\begin{theorem}\label{th_ulc1}
Let $F \in \tilde H_{nd} \cap \tilde H_r $. Then we can write $\mathcal{E}^{(u)}_{n,F,r,m-r}$, the sum of the contributions  of $F$ together with those of $\widetilde{F} \in \tilde H_{m-r-\nu}(F)$ with $0 \leq \nu \leq m-r$ to the $m$ th centralized moment of the multiple point range with all logarithmic corrections of any order  in the form
\begin{multline}
E^{(n.r.)}_{n}\left(\prod_{\ell=1}^{m} \left( N_{2k_{\ell}}(w) - E^{(n.r.)}_n( N_{2k_{\ell}}(w)) \right) \right) = \\
 \widehat{\mathcal{E}}^{(u)}_{n,F_1,m} + \sum_{r=2}^m \left( 
\sum_{F \in \tilde H_{nd} \cap \tilde H_r \cap \tilde H^{(M)}_k} \mathcal{E}^{(u)}_{n,F,r,m-r}
\right)   
\end{multline}
where $E_n^{(n.r.)}(.)$ is the expectation value for the non restricted simple random walk of length $n$. It 
is given by the sum of three terms:
\begin{equation}
 \mathcal{E}^{(u)}_{n,F,r,m-r} =  \sum_{i=1}^3 \mathcal{E}^{(u,i)}_{n,F,r,m-r}
\end{equation}
with 
\begin{multline}\label{eq_ucontras1}
\mathcal{E}^{(u,1)}_{n,F,r,p} := \frac{ n^{p+r}}{r\,!} \sum_{\sigma \in S_{p+r}}  \sum_{\nu = 0}^p \Biggl(  \left(   \prod_{j= r+ p - \nu +1}^{r + p} g^{(u)}_{1,k_{\sigma(j)}}(y)  \right) \cdot \frac{ (-1)^{\nu}}{\nu\,! (p-\nu)\,!} \cdot \\
\sum_{q\geq p - \nu}^{M} q\,! \cdot \gamma_q^{(r+1)} \cdot y^{q+1} \cdot \\
\left(\sum_{\substack{q_{r+1} \geq 1,\ldots,q_{r+ p - \nu} \geq 1 \\ q_{g} \geq 0 \\ q_{r+1} + \ldots + q_{r+ p - \nu} + q_{g} = q}} \frac{gsu_1^{(q_g)}(y)}{q_g\,!} \prod_{i=r+1}^{r + p - \nu} \frac{f^{(u)(q_{i})}_{1,k_{\sigma(i)}}(y)}{q_{i}\,!} \right) \Biggr)  \Biggr|_{y = -\frac{1}{\ln(n)}} \\
+ o\left( \frac{n^{p+r}}{\ln(n)^M} \right)
\end{multline}
and for $\alpha = 2,3$ and $p > 0$ 
\begin{multline}\label{eq_ucontras2}
\mathcal{E}^{(u,\alpha)}_{n,F,r,p} := \frac{ n^{p+r}}{r\,!} \sum_{\sigma \in S_{p+r}}  \sum_{\nu = 0}^{p-1} \Biggl(  \left(   \prod_{j= r+ p - \nu +1}^{r + p} g^{(u)}_{1,k_{\sigma(j)}}(y)  \right) \cdot \frac{ (-1)^{\nu}}{\nu\,! (p-\nu)\,!} \\
\sum_{q\geq  p-\nu - 1}^{M} q\,! \cdot \gamma_q^{(r+2)} \cdot y^{q+1} \cdot
\\ \left(
\sum_{c=r+1}^{r + p - \nu}
\sum_{\substack{q_{\beta} \geq 1 \\ \forall \beta:  r+1 \leq \beta \leq r + p - \nu; \beta \neq c \\ q_{g} \geq 0 \\ q_{g} + \sum_{r+1 \leq \beta \leq r + p - \nu; \beta \neq c} q_{\beta}= q} }
\frac{gsu_{\alpha}^{(q_g)}(c,y)}{q_g\,!} \prod_{\substack{i\geq r+1 \\ i \neq c}}^{r + p - \nu} \frac{f^{(u)(q_{i})}_{1,k_{\sigma(i)}}(y)}{q_{i}\,!} \right) \Biggr)  \Biggr|_{y = -\frac{1}{\ln(n)}} \\
+ o\left( \frac{n^{p+r}}{\ln(n)^M} \right)
\end{multline}
(for $p = 0$ and $\alpha = 2,3$ we define $\mathcal{E}^{(u,\alpha)}_{n,F,r,p} := 0$)\\
We have used the notations:
\begin{equation}
gsu_{1}(y) := y^{q-1} gd_{1,k,\sigma,p}(y)
\end{equation} 
and for $\alpha = 2,3$
\begin{equation}
gsu_{\alpha}(c,y) := y^{q-1} gd_{\alpha,k,\sigma,p}(c,y)
\end{equation} 
where for a permutation $\sigma \in S_{p+r}$ 
\begin{multline}
gd_{1,k,\sigma,p}(y) := \frac{ M(F) \cdot cof (A-F)}{4^{r-1}} \\ \left(\sum_{i=1}^{r} \left( \left(\frac{ \mathcal{I}(F) }{r} \cdot f^{(u)}_{h_i,k_{\sigma(i)}}(y)   + I(F) \cdot f^{(u)}_{h_i+1,k_{\sigma(i)}}(y)\right) \prod_{\substack{j\geq1 \\ j \neq i}}^{r}f^{(u)}_{h_j,k_{\sigma(j)}}(y)  \right)\right)
\end{multline}
\begin{equation}
gd_{2,k,\sigma,p}(c,y) := (-1)  \cdot \frac{1}{\pi} \cdot f^{(u)}_{2,k_{\sigma(c)}}(y)  \cdot gd_{1,k,\sigma,p}(y) 
\end{equation}
\begin{multline}
gd_{3,k,\sigma,p}(c,y) :=   f^{(u)}_{2,k_{\sigma(c)}}(y)\cdot \frac{I(F) \cdot M(F) \cdot cof (A-F)}{4^{r-1}} \\ \left(\sum_{i=1}^{r}
\left(\frac{ (r-1)}{r} \cdot f^{(u)}_{h_i,k_{\sigma(i)}}(y)   + \frac{1}{\pi} \cdot f^{(u)}_{h_i+1,k_{\sigma(i)}}(y)\right) \cdot 
 \left( \prod_{\substack{j\geq1 \\ j \neq i}}^{r}f^{(u)}_{h_j,k_{\sigma(j)}}(y) \right)\right)
\end{multline}
and 
\begin{equation}
 g^{(u)}_{1,k}(y) :=  \sum_{i = 1}^M \gamma^{(2)}_i \cdot y^{i+1}\left(\frac{d}{dy}\right)^i \left(y^{i-1} f^{(u)}_{1,k}(y) \right)
\end{equation}
The leading behaviour is given by
$\mathcal{N}_1(F) + \mathcal{N}_2(F) + \mathcal{N}_3(F)$.
If we define
\begin{multline}
\mathcal{Q}(F) := \left(\frac{\pi^2 n}{\ln(n)^3}\right)^{p+r} (p+r)\,!  \cdot   \left(\frac{(-1)\cdot \pi}{\ln(n)}  \right)^{\sum_{j=1}^r (h_j -2)} \cdot \prod_{j=1}^r \frac{h_j\,!}{2\,!} \cdot \\
\Biggl( 4 \cdot  M(F) \cdot cof(A-F) \cdot \frac{1}{r\,!} \left( -\frac{\pi}{2} \right)^r
2^p (-1)^p \Biggr)
\end{multline}
Then 
\begin{equation}\label{eq_umomf1}
\mathcal{N}_1(F) = \mathcal{Q}(F) \cdot \mathcal{I}(F) \Biggl( \sum_{\nu = 0}^p \gamma^{(r+1)}_{p-\nu} \frac{ (- \gamma^{(2)}_1)^{\nu}}{\nu\,!} \Biggr)
\end{equation}
\begin{equation}\label{eq_umomf2}
\mathcal{N}_2(F) = (-1) \cdot \mathcal{Q}(F) \cdot \mathcal{I}(F) \cdot \Biggl( \sum_{\nu = 0}^{p-1} \gamma^{(r+2)}_{p-\nu-1} \frac{ (- \gamma^{(2)}_1)^{\nu}}{\nu\,!} \Biggr)
\end{equation}
\begin{equation}\label{eq_umomf3}
\mathcal{N}_3(F) = \pi \cdot (r-1) \cdot  \mathcal{Q}(F) \cdot I(F)    \cdot \Biggl( \sum_{\nu = 0}^{p-1} \gamma^{(r+2)}_{p-\nu-1} \frac{ (- \gamma^{(2)}_1)^{\nu}}{\nu\,!} \Biggr)
\end{equation}
where $\mathcal{N}_2(F)$ and $\mathcal{N}_3(F)$ only exist for $p > 0$ and are to be set to $0$ for $p=0$.
\end{theorem}
\begin{proof} The proof is analogous to the one of Theorem \ref{th_lc1}, taking into account that $\gamma^{(2)}_0 = 1$. We start with Lemma \ref{le_urd1} and reformulate in the first term on the right hand side of \eqref{eq_urd1}:
\begin{multline}\label{eq_urdd}
 \left( z\cdot \frac{\partial}{\partial z}\right)^{m-1}\Biggl[
\left( \prod_{j=r+1}^{r+m} vf^{(u)}(\sigma,j,z) \right) \cdot  \left( z\cdot \frac{\partial}{\partial z} \right) gfd^{(u)}(\sigma,z) 
\Biggr] = \\
\left( z\cdot \frac{\partial}{\partial z}\right)^{m}\Biggl[
\left( \prod_{j=r+1}^{r+m} vf^{(u)}(\sigma,j,z) \right) \cdot gfd^{(u)}(\sigma,z) 
\Biggr] - \\
\left( z\cdot \frac{\partial}{\partial z}\right)^{m-1}\Biggl[ \sum_{i=r+1}^{r+m}
\left( \prod_{j \geq r+1 \wedge j \neq i}^{r+m} vf^{(u)}(\sigma,j,z) \right) \cdot  gfd^{(u)}(\sigma,z) \cdot \\
\left( z\cdot \frac{\partial}{\partial z} \right) vf^{(u)}(\sigma,i,z)
\Biggr]
\end{multline}
We remember from Lemma \ref{le_urd1} that this equation is only valid for $m > 0$. Now contributions belonging to the first term on the right hand side of equation \eqref{eq_urdd} together with the contribution of $F$ itself (which takes the place of $m=0$) result in $\mathcal{E}^{(u,1)}_{n,F,r,p}$ when the cancellation scheme of Theorem \ref{th_lc1} is applied to them. The contributions belonging to the second term on the right hand side of equation \eqref{eq_urdd} when the cancellation scheme is applied for those $p-1$ factors $vf$ which do not belong to $j=i$ gives the the contribution $\mathcal{E}^{(u,2)}_{n,F,r,p}$. The second term on the right hand side of equation \eqref{eq_urd1} with the cancellation scheme of Theorem \ref{th_lc1}  applied to the $vf$ not having $j=i$ leads to the term $\mathcal{E}^{(u,3)}_{n,F,r,p}$.
\end{proof}
\begin{theorem}\label{th_ulc2}
The leading contribution with all logarithmic corrections of $\tilde H_{p-\mu-2}(F_1)$ for $\mu = 0,\ldots,p-2$ to the $p$ th centralized moment of the non restricted simple random walk of length $n$ is
\begin{multline}\label{eq_ucontras3}
\widehat{\mathcal{E}}^{(u)}_{n,F_1,p} =  n^{p} \sum_{\sigma \in S_{p}}  \sum_{\nu = 0}^{p} \Biggl(  \left(  \prod_{j=  p  -\nu +1}^{ p} g^{(u)}_{1,k_{\sigma(j)}}(y)  \right) \cdot \frac{(-1)^{\nu}}{\nu\,! (p -\nu)\,!} \\
\sum_{q\geq p - \nu}^{M} q\,! \cdot \gamma_{q}^{(2)} \cdot y^{q+1} 
\sum_{\substack{q_{1} \geq 1,\ldots,q_{p  - \nu} \geq 1 \\ q_{g} \geq 0 \\ q_{1} + \ldots + q_{p -\nu} + q_{g} = q}} \frac{gtu^{(q_g)}(y)}{q_g\,!} 
\prod_{i=1}^{ p  - \nu} \frac{f^{(u)(q_{i})}_{1,k_{\sigma(i)}}(y)}{q_{i}\,!} \Biggr)  \Biggr|_{y = -\frac{1}{\ln(n)}} \\ + o\left( \frac{n^p}{\ln(n)^{M}}\right)
\end{multline}
where 
\begin{equation}
gtu(y) := y^{q-1}
\end{equation}
(for the term with $\nu=p$ the sum over $q$ at the right hand side of equation \eqref{eq_ucontras3} only has one nonzero contribution $1$ for $q =0$ as has been discussed in the proof of Lemma \ref{le_yexp}).\\
Its leading behaviour is given by
\begin{equation}
\left(\frac{\pi^2 n}{\ln(n)^3}\right)^{p} p\,!  \cdot 2^p (-1)^p\sum_{\nu = 0}^p \gamma^{(2)}_{p-\nu} \frac{ (- \gamma^{(2)}_1)^{\nu}}{\nu\,!}
\end{equation}
\end{theorem}
\begin{proof}
The proof follows directly from Lemma \ref{le_urd2} with the cancellation scheme of Theorem \ref{th_lc1}.
\end{proof}
\begin{proof}[ Proof of Theorem 1.1] For the proof we use  Hamanas relation  \citep[Theorem 3.5]{hamana_ann} of $\ln(n)^3/n$ times the distribution of the multiple point range of a random walk converging asymptotically in the distribution sense towards $c_1$ times the distribution of the intersection local time of the planar Brownian motion. $c_1$ was given in equation \eqref{eq_c1} therefore the overall factor with which we have to scale is $-4\cdot \pi^3$.
 The proof then follows directly from Theorems \ref{th_ulc1} and \ref{th_ulc2}. Again we notice that the contributions of points with odd multiplicity $N_{2k-1}(w)$ in the result of Hamana are irrelevant for our discussion, as their number is smaller or equal to $2$ for a given walk.
\end{proof}
\emph{Discussion and Conjecture:}
The fact that the characteristic function in equation \eqref{eq_gener} and equation \eqref{eq_charbr} depends completely on the matrices $F$ belonging to Eulerian multidigraphs strongly supports the idea that indeed the geometrical approach to the problem of the multiple point range is a natural one. The characteristic functions give us some geometrical insight.  The global geometry of the planar random walk is dominated by dams which give the leading terms for the moments. The random walk predominantly returns to a point via loops, different points it returns to do not interact at this level. But if we subtract this more or less average behaviour which amounts to subtracting the dams, as we have seen, the next layer of geometry are Eulerian multidigraphs with two ingoing and two outgoing edges at each vertex and they all contribute to the centralized moments. There is a contribution though always from the dams even after going to the centralized moments which is expressed in the factors 
of the form
\begin{equation}
\frac{e^{(1-\gamma)\cdot \frac{it}{2\pi}}}{\Gamma(r +2 +   \frac{it}{2\pi})}
\end{equation}
whose Taylor coefficients however vanish rapidly. \\
Looking at the integrals $I(F)$ and $\mathcal{I}(F)$ one realizes that they are the Feynman integrals of the two dimensional $\Phi^4$ theory as only matrices $ F \in \tilde H_r(2,\ldots,2)$ contribute to the leading order. One only has to remember that  the two dimensional Euclician propagator of a free boson with mass $m$ (which is $1$ in our case)  in real space is \citep[eq. 2.104]{confft}
\begin{equation}
\frac{1}{2\pi}K_0(m\left| x - y \right|)
\end{equation}
(i.e. differs from our convention by a factor of $1/4$ which explains why we have the huge factor $8^r$ in the denominator of our formulas).
The Brownian motion and the $\Phi^4$ theory therefore are very closely related to each other naturally in 2 dimensions. In the logarithmic corrections to the leading order for the distribution of the multiple point range terms from $\Phi^{2m}$ with $m > 2$ come up but they are scaled out for the Brownian motion as a universality class.  
\\
Let us now take a look at the summation over $r$. The characteristic function is a formal summation in equation \eqref{eq_gener} and equation \eqref{eq_charbr}. Now $M(F) \leq 1$ and $cof(A-F) \leq 2^{r-1}$ because any Eulerian path at each vertex has at most two alternatives to go in a graph $G(F)$ with $F \in \tilde H_r(2,\ldots,2)$. The integrals $I(F) \leq D^r$  and $\mathcal{I}(F) \leq D^r$ with some constant $D$ as four edges are on a given vertex for  $F \in \tilde H_r(2,\ldots,2)$ and the integrals are dominated by points were $x_i - x_j =0$.  On the other hand the number 
$\#(H_r(2,\ldots,2))$ should grow slower than $ \sim (2r)\,! \cdot 2^r$ 
up to powers of $r$ as the corresponding undirected graphs are a subset of those arising from a scalar $\Phi^4$ theory where the growth is known \citep[eq. 3.3]{cvitan} and there are less than $2^r$ ways to make a directed graph out of an undirected. Taking all this together this would mean that the sum over $r$ in equation \eqref{eq_gener} and \eqref{eq_charbr} converges for sufficiently small $t$ but the radius of convergence is not infinite. This is in line with what Le Gall \citep{legall2} has proven about
\begin{equation}
E(e^{\lambda \cdot \beta_1})
\end{equation}
becoming singular for large real positive $\lambda$.
From the close connection between the $\Phi^4$ theory and the Brownian motion it then makes sense to conjecture that asymptotically for $r \rightarrow \infty$ we have
\begin{multline}\label{eq_asymp1}
 \sum_{F \in \tilde H_r(2,\ldots,2)} I(F) \cdot M(F) \cdot cof(A-F) = \\ (2r)\,!\cdot A_0 \cdot \eta_0^r \cdot r^ {\alpha_0} \left(1 +\sum_{K = 1}^M B^{(0)}_K\cdot \left(\frac{1}{r} \right)^K + o\left(\frac{1}{r^{M}} \right) \right)
\end{multline}
and
\begin{multline}\label{eq_asymp2}
 \sum_{F \in \tilde H_r(2,\ldots,2)} \mathcal{I}(F) \cdot M(F) \cdot cof(A-F) = \\ (2r)\,! \cdot A_1  \cdot \eta_1^r \cdot r^ {\alpha_1} \left(1 +\sum_{K = 1}^M B^{(1)}_K\cdot \left(\frac{1}{r} \right)^K + o\left(\frac{1}{r^{M}}\right) \right)
\end{multline}
The constants $A_0,\eta_0,\alpha_0,A_1,\eta_1, \alpha_1$ then will determine the behaviour of large moments, and therefore the distribution in its tails, the constants $B_K$ will give corrections and lead to good numerical calculations. As the quantity in \eqref{eq_gener} is something like the partition function and the quantity in \eqref{eq_charbr} then is something like the two point correlation function we would conjecture that $ \eta: = \eta_0 = \eta_1$; there should be a common ``critical temperature''. It should not be too hard to calculate the constants $A_0,\eta_0,\alpha_0,,A_1,\eta_1, \alpha_1$ similar to the calculations of  the high orders of the $\Phi^4$ theory done by Lipatov \citep{lipatov} and recently extended much further to include even the analog of the  $B_K$ by Lobaskin and Suslov \citep{lobask}. From this approach we conjecture that $\alpha_0, \alpha_1$ are both rational numbers. Now if our conjecture is true, then there has been some work done related to the constant $\eta$ by Le Gall \citep{legall2}, who showed that it is finite, and Bass and Chen \citep[Theorem 1.1]{basschen} who calculated it. Indeed their relating the constant $\gamma_{\beta}$  (which is up to a constant an inverse of $\eta$) to the best solution of a Gagliardo–-Nirenberg inequality (identical to the instanton solution of a mass $1$ $\Phi^4$ theory) is exactly what one would expect from the above reasoning along the lines of Lipatov. 
\section{Conclusion}
In this article we have built on the relationship between moment functions $A(w)$ of  closed simple random walks $w$ written as formal power series of the form 
\begin{equation}
P_A(z) := \sum_{w \in W_{2N}} A(w) \cdot z^{length(w)}
\end{equation}
and a formal series over contributions from the geometric archetypes as expressed in adjaceny matrices $F$ of Eulerian multidigraphs of the form 
\begin{equation}\label{eq_fterms}
P_A(z) = \sum_{F} g_{A,F}(z)
\end{equation}
with known functions $g_{A,F}(z)$ developed in \citep{dhoef1}. We extended those relationships to non closed walks. 
We have by an analytical study of the functions $g_{A,F}(z)$ for $d \geq 2$ been able to transform the sum over $F$ in \eqref{eq_fterms} into an absolutely converging series using the first hit determinant $\Delta_r(Y,z)$ (as defined in \eqref{eq_fhitd1} and \eqref{eq_fhitd2}) as the key. From there we have been able to define $P_A(z)$ as an analytical function, find its singularities at the radius of convergence and relate its behaviour at the singularities to the behaviour of $g_{A,F}(z)$ at those points. For $d=2$ we have calculated the leading asymptotic behaviour of $A(w)$ for large $length(w)$ including all its logarithmic corrections and have therefore been able to find the characteristic function in closed form for the closed and the non restricted two dimensional simple random walk and therefore of the renormalized intersection local time of the planar Brownian motion. This shows the relevance of our geometric approach as expressed in equation \eqref{eq_fterms}. It is this approach together with the analytic solidification in this article which is important far beyond concrete results like the characteristic functions in $d=2$. As our approach is also heavily inspired by and has  in turn lead to Feynman graphs in the characteristic function in $d=2$, it should also show a way to solidify summations over such graphs in other situations.\\
What are some typical topics of research from here? A short subjective list:
\begin{enumerate}
\item calculate the distribution of $N_{2k+1}(w)$ and the joint distribution with the $N_{2k}(w)$ using the full class of semi Eulerian graphs
\item calculate the numerical value of the first (if possible ten) moments of the renormalized intersection local time of the planar Brownian motion to high precision.
\item calculate the high order behaviour of those moments, prove the conjecture or something like it and possibly calculate the constants in the equations \eqref{eq_asymp1} and \eqref{eq_asymp2}
\item calculate the behaviour of the multiple point range and weigh\-ted ran\-ges of the form of Theorem \ref{th_closedch} of intersecting random walks and relate the result in $d=2$ to the Onsager solution and the Kosterlitz Thouless transition 
\item understand if the Taylor coefficients of $1 - \Delta_r(Y,z)$ can be interpreted as numbers of geometrical entities for $r> 2$
\item take the geometrical concept of this article to other Markov processes
\item calculate the behaviour of the moments for $d \geq 3$ using concepts like that of the dams for $d=2$ to get to the relevant geometrical information   
\item understand what the relationship between the concept of this article and the Schramm Loewner evolution \citep{schramm} is
\item understand where relations like \eqref{eq_fterms} can be established in other fields of statistical geometry, e.g. with CW complexes of higher dimension to e.g. calculate the behaviour of random surfaces 
\item understand what the equivalent to the first hit determinant in other areas where summations over graphs  are important is. This would help to make summations over Feynman graphs of other theories solid mathematically.
\end{enumerate} 
\bibliographystyle{imsart-number}
\bibliography{hoef2}

\begin{thebibliography}{31}
% BibTex style file: imsart-number.bst, 2013-01-28
% Default style options (sort=1,type=number).
% Used options (sort=1,type=number).

\bibitem{basschen}
\begin{barticle}[author]
\bauthor{\bsnm{Bass},~\bfnm{Richard~F.}\binits{R.~F.}} \AND
  \bauthor{\bsnm{Chen},~\bfnm{Xia}\binits{X.}}
(\byear{2004}).
\btitle{Self-intersection local time: critical exponent, large deviations, and
  laws of the iterated logarithm}.
\bjournal{Ann. Probab.}
\bvolume{32}
\bpages{3221--3247}.
\bdoi{10.1214/009117904000000504}
\bmrnumber{2094444 (2005i:60149)}
\end{barticle}
\endbibitem

\bibitem{bryant}
\begin{bbook}[author]
\bauthor{\bsnm{Bryant},~\bfnm{V.}\binits{V.}}
(\byear{1992}).
\btitle{Aspects of Combinatorics. A wide-ranging introduction}.
\bpublisher{Cambridge University Press}, \baddress{Cambridge}.
\end{bbook}
\endbibitem

\bibitem{bry}
\begin{barticle}[author]
\bauthor{\bsnm{Brydges},~\bfnm{David}\binits{D.}},
  \bauthor{\bsnm{Fr{\"o}hlich},~\bfnm{J{\"u}rg}\binits{J.}} \AND
  \bauthor{\bsnm{Spencer},~\bfnm{Thomas}\binits{T.}}
(\byear{1982}).
\btitle{The random walk representation of classical spin systems and
  correlation inequalities}.
\bjournal{Comm. Math. Phys.}
\bvolume{83}
\bpages{123--150}.
\bmrnumber{648362 (83i:82032)}
\end{barticle}
\endbibitem

\bibitem{chenxia}
\begin{barticle}[author]
\bauthor{\bsnm{Chen},~\bfnm{Xia}\binits{X.}}
(\byear{2008}).
\btitle{Intersection local times: large deviations and laws of iterated
  logarithm}.
\bjournal{Adv. Lect. Math. (ALM)}
\bvolume{2}
\bpages{195--253}.
\end{barticle}
\endbibitem

\bibitem{cvitan}
\begin{barticle}[author]
\bauthor{\bsnm{Cvitanovi\'{c}},~\bfnm{Predrag}\binits{P.}},
  \bauthor{\bsnm{Lautrup},~\bfnm{B.~E.}\binits{B.~E.}} \AND
  \bauthor{\bsnm{Pearson},~\bfnm{Robert~B.}\binits{R.~B.}}
(\byear{1978}).
\btitle{Number and weights of Feynman Diagrams}.
\bjournal{Phys.Rev.}
\bvolume{D18}
\bpages{1939 -- 1956}.
\bdoi{10.1103/PhysRevD.18.1939}
\end{barticle}
\endbibitem

\bibitem{danraf}
\begin{bbook}[author]
\bauthor{\bsnm{Danos},~\bfnm{M.}\binits{M.}} \AND
  \bauthor{\bsnm{Rafelski},~\bfnm{J.}\binits{J.}}
(\byear{1984}).
\btitle{Pocketbook of Mathematical Functions}.
\bpublisher{Harry Deutsch}, \baddress{Thun and Frankfurt am Main}.
\end{bbook}
\endbibitem

\bibitem{confft}
\begin{bbook}[author]
\bauthor{\bsnm{Di~Francesco},~\bfnm{P.}\binits{P.}},
  \bauthor{\bsnm{Mathieu},~\bfnm{P.}\binits{P.}} \AND
  \bauthor{\bsnm{S\'en\'echal},~\bfnm{D.}\binits{D.}}
(\byear{1996}).
\btitle{Conformal Field Theory}.
\bpublisher{Springer Verlag}, \baddress{New York}.
\end{bbook}
\endbibitem

\bibitem{Tek}
\begin{barticle}[author]
\bauthor{\bsnm{Dvoretzky},~\bfnm{A.}\binits{A.}},
  \bauthor{\bsnm{Erd{\"o}s},~\bfnm{P.}\binits{P.}} \AND
  \bauthor{\bsnm{Kakutani},~\bfnm{S.}\binits{S.}}
(\byear{1950}).
\btitle{Double points of paths of {B}rownian motion in {$n$}-space}.
\bjournal{Acta Sci. Math. Szeged}
\bvolume{12}
\bpages{75--81}.
\bmrnumber{0034972 (11,671e)}
\end{barticle}
\endbibitem

\bibitem{Fla}
\begin{barticle}[author]
\bauthor{\bsnm{Flatto},~\bfnm{Leopold}\binits{L.}}
(\byear{1976}).
\btitle{The multiple range of two-dimensional recurrent walk}.
\bjournal{Ann. Probability}
\bvolume{4}
\bpages{229--248}.
\bmrnumber{0431388 (55 \#\#4388)}
\end{barticle}
\endbibitem

\bibitem{gasull}
\begin{barticle}[author]
\bauthor{\bsnm{Gasull},~\bfnm{A.}\binits{A.}},
  \bauthor{\bsnm{Li},~\bfnm{W.}\binits{W.}},
  \bauthor{\bsnm{Llibre},~\bfnm{J.}\binits{J.}} \AND
  \bauthor{\bsnm{Zhang},~\bfnm{Z.}\binits{Z.}}
(\byear{2002}).
\btitle{Chebychev Property of Complete Elliptic Integrals and its Application
  to Abelian Integrals}.
\bjournal{Pacific J. Math.}
\bvolume{202}
\bpages{341--360}.
\end{barticle}
\endbibitem

\bibitem{rg}
\begin{bbook}[author]
\bauthor{\bsnm{Gradstein},~\bfnm{I.~S.}\binits{I.~S.}} \AND
  \bauthor{\bsnm{Ryshik},~\bfnm{I.~M.}\binits{I.~M.}}
(\byear{1981}).
\btitle{Summen- Produkt- und Integraltafeln}.
\bpublisher{Harry Deutsch}, \baddress{Thun and Frankfurt am Main}.
\end{bbook}
\endbibitem

\bibitem{Ham1}
\begin{barticle}[author]
\bauthor{\bsnm{Hamana},~\bfnm{Yuji}\binits{Y.}}
(\byear{1992}).
\btitle{On the central limit theorem for the multiple point range of random
  walk}.
\bjournal{J. Fac. Sci. Univ. Tokyo Sect. IA Math.}
\bvolume{39}
\bpages{339--363}.
\bmrnumber{1179772 (93h:60112)}
\end{barticle}
\endbibitem

\bibitem{Ham2}
\begin{barticle}[author]
\bauthor{\bsnm{Hamana},~\bfnm{Yuji}\binits{Y.}}
(\byear{1995}).
\btitle{On the multiple point range of three-dimensional random walks}.
\bjournal{Kobe J. Math.}
\bvolume{12}
\bpages{95--122}.
\bmrnumber{1391188 (97e:60118)}
\end{barticle}
\endbibitem

\bibitem{hamana_ann}
\begin{barticle}[author]
\bauthor{\bsnm{Hamana},~\bfnm{Yuji}\binits{Y.}}
(\byear{1997}).
\btitle{The fluctuation result for the multiple point range of two-dimensional
  recurrent random walks}.
\bjournal{Ann. Probab.}
\bvolume{25}
\bpages{598--639}.
\bdoi{10.1214/aop/1024404413}
\bmrnumber{1434120 (98f:60136)}
\end{barticle}
\endbibitem

\bibitem{Ham3}
\begin{barticle}[author]
\bauthor{\bsnm{Hamana},~\bfnm{Yuji}\binits{Y.}}
(\byear{1998}).
\btitle{A remark on the multiple point range of two-dimensional random walks}.
\bjournal{Kyushu J. Math.}
\bvolume{52}
\bpages{23--80}.
\bdoi{10.2206/kyushujm.52.23}
\bmrnumber{1608981 (99b:60110)}
\end{barticle}
\endbibitem

\bibitem{tsp}
\begin{bbook}[author]
\bauthor{\bsnm{Hansen},~\bfnm{R.}\binits{R.}}
(\byear{1975}).
\btitle{A Table of Series and Products}.
\bpublisher{Prentice Hall}, \baddress{Englewood Cliffs, NJ}.
\end{bbook}
\endbibitem

\bibitem{dhoef1}
\begin{barticle}[author]
\bauthor{\bsnm{H{\"o}f},~\bfnm{D.}\binits{D.}}
(\byear{2006}).
\btitle{Distribution of the {$k$}-multiple point range in the closed simple
  random walk. {I}}.
\bjournal{Markov Process. Related Fields}
\bvolume{12}
\bpages{537--560}.
\bmrnumber{2246263 (2007k:60130)}
\end{barticle}
\endbibitem

\bibitem{itzykson}
\begin{bbook}[author]
\bauthor{\bsnm{Itzykson},~\bfnm{C.}\binits{C.}} \AND
  \bauthor{\bsnm{Zuber},~\bfnm{J.~B.}\binits{J.~B.}}
(\byear{1988}).
\btitle{Quantum Field Theory}.
\bpublisher{Mc Graw Hill}, \baddress{New York}.
\end{bbook}
\endbibitem

\bibitem{jain}
\begin{binproceedings}[author]
\bauthor{\bsnm{Jain},~\bfnm{Pruitt W.~E.}\binits{P.~W.~E.} \bsuffix{N.~C.}}
(\byear{1973}).
\btitle{The range of random walks}.
In \bbooktitle{Proc. Sixth Berkeley Symp. Math. Stat. Probab.}
\bpages{31--50}.
\end{binproceedings}
\endbibitem

\bibitem{legall}
\begin{barticle}[author]
\bauthor{\bsnm{Le~Gall},~\bfnm{J.~F.}\binits{J.~F.}}
(\byear{1986}).
\btitle{Propri\'et\'es d'intersection des marches al\'eatoires. {I}.
  {C}onvergence vers le temps local d'intersection}.
\bjournal{Comm. Math. Phys.}
\bvolume{104}
\bpages{471--507}.
\bmrnumber{840748 (88d:60182)}
\end{barticle}
\endbibitem

\bibitem{legall2}
\begin{bincollection}[author]
\bauthor{\bsnm{Le~Gall},~\bfnm{Jean-Fran{\c{c}}ois}\binits{J.-F.}}
(\byear{1994}).
\btitle{Exponential moments for the renormalized self-intersection local time
  of planar {B}rownian motion}.
In \bbooktitle{S\'eminaire de {P}robabilit\'es, {XXVIII}}.
\bseries{Lecture Notes in Math.}
\bvolume{1583}
\bpages{172--180}.
\bpublisher{Springer}, \baddress{Berlin}.
\bdoi{10.1007/BFb0073845}
\bmrnumber{1329112 (96c:60092)}
\end{bincollection}
\endbibitem

\bibitem{lipatov}
\begin{barticle}[author]
\bauthor{\bsnm{Lipatov},~\bfnm{L.~N.}\binits{L.~N.}}
(\byear{1977}).
\btitle{Divergence of the Perturbation Theory Series and the Quasiclassical
  Theory}.
\bjournal{Sov.Phys.JETP}
\bvolume{45}
\bpages{216--223}.
\end{barticle}
\endbibitem

\bibitem{lobask}
\begin{barticle}[author]
\bauthor{\bsnm{Lobaskin},~\bfnm{D.~A.}\binits{D.~A.}} \AND
  \bauthor{\bsnm{Suslov},~\bfnm{I.~M.}\binits{I.~M.}}
(\byear{2004}).
\btitle{High-order corrections to the Lipatov asymptotics in the $\Phi^4$
  theory}.
\bjournal{J.Exp.Theor.Phys.}
\bvolume{99}
\bpages{234-253}.
\bdoi{10.1134/1.1800180}
\end{barticle}
\endbibitem

\bibitem{miller}
\begin{bbook}[author]
\bauthor{\bsnm{Miller},~\bfnm{P.~D.}\binits{P.~D.}}
(\byear{2006}).
\btitle{Applied Asymptotic Analysis}.
\bseries{Graduate Studies in Mathematics}
\bvolume{75}.
\bpublisher{AMS}, \baddress{Rhode Island}.
\end{bbook}
\endbibitem

\bibitem{Pitt}
\begin{barticle}[author]
\bauthor{\bsnm{Pitt},~\bfnm{Joel~H.}\binits{J.~H.}}
(\byear{1974}).
\btitle{Multiple points of transient random walks}.
\bjournal{Proc. Amer. Math. Soc.}
\bvolume{43}
\bpages{195--199}.
\bmrnumber{0386021 (52 \#\#6880)}
\end{barticle}
\endbibitem

\bibitem{shmo}
\begin{barticle}[author]
\bauthor{\bsnm{Sherman},~\bfnm{Jack}\binits{J.}} \AND
  \bauthor{\bsnm{Morrison},~\bfnm{Winifred~J.}\binits{W.~J.}}
(\byear{1950}).
\btitle{Adjustment of an inverse matrix corresponding to a change in one
  element of a given matrix}.
\bjournal{Ann. Math. Statistics}
\bvolume{21}
\bpages{124--127}.
\bmrnumber{0035118 (11,693d)}
\end{barticle}
\endbibitem

\bibitem{sidi}
\begin{barticle}[author]
\bauthor{\bsnm{Sidi},~\bfnm{Avram}\binits{A.}}
(\byear{2012}).
\btitle{Euler-{M}aclaurin expansions for integrals with arbitrary
  algebraic-logarithmic endpoint singularities}.
\bjournal{Constr. Approx.}
\bvolume{36}
\bpages{331--352}.
\bdoi{10.1007/s00365-011-9140-0}
\bmrnumber{2996435}
\end{barticle}
\endbibitem

\bibitem{symanzik}
\begin{bincollection}[author]
\bauthor{\bsnm{Symanzik},~\bfnm{Kurt}\binits{K.}}
(\byear{1969}).
\btitle{Euclidean quantum field theory}.
In \bbooktitle{Local quantum field theory}
(\beditor{\bfnm{Res}\binits{R.}~\bsnm{Jost}}, ed.)
\bpublisher{Academic Press}, \baddress{New York}.
\end{bincollection}
\endbibitem

\bibitem{tutte}
\begin{bbook}[author]
\bauthor{\bsnm{Tutte},~\bfnm{W.}\binits{W.}}
(\byear{1984}).
\btitle{Graph Theory}.
\bpublisher{Addison Wesley}, \baddress{New York}.
\end{bbook}
\endbibitem

\bibitem{aar}
\begin{barticle}[author]
\bauthor{\bparticle{van} \bsnm{Aardenne-Ehrenfest},~\bfnm{T.}\binits{T.}} \AND
  \bauthor{\bparticle{de} \bsnm{Bruijin},~\bfnm{N.~G.}\binits{N.~G.}}
(\byear{1951}).
\btitle{Circuits and trees in oriented linear graphs}.
\bjournal{Simon Stevin Wis. Natuurkd. Tijdschr.}
\bvolume{28}
\bpages{203--217}.
\end{barticle}
\endbibitem

\bibitem{schramm}
\begin{bbook}[author]
\bauthor{\bsnm{Werner},~\bfnm{W.}\binits{W.}}
\btitle{Random planar curves and Schramm--Loewner evolutions}.
\bseries{Lecture Notes in Math.}
\bpublisher{Springer Verlag}.
\end{bbook}
\endbibitem

\end{thebibliography}
\end{document}